\documentclass[11pt, reqno]{amsart}
\usepackage{amsmath, amssymb, amsfonts, amsthm, verbatim, color}
\usepackage{mathtools}
\mathtoolsset{showonlyrefs}
\usepackage{bbm}
\usepackage{float}
\usepackage{amsmath, amssymb, amsfonts, amsthm, verbatim, dsfont, graphicx}
\usepackage{epstopdf}
\usepackage[hidelinks]{hyperref}
\usepackage[T1]{fontenc}
\usepackage{amssymb,amsmath, amsthm, amsfonts}
\usepackage{graphicx}

\allowdisplaybreaks

\mathtoolsset{showonlyrefs}
\numberwithin{figure}{section}

\usepackage{enumerate}
\usepackage{tensor}

\usepackage{graphicx}
\usepackage{xcolor} 
\usepackage{tensor}
\usepackage{slashed}
\usepackage{cite}
\definecolor{green}{rgb}{0,0.8,0} 

\newcommand{\ud}{\mathrm{d}}

\newcommand{\eps}{\epsilon}

\newcommand{\calC}{\mathcal C}

\newcommand{\calG}{\mathcal G}

\newcommand{\calS}{\mathcal S}

\renewcommand{\H}{\mathcal{H}}

\usepackage{wrapfig}
\usepackage{tikz}
\usetikzlibrary{arrows,calc,decorations.pathreplacing}
\definecolor{light-gray1}{gray}{0.90}
\definecolor{light-gray2}{gray}{0.80}

\definecolor{deepgreen}{cmyk}{1,0,1,0.5}

\newcommand{\B}{\mathcal{B}}

\newcommand{\HH}{\mathcal{H}}

\newcommand{\NN}{\mathcal{N}}

\newcommand{\N}{\mathbb{N}}
\newcommand{\R}{\mathbb{R}}
\newcommand{\Sp}{\mathbb{S}}
\newcommand{\Z}{\mathbb{Z}}

\newcommand{\al}{\alpha}
\newcommand{\be}{\beta}

\newcommand{\de}{\delta}
\newcommand{\e}{\eps}

\newcommand{\om}{\omega}
\newcommand{\la}{\lambda}
\newcommand{\te}{\theta}

\newcommand{\s}{\sigma}

\newcommand{\ka}{\kappa}

\newcommand{\De}{\Delta}
\newcommand{\Om}{\Omega}

\newcommand{\p}{\partial}
\newcommand{\na}{\nabla}

\newcommand{\supp}{\operatorname{supp}}

\newcommand{\ext}{\operatorname{ext}}

\makeatletter

\newcommand{\Rmnum}[1]{\expandafter\@slowromancap\romannumeral #1@}
\makeatother

\newcommand{\I}{\infty}

\newcommand{\ti}{\widetilde}

\newcommand{\ang}[1]{\left\langle{#1}\right\rangle}
\newcommand{\abs}[1]{\left\lvert{#1}\right\rvert}

\newcommand{\ant}[1]{\begin{align*}\begin{split} #1 \end{split}\end{align*}}
\newcommand{\EQ}[1]{\begin{equation}\begin{split} #1 \end{split}\end{equation}}

\newcommand{\Del}[1]{}

\numberwithin{equation}{section}

\newtheorem{thm}{Theorem}[section]
\newtheorem{cor}[thm]{Corollary}
\newtheorem{lem}[thm]{Lemma}
\newtheorem{prop}[thm]{Proposition}

\newtheorem{claim}[thm]{Claim}
\theoremstyle{remark}
\newtheorem{rem}[thm]{Remark}
\newtheorem{defn}[thm]{Definition}

\newcommand{\mand}{{\ \ \text{and} \ \  }}

\newcommand{\mif}{{\ \ \text{if} \ \ }}

\newcommand{\mas}{{\ \ \text{as} \ \ }}

\newcommand{\nsp}{|\nabla|^{s_p}}
\newcommand{\nspt}{|\nabla|^{s_p-1}}

\def\glei{\mathrm{eq}}

\newcommand{\qtq}[1]{\quad\text{#1}\quad}

\newcommand{\bR}{{\mathbb R}}

\newcommand{\cN}{{\mathcal{N}}}

\newcommand{\cH}{{\mathcal{H}}}

\newcommand{\inte}{{\mathrm{int}}}

\begin{document}

\title[Subcritical wave equations]{Scattering for defocusing energy subcritical nonlinear wave equations}
\author{B. Dodson}
\author{A. Lawrie}
\author{D. Mendelson}
\author{J. Murphy}

\begin{abstract}
We consider the Cauchy problem for the defocusing power type nonlinear wave equation in $(1+3)$-dimensions for energy subcritical powers $p$ in the super-conformal range $3 < p< 5$.   We prove that any solution is global-in-time and scatters to free waves in both time directions as long as its critical Sobolev norm stays bounded on the maximal interval of existence. 
\end{abstract}

\thanks{B. Dodson gratefully acknowledges support from  NSF DMS-1500424 and NSF DMS-1764358. A. Lawrie gratefully acknowledges support from NSF grant DMS-1700127. D. Mendelson gratefully acknowledges support from NSF grant DMS-1800697. J. Murphy gratefully acknowledges support from DMS-1400706.  The authors thank the MSRI program ``New Challenges in PDE: Deterministic Dynamics and Randomness in High and Infinite Dimensional Systems,''  where this work began. The first and second authors also thank IHES program ``Trimester on Nonlinear Waves,'' where part of this work was completed. The first and third authors gratefully acknowledge support from the Institute for Advanced Study, where part of this work was completed.
}

\maketitle

\section{Introduction}

 We study the Cauchy problem for the power-type nonlinear wave equation in $\R^{1+3}$, 
\begin{equation} \label{eq:nlw} 
\left\{
\begin{aligned}
&\Box  u  =    \pm u \abs{u}^{p-1} \\ 
& \vec u(0) = (u_{0}, u_{1}) \quad u  = u(t, x), \quad (t, x) \in \R^{1+3}_{t, x} .
\end{aligned}
 \right.
\end{equation}
Here $\Box  = - \p_{t}^{2} + \De$ so the $+$ sign above above yields the defocusing equation and the $-$ sign the focusing equation. The equation has the following scaling symmetry: if $\vec u(t, x) = (u, \p_t u)(t, x)$ is a solution, then so is 
\EQ{ \label{eq:scaling} 
\vec u_{\la}(t, x) =  \left( \la^{-\frac{2}{p-1}}u( t/ \la, x/ \la ),   \la^{-\frac{2}{p-1} -1} \p_t u (t/ \la, x/ \la)\right).
} 
The conserved energy, or Hamiltonian,   is 
\EQ{ \label{eq:en} 
E( \vec u(t))  = \int_{\{t\} \times \R^{3}} \frac{1}{2} \left(\abs{u_{t}}^{2} + \abs{\na u}^{2}\right)   \pm  \frac{1}{p+1} \abs{u}^{p+1} \, \ud x  = E(\vec u(0))
}
which scales like 
\[
E( \vec u_{\la}) =  \la^{3 - 2\frac{p+1}{p-1}} E( \vec u). 
\]
The energy is invariant under the scaling of the equation only when $p= 5$, which is referred to as the energy critical exponent. The range $p<5$ is called energy subcritical, since concentration of a solution by rescaling requires divergent energy, i.e.,   $\la \to 0 \Rightarrow E( \vec u_{\la}) \to \infty$. Conversely, the  range $p>5$ is called energy supercritical,  and here $E( \vec u_{\la})  \to 0 \mas \la \to 0$, i.e., concentration by rescaling is energetically favorable.  

Fixing $p$,  the critical Sobolev exponent $s_p:= \frac{3}{2} - \frac{2}{p-1}$ is defined to be the unique $s_p \in \R$ so that  $\dot{H}^{s_p} \times \dot{H}^{s_p-1}( \R^3)$ is invariant under the scaling~\eqref{eq:scaling}.  We will often use the shorthand notation 
\EQ{
 \dot \HH^s:=  \dot{H}^s \times \dot{H}^{s-1}(\R^3). 
 }

The power-type wave equation on $\R_{t, x}^{1+ 3}$ has been extensively studied. In the defocusing setting, the positivity of the conserved energy can be used to extend a local existence result to a global one for sufficiently regular initial data. In 1961, J\"orgens showed global existence for the defocusing equation for smooth compactly supported data~\cite{Jorg}.  In 1968, Strauss proved global existence for smooth solutions and moreover that these solutions decay in time and scatter to free waves~\cite{Strauss68} -- this remarkable paper was the first work that proved scattering for \emph{any} nonlinear wave equation. There are many works extending the local well-posedness theorem of Lindblad and Sogge~\cite{LinS} in $\dot\H^{s}$ for $s > s_p$ to an unconditional global well-posedness statement and we refer the reader to \cite{KPV00, GP03, BC06, Roy09} and references therein. These works do not address global dynamics of the solution, in particular scattering.  In the radial setting the first author has made significant advances in this direction, proving in~\cite{D16, D17} an unconditional global well-posedness and scattering result for the defocusing cubic equation for data in a Besov space with the same scaling as $\dot \H^{\frac{1}{2}}$.  In very recent work \cite{D18a, D18b}, the first author has proved unconditional scattering for the defocusing equation for radial data in the critical Sobolev space in the entire range $3\leq p<5$. 

The goal of this paper is to address global dynamics for \eqref{eq:nlw} in the non-radial setting. Our main result is the following theorem. 
\begin{thm}[Main Theorem] \label{t:main} 
 Consider~\eqref{eq:nlw} for energy subcritical exponents $3<p<5$ and with the \emph{defocusing} sign. Let $\vec u(t) \in \dot{\H}^{s_p}(\R^3)$ be a solution to~\eqref{eq:nlw} on its maximal interval of existence $I_{\max}$. Suppose that 
\EQ{ \label{eq:fantasy} 
\sup_{t \in I_{\max}} \| \vec u(t) \|_{\dot{\H}^{s_p} (\R^{3})} < \infty.
}
Then, $\vec u(t)$ is defined globally in time, i.e., $I_{\max} = \R$. In addition, we have 
\EQ{ \label{scattering} 
\| u \|_{L^{2(p-1)}_{t, x}( \R^{1+3})} <  \infty, 
}
which implies that $\vec u(t)$ scatters to a free wave in both time directions, i.e., there exist solutions $\vec v_{L}^{\pm}(t)  \in \dot{\H}^{s_p}(\R^3)$ to the free wave equation, $\Box v_L^{\pm} = 0$,  so that 
\EQ{
\| \vec u(t)- \vec v^{\pm}_L(t) \|_{\dot{\H}^{s_p} (\R^3)} \longrightarrow 0 \mas t \rightarrow \pm \I.
}
\end{thm}

A version of Theorem~\ref{t:main} restricted to radially symmetric data was established in  the earlier work of Shen \cite{Shen}; see also \cite{DL1} for the cubic power.  This type of conditional scattering result first appeared in the work of Kenig and Merle \cite{KM10} in the setting of the 3$d$ cubic NLS and has since attracted a great deal of research activity, see e.g. \cite{KM11a, KM11b, KV11a, KV11b, Bul12a, Bul12b, DL2, CR5d, DKM5, DR} for this type of result for the NLW. 

In the energy critical regime, the bound \eqref{eq:fantasy} is guaranteed by energy conservation and  the analogue of Theorem \ref{t:main} was proved in the seminal works of Shatah and Struwe \cite{SS93,SS94}, Bahouri and Shatah \cite{BS} and Bahouri and G\'erard \cite{BG}.  In the energy supercritical regime, the analogue of Theorem \ref{t:main} was obtained by Killip and Visan in  \cite{KV11b}.

The regime treated in this work, namely energy-subcritical with non-radial data,  necessitates several new technical developments, which may prove useful in contexts beyond the scope of Theorem~\ref{t:main}.

\begin{rem}
It is conjectured that for the defocusing equation all solutions with data in $\dot \H^{s_p}$ scatter in both time directions as in the energy critical case $p=5$. Theorem~\ref{t:main} is a conditional result, specifically we do not determine a priori which data satisfy~\eqref{eq:fantasy}.   It is perhaps useful to think of the theorem in its contrapositive formulation: if initial data in the critical space $\dot \H^{s_p}$ were to lead to an evolution that does not scatter in forward time, then the $\dot \H^{s_p}$ norm of the solution must diverge along at least one sequence of times tending to the maximal forward time of existence.  
\end{rem}

\begin{rem}
The dynamics are much different in the case of the energy subcritical \emph{focusing} equation. In remarkable works, Merle and Zaag~\cite{MZ03ajm, MZ05ma} classified the blow up dynamics by showing that all blow-up solutions must develop the singularity  \emph{at the self-similar rate}.  In the radial case, an infinite family of smooth self-similar solutions is constructed by Bizo\'n et al. in~\cite{BBMW10}.  In~\cite{DS12, DS17}, Donninger and Sch\"orkhuber address the stability of the self-similar blow up.
\end{rem}

\subsection{Comments about the proof} 
The proof of Theorem~\ref{t:main} follows the fundamental concentration compactness/rigidity method of Kenig and Merle, which first appeared in~\cite{KM06, KM08}. The proof is by contradiction -- if Theorem~\ref{t:main} were to fail, the profile decomposition of Bahouri-G\'erard~\cite{BG} yields a minimal nontrivial solution to~\eqref{eq:nlw}, referred to as a \emph{critical element} and denoted by $\vec u_c$, that does not scatter. Here `minimal' refers to the size of the norm in~\eqref{eq:fantasy}. This standard construction is outlined in Section~\ref{s:cc}. The key feature of a critical element is that its trajectory is pre-compact modulo symmetries in the space $\dot \H^{s_p}$, see Proposition \ref{p:ce}. The proof is completed by showing that  this compactness property is too rigid for a nontrivial  solution and thus the critical element cannot exist.  

The major obstacle to rule out a critical element $\vec u_{c}(t)$ in this energy subcritical setting is the fact that  $\vec u_{c}(t)$ is \emph{a priori} at best an $\dot \H^{s_p}$ solution, while all known global monotonicity formulae, e.g.,  the conserved energy, virial and Morawetz type inqualities, require more regularity.  In general, solutions to a semilinear wave equation are only as regular as their initial data because of the free propagator $S(t)$ in the Duhamel representation of a solution 
\EQ{\label{eq:du}
\vec u_c(t_0)  = S(t_0-t) \vec u_c(t)  + \int_{t}^{t_0} S(t_0-\tau) (0, \pm \abs{u}^{p-1} u (\tau)) \, \ud \tau.
}
However, for a \emph{critical element}  the pre-compactness of its trajectory is at odds with the dispersion of the free part, $S(t_0-t) \vec u(t)$, which means  the first term on the right-hand-side above must vanish weakly as $t  \to \sup I_{\max} $ or as $t \to \inf I_{\max}$, where $I_{\max}$ is as in Theorem \ref{t:main}. Thus, the Duhamel integral on the right-hand-side of~\eqref{eq:du} encodes the regularity of a critical element and additional regularity can be expected due to the nonlinearity. As in~\cite{DL1} the gain in regularity  at a fixed time $t_0$ is observed via the so-called  ``double Duhamel trick'', which refers to the analysis of the pairing 
\EQ{ \label{eq:pair} 
\bigg\langle \int_{T_1}^{t_0} S(t_0-t) (0, \pm \abs{u}^{p-1} u) \, \ud t, \,  \int^{T_2}_{t_0} S(t_0-\tau) (0, \pm \abs{u}^{p-1} u) \, \ud\tau \bigg\rangle
} 
where we take $T_1<t_0$ and $T_2>t_0$.  The basic outline of this technique was introduced  by Tao in~\cite{Tao07} and was used within the Kenig-Merle framework for nonlinear Schr{\"o}dinger equations by Killip and Visan~\cite{KV10CPDE, KV10AJM, KVClay},  and for nonlinear wave equations in, e.g.,~\cite{KV11b, Shen}. This method is also closely related to the in/out decomposition used by Killip, Tao, and  Visan in~\cite[Section~$6$]{KTV09}. 

Here we employ several novel interpretations of the double Duhamel trick, substantially building on the simple implementation developed by the first two authors in the radial setting in~\cite{DL1, DL2} for $p=3$, which exploited the sharp Huygens principle to overcome the difficulties arising from the both the slow $\ang{t}^{-1}$ decay of $S(t)$ in dimension $3$ and the small power $p=3$ that precluded this case from being treated by techniques introduced in earlier works. The general case (non-radial data) considered here requires several new ideas.  
\medskip

We briefly describe the set-up and several key components of the proof. A critical element has compact trajectory up to action by one parameter families (indexed by $t \in I_{\max}( \vec u_c)$)  of translations $x(t)$ that mark the \emph{spatial center} of the bulk of $\vec u_c(t)$,   and rescalings $N(t)$ that record the \emph{frequency scale} at which $\vec u_c(t)$ is concentrated.  In Section~\ref{s:cc} we perform a reduction to four distinct behaviors of the parameters $x(t)$ and $N(t)$. First, following the language of~\cite{KV11b} we distinguish between $x(t)$ that are \emph{subluminal}, roughly that $|x(t) - x(\tau)| \le (1-\de) \abs{t-\tau}$ for some $\de>0$, and those that \emph{fail to be subliminal}, i.e., if $x(t)$ forever moves at the speed of light, or more precisely, $\abs{x(t)} \simeq  \abs{t}$ (in a certain sense) for all $t$. The latter case is quite delicate in this energy-subcritical setting and we introduce several new ideas to treat it, see Section \ref{s:hans}. We elaborate further on these two cases. 

\medskip
\emph{Subluminal critical elements.} When $x(t)$ is subluminal,  we distinguish between what we call a \emph{soliton-like critical element} where $N(t) = 1$, a \emph{self-similar-like critical element} where $N(t) = t^{-1}$, $t >0$, and a \emph{global concentrating critical element} where,  $\limsup_{t \to \infty} N(t) = \infty$.  These distinct cases are treated in Sections $4, 5, 6$ respectively.

In Section~\ref{s:soliton}, we set out to show as in~\cite{DL1} that soliton-like critical elements must be uniformly bounded in $\dot \H^{1+\eps} \cap \dot\H^{s_p}$ and hence the trajectory is pre-compact in $\dot \H^1$.  Once this is accomplished we can access nonlinear monotonicity formulae to show that such critical elements cannot exist. In this latter step we employ a version of a standard argument based on virial identity, after shifting the spatial center of the solution to $x =0$ by the Lorentz group, which is compactified by the bound in $\dot \H^1$. The heart of the argument in Section $4$ is thus establishing the additional regularity of a soliton-like critical element. The goal, roughly, is to show that the pairing~\eqref{eq:pair} can be estimated in $\dot \H^1$.   In~\cite{DL1} the proof relied crucially on radial Sobolev embedding. As this is no longer at our disposal in the current, non-radial setting, we have introduced a substantial reworking of the argument from~\cite{DL1} that both simplifies it, and removes the reliance on radial Sobolev embedding. Examining the pairing~\eqref{eq:pair} at time $t_0 =0$ we divide spacetime into three types of regions;  see Figure~\ref{f:bowtie}.  The first region is a fixed time interval of the form  $[t_0-R, t_0+ R]$, where $R>0$ is chosen so that the bulk of $\vec u_c(t)$ is captured by the light cone emanating from $(t_0, 0)$ in both time directions. In this region~\eqref{eq:pair} is estimated using an argument based on Strichartz estimates, using crucially that $R>0$ is finite and can be chosen independently of $t_0$ by compactness.  The second region is the region of spacetime exterior to this time interval and exterior to the cone. Here the $\dot \H^{s_p}$ norm of the solution is small on any fixed time slice and hence an argument based on the small data theory can be used to absorb the time integrations in~\eqref{eq:pair}. Lastly, the heart of the double Duhamel trick is employed to note the interaction between the two regions in the interior of the light cone, one for times $t< -R$ and the other for times $t>R$ is identically $=0$ by the sharp Huygens principle!

In Section~\ref{s:ss} we show that a self-similar-like critical element cannot exist. Here we again use a double Duhamel argument centered at $t_0 \in (0, \infty)$, but with  $T_1 = \inf I_{\max} = 0$ and $T_2 =\sup I_{\max} = \infty$ in~\eqref{eq:pair}. The argument exploiting Huygens principle given in Section $4$ no longer applies since the forward and backwards cones emanating from time, say $t_0 = 1$ can never capture the bulk of the solution  since $N(T) =T^{-1}$ is an expression of the fact that the solution is localized to the physical scale $T$ at time $T$, see \eqref{r:Reta}.  However here, we use a different argument based on a version of the long-time Strichartz estimates introduced by the first author in~e.g.,~\cite{D-JAMS, D-Duke}, which allow us to control Strichartz norms of the projection of $\vec u_c$ to high frequencies $k \gg 1$ on time intervals $J$ which are long in the sense that $\abs{J} \simeq 2^{\al k}$ for $\al\ge1$. 

In Section~\ref{s:sword},  $N(t)$ is no longer a given fixed function. We establish a dichotomy which we refer to colloquially as the sword or the shield: either additional regularity for the critical element can be established using essentially the same argument used in Section~\ref{s:soliton-reg}, or a self-similar-like critical element can be extracted by passing to a suitable limit. To apply the argument from Section $4$ the following must be true -- fixing any time $t_0$,  the amount of time (\emph{but where now time is measured relative to the scale $N(t)$}) that one has to wait until the bulk of the solution is absorbed by the cone emanating from time $t_0$ must be uniform in $t_0$. We define functions $C_{\pm}(t_0)$ whose boundedness (or unboundedness)  measures whether or not this criteria is satisfied; see the introduction to Section~\ref{s:sword}. The rest of the section is devoted to showing how to apply the arguments from Section~\ref{s:soliton} in the case where $C_{\pm}(t_0)$ are uniformly bounded, and how to extract a self-similar solution-like critical element in the case that one of $C_\pm(t_0)$ are not bounded. 

\medskip
\emph{Critical elements that are not subluminal:} In Section $7$ we show that critical elements with spatial center $x(t)$ traveling at the speed of light cannot exist. The technique in this section is novel and may be useful in other settings. First we note that such critical elements are easily ruled out for solutions with finite energy, as is shown in~\cite{KM08, Tao37, NakS} using an argument based on the conserved momentum,  and even in the energy supercritical setting; see~\cite{KV11b} using the energy/flux identity. None of these techniques (which provide an \emph{a priori} limit on the speed of $x(t)$) apply in our setting so we must rule out this critical element by other means, namely, \emph{by first showing that such critical elements have additional regularity}. 

In Section \ref{subsec:analysis_of_props} we lay the necessary groundwork and show, using finite speed of propagation, that any such critical element must have a fixed scale, i.e., $N(t) = 1$ and $x(t)$ must choose a fixed preferred direction up to deviation in angle by $\frac{1}{\sqrt{t}}$. The model case one should consider is $x(t) = (t, 0, 0)$ for all $t \in \R$, which means that the bulk of $\vec u_c(t)$ travels along the $x_1$-axis at speed $t$. We are able to show that such critical elements have up to $1-\nu$ derivatives in the $(x_2, x_3)$ directions for any $\nu>0$. This is enough to show that such critical elements cannot exist via a Morawetz estimate adapted to the direction of $x(t)$ -- this is the only place in the paper where the arguments are limited to the defocusing equation. 

The technical heart of this section is the proof of extra regularity ($1-\nu$ derivatives)  in the $x_{2, 3}$ directions. We again divide spacetime into three regions. For a solution projected to a fixed frequency $N \gg1$, we call region $A$ the strip $[0, N^{1-\eps}] \times \R^3$ for $\eps>0$ sufficiently small relative to $\nu$. On this region we can control the solution  by a version of the long-time Strichartz estimates proved in Section~\ref{s:ltse}. At time $t = N^{1-\eps}$ we then divide the remaining part of spacetime for positives times into two regions. Region $B$ is the set including all times $t \ge N^{1-\eps}$ exterior to the light cone of initial width $R(\eta_0)$ emanating from the point $(t, x) = (N^{1-\eps}, x(N^{1-\eps}))$ where $R(\eta_0)$ is chosen large enough so that at  $\vec u_v(N^{1-\eps})$ has $\dot \H^{s_p}$ norm less than $\eta_0$ exterior to the ball of radius $R(\eta_0)$ centered at $x(N^{1-\eps})$. The solution is then controlled on region $B$ using small data theory. Estimating the interaction of the two terms in the pair~\eqref{eq:pair} on the remaining region $C$ (the region $\{\abs{x - x(N^{1-\eps})} \le R(\eta_0) + t- N^{1-\eps}, t \ge N^{1-\eps} \}$) and the analogous region $C'$ for negative times $\tau \le -N^{1-\eps}$ provides the most delicate challenge. Any naive implementation of the double Duhamel trick based on Huygens principle is doomed to fail here since the left and right-hand components of the pair~\eqref{eq:pair} restricted to $C, C'$ interact \emph{in the wave zone} $\abs{x} \simeq \abs{t}$. Furthermore, since we are in dimension $d=3$, the $\ang{t}^{-1}$  decay from the wave propagator $S(t)$ in~\eqref{eq:pair} is not sufficient for integration in time. For this reason we introduce an auxiliary frequency localization to frequencies  $\abs{(\xi_{2}, \xi_{3})} \simeq M$ in the $(\xi_{2}, \xi_{ 3})$ directions after first localizing in all directions to frequencies $\abs{\xi} \simeq N$.  We call this angular frequency localization $\hat P_{N, M}$. The key observation is that the intersection of the wave zone $\{\abs{x} \simeq \abs{t}\}$ with region $C$ requires the spatial variable $x = (x_1, x_{2, 3})$ to satisfy 
\EQ{
\frac{\abs{x_{2, 3}}}{\abs{x}} \ll \frac{M}{N}
}
for all $M \ge N^{\frac{s_p}{1-\nu}}$ as long as $\eps>0$ is chosen small enough relative to $\nu$, whereas 
application of $\hat P_{N, M}$ restricts to frequencies $\xi = (\xi_1, \xi_{2, 3})$ with 
\EQ{
\frac{\abs{\xi_{2, 3}}}{ \abs{\xi}} \simeq \frac{M}{N}.
}
This yields \emph{angular separation} in the kernel of $\hat P_{N, M} S(t)$ and allows us to deduce arbitrary time decay for the worst interactions in~\eqref{eq:pair}, see Lemma \ref{l:angle}. The remaining interactions in~\eqref{eq:pair} are dealt with using an argument based on the sharp Huygens principle, which is complicated due to the blurring of supports caused by $\hat P_{N, M}$.


\begin{rem}\label{r:p=3a} 
The proof of Theorem~\ref{t:main} serves as the foundation for the more complicated case of the cubic equation, $p=3$, as well as for the analogous result for the focusing equation; see for example~\cite{DL1} where the focusing and defocusing equations are treated in the same framework in the radial setting. 

 Much of the argument given here carries over to the defocusing equation when $p=3$. However, in this case  we have $s_p = 1/2$ and the critical space  $\dot \HH^{\frac{1}{2}}$ is the unique Sobolev space that is invariant under Lorentz transforms. This introduces several additional difficulties, described more in detail in Remark~\ref{r:p=3b}.  Additionally, certain estimates in Section \ref{s:hans} fail at the $p=3$ endpoint and would require modification.
 
 Similarly the argument in Sections 4-6 applies equally well to the focusing equation. However the argument in Section \ref{s:hans} used to rule out the traveling wave critical element is specific to the defocusing equation as it relies on a Morawetz-type estimate only valid in that setting. 
\end{rem} 

\section{Preliminaries} 

\subsection{Notation, definitions, inequalities} 

We write $A\lesssim B$ or $B\gtrsim A$ to denote $A\leq CB$ for some $C>0$.  Dependence of implicit constants will be denoted with subscripts.  If $A\lesssim B\lesssim A$, we write $A\simeq B$.  We will use the notation $a\pm$ to denote the quantity $a\pm\eps$ for some sufficiently small $\eps>0$. 

We will denote by $P_N$  the Littlewood-Paley projections onto frequencies of size $ \abs{ \xi} \simeq N$ and by $P_{ \le N}$ the projections onto frequencies of size $ \abs{ \xi} \lesssim N$. Often we will consider the case when $N = 2^k$, $k \in \Z$ is a dyadic number and in this case we will employ the following notation: when write $P_k$ with a \emph{lowercase} subscript $k$ this will mean projection onto frequencies $\abs{\xi} \simeq 2^k$.  We will often write $u_N$ for $P_N u$, and similarly for $P_{\leq N}$, $P_{>N}$, $P_k$, and so on. 

These projections satisfy Bernstein's inequalities, which we state here.  
\begin{lem}[Bernstein's inequalities]\emph{\cite[Appendix A]{Taobook}} \label{l:bern} Let $1 \le p \le q \le \infty$ and $s \ge 0$. Let $ f: \R^d \to \R$. Then 
\EQ{ \label{bern}
&\|P_{\ge N} f\|_{L^p} \lesssim N^{-s} \| \abs{\na}^s P_{\ge N} f\|_{L^p},\\
&\|P_{\le N} \abs{\na}^s f\|_{L^p} \lesssim N^{s} \|  P_{\le N} f\|_{L^p}, \, \quad
\|P_{ N} \abs{\na}^{\pm s} f\|_{L^p} \simeq N^{ \pm s} \|  P_{ N} f\|_{L^p}\\
&\|P_{\le N} f\|_{L^q} \lesssim N^{\frac{d}{p}- \frac{d}{q}} \| P_{\le N} f\|_{L^p}, \, \quad
\|P_{ N} f\|_{L^q} \lesssim N^{\frac{d}{p}- \frac{d}{q}} \| P_{N} f\|_{L^p}.
}
\end{lem}

We will write either
\[
\|u\|_{L_t^q L_x^r(I\times\R^3)} \quad\text{or}\quad \|u\|_{L_t^q(I;L_x^r(\R^3)}
\]
to denote the space-time norm
\[
\biggl(\int_I \biggl(\int_{\R^3} |u(t,x)|^{q,r}\,dx\biggr)^{\frac{q}{r}}\,dt\biggr)^{\frac{1}{q}},
\]
with the usual modifications if $q$ or $r$ equals infinity. 

Given $s\in\R$ we define the space $\dot\H^s$ by
\[
\dot\H^s=\dot H^s(\R^3)\times\dot H^{s-1}(\R^3). 
\]
For example, we work with initial data in $\dot \H^{s_p}$. 

We also require the notion of a frequency envelope. 
\begin{defn}{\cite[Definition~$1$]{Tao1}} \label{d:fren} A \emph{frequency envelope} is a sequence $\be = \{\be_k\}$ of positive numbers with $\be \in \ell^2$ satisfying the local constancy condition 
\ant{
2^{-\s\abs{j-k}} \be_k \lesssim \be_j \lesssim 2^{\s\abs{j-k}} \be_k,
}
where  $\s>0$ is a small, fixed constant.  If $\be$ is a frequency envelope and $(f, g) \in \dot{H}^s \times \dot{H}^{s-1}$ then we say that \emph{$(f, g)$  lies underneath $\be$} if 
\ant{
\| (P_k f,P_k g)\|_{\dot H^s \times \dot{H}^{s-1}} \le \be_k \, \quad \forall k \in \Z.
} 
Note that if $(f, g)$ lies underneath $\be$ then we have 
\ant{
\| (f, g)\|_{\dot H^s \times \dot{H}^{s-1}} \lesssim \| \be\|_{\ell^2(\Z)}.
}
\end{defn}

In practice, we will need to choose the parameter $\sigma$ in the definition of frequency envelope sufficiently small depending on the power $p$ of the nonlinearity.

We next record a commutator estimate. 

\begin{lem}\label{L:commutator} Let $\chi_R$ be a smooth cutoff to $|x|\geq R$.  For $0\leq s\leq 1$ and $N\geq 1$,
\begin{align*}
\|P_N\chi_R f - \chi_R P_N f\|_{L^2} &\lesssim N^{-s}(N R)^{-(1-s)}\|f\|_{\dot H^s}, \\
\|P_N\chi_R f - \chi_R P_N f\|_{L^2} &\lesssim R^{-s}(N R)^{-(1-s)}\|f\|_{\dot H^{-s}}. 
\end{align*}
\end{lem}

\begin{proof} We write the commutator as an integral operator in the form
\[
[P_N\chi_R f - \chi_R P_N f](x) = N^{d}\int K(N(x-y))[\chi_R(x)-\chi_R(y)]f(y)\,dy. 
\]
Thus, using the pointwise bound
\[
|\chi_R(x)-\chi_R(y)| \lesssim N|x-y|\cdot N^{-1}R^{-1}
\]
and Schur's test, we first find
\[
\|P_N\chi_R f - \chi_R P_N f\|_{L^2} \lesssim N^{-1}R^{-1} \|f\|_{L^2}.
\]
Next, a crude estimate via the triangle inequality, Bernstein's inequality, H\"older's inequality, and Sobolev embedding gives
\begin{align*}
\|P_N\chi_R f - \chi_R P_N f\|_{L^2} & \lesssim N^{-1}\|\nabla(\chi_R f)\|_{L^2} + N^{-1}\|\nabla f\|_{L^2} \lesssim N^{-1}\|f\|_{\dot H^1}. 
\end{align*}
The first bound now follows from interpolation. For the second bound, we write
\[
[P_N\chi_R f - \chi_R P_N f](x) = N^{d}\int K(N(x-y))[\chi_R(x)-\chi_R(y)]\nabla\cdot\nabla\Delta^{-1}f(y)\,dy
\]
and integrate by parts.  Estimating as above via Schur's test, we deduce
\[
\|P_N\chi_R f - \chi_R P_Nf\|_{L_x^2} \lesssim R^{-1} \||\nabla|^{-1}f\|_{L^2},
\]
so that the second bound also follows from interpolation.  \end{proof}

\subsection{Strichartz estimates} 
The main ingredients for the small data theory are Strichartz estimates for the linear wave equation in $\R^{1+3}$, 
\EQ{ \label{eq:lw}
&\Box v = F,\\
&\vec v(0) = (v_0, v_1).
}
A free wave means a solution to~\eqref{eq:lw} with $F=0$ and will be often denoted using the propagator notation $\vec v(t) = S(t) \vec v(0)$. We define a pair $(r, q)$ to be wave-admissible in $3d$ if 
\EQ{ \label{adm} 
q, r \ge 2,  \, \, \frac{1}{q} + \frac{1}{r}   \le \frac{1}{2}, \quad  ( q, r) \neq (2, \infty)
}
The Strichartz estimates stated below are standard and we refer to~\cite{Kee-Tao, LinS}, the book~\cite{Sogge}, and references therein for more.  

\begin{prop}[Strichartz Estimates]\emph{\cite{Kee-Tao, LinS, Sogge}} \label{p:strich}Let $\vec v(t)$ solve~\eqref{eq:lw} with data $\vec v(0) \in \dot{H}^s \times \dot{H}^{s-1}(\R^3)$, with $s >0$. Let $(q, r)$, and $(a, b)$ be admissible pairs satisfying the gap condition 
\EQ{
\frac{1}{q} + \frac{3}{r} = \frac{1}{a'} + \frac{3}{b'} - 2  = \frac{3}{2} - s.
}
where $(a', b')$ are the conjugate exponents of $(a, b)$. 
Then, for any time interval $I \ni 0$ we have the bounds 
\EQ{\label{eq:str}
\| v \|_{L^{q}_t(I; L^r_x)} \lesssim  \| \vec v(0) \|_{\dot{H}^s \times \dot{H}^{s-1}} + \|F\|_{L^{a'}_t(I; L^{b'}_x)}.
}
\end{prop}

\subsection{Small data theory -- global existence, scattering, perturbative theory}
A standard argument based on Proposition~\ref{p:strich} yields the scaling critical small data well-posedness and scattering theory. We define the following notation for a collection of function spaces that we will make extensive use of. In this sub-section we fix $p \in [3, 5]$ (later we will fix $p \in (3, 5)$) and let $I  \subset \R$ be a time interval. We define 
\EQ{
&S(I) : = L^{2(p-1)}_t ( I; L^{2(p-1)}_x( \R^3)).
}
For example, when $p = 3$, $S = L^{4}_{t, x}$ while for $p =5$ we have $S = L^{8}_{t, x}$.

\begin{rem} 
There are a few other function spaces related to 
\EQ{ \label{eq:Hdef} 
\dot \HH^{s_p}:= \dot{H}^{s_p} \times \dot{H}^{s_p-1}(\R^3)
}
 that will appear repeatedly in our analysis. First note the Sobolev embedding $\dot{H}^{s_p}(\R^3) \hookrightarrow L^{\frac{3}{2}(p-1)}(\R^3)$, which means 
\EQ{
\| f \|_{L^{\frac{3}{2}(p-1)}(\R^3)} \lesssim  \| f \|_{\dot{H}^{s_p}(\R^3)}.
} 
\end{rem}

\begin{prop}[Small data theory]\label{small data}
Let $3 \le p <5$ and suppose that  $\vec u(0) = (u_0, u_1) \in \dot{H}^{s_p} \times \dot{H}^{s_p-1}(\R^3)$. Then there is a unique solution $\vec u(t) \in \dot\H^{s_p}$ with maximal interval of existence $I_{\max}( \vec u )= (T_-(\vec u), T_+( \vec u))$.  Moreover, for any compact interval $J \subset I_{\max}$, 
\ant{ 
\| u\|_{S(J)} < \infty.
}
Additionally, a globally defined solution $\vec u(t)$ on $t\in [0, \I)$ scatters as $t \to \I$ to a free wave
if and only if  $ \|u \|_{S([0, \infty))}< \infty$. In particular, there exists a constant $\de_0>0$ so that  
\EQ{ \label{eq:apsmall}
 \| \vec u(0) \|_{\dot{H}^{s_p} \times \dot{H}^{s_p-1}} < \de_0 \Longrightarrow  \| u\|_{S(\R)} \lesssim \|\vec u(0) \|_{\dot{H}^{s_p} \times \dot{H}^{s_p-1}}  \lesssim \de_0
 }
and thus $\vec u(t)$ scatters to free waves as $t \to \pm \infty$. Finally, we have the standard finite time blow-up criterion: 
\EQ{ \label{ftbuc}
T_+( \vec u)< \infty \Longrightarrow \|u \|_{S([0,T_+( \vec u)) } = + \infty
}
An analogous statement holds if $-\infty< T_-( \vec u)$. 
\end{prop}

The concentration compactness procedure in Section~\ref{s:cc} requires the following nonlinear perturbation lemma for approximate solutions to~\eqref{eq:nlw}. 

\begin{lem}[Perturbation Lemma]\emph{\cite{KM06, KM08}} \label{l:pert}  
There exist continuous functions $\e_0,C_0: (0,\infty) \to (0,\infty)$ so that the following holds true. 
Let $I\subset \R$ be an open interval (possibly unbounded), $\vec u, \vec v\in C(I; \dot{H}^{s_p} \times \dot{H}^{s_p-1})$  satisfying for some $A>0$
\ant{
 \|\vec v\|_{L^\infty(I;\dot{H}^{s_p} \times \dot H^{s_p-1})} +   \|v\|_{S(I)} & \le A \\
 \|\abs{\na}^{s_p-\frac{1}{2}}\glei(u)\|_{L^{\frac{4}{3}}_t(I; L^{\frac{4}{3}}_x)}
   + \|\abs{\na}^{s_p-\frac{1}{2}}\glei(v)\|_{L^{\frac{4}{3}}_t(I; L^{\frac{4}{3}}_x)} + \|w_0\|_{S(I)} &\le \e \le \e_0(A),
   }
where $\glei(u):=\Box u \pm \abs{u}^{p-1} u$ in the sense of distributions, and $\vec w_0(t):=S(t-t_0)(\vec u-\vec v)(t_0)$ with $t_0\in I$ fixed, but arbitrary.  Then
\ant{ 
  \|\vec u-\vec v-\vec w_0\|_{L^\infty(I;\dot{H}^{s_p} \times \dot H^{s_p-1})}+\|u-v\|_{S(I)} \le C_0(A)\e.}
  In particular,  $\|u\|_{S(I)}<\I$.
\end{lem}

\section{Concentration compactness and the reduction of Theorem~\ref{t:main}} \label{s:cc}

We begin the proof of Theorem~\ref{t:main} using the concentration compactness and rigidity method of Kenig and Merle~\cite{KM06, KM08}. The concentration compactness aspect of the argument is by now standard and we follow the scheme from~\cite{KM10}, which is a refinement of the scheme in~\cite{KM06, KM08}. The main conclusion of this section is the following:  If Theorem~\ref{t:main} fails, there exists a minimal, nontrivial, non-scattering solution to~\eqref{eq:nlw}, which we call a \emph{critical element}.

We follow the notation from~\cite{KM10} for convenience.  Given initial data $ (u_0, u_1) \in \dot H^{s_p} \times \dot H^{s_p-1}$ we let  $\vec u(t) \in \dot H^{s_p} \times \dot H^{s_p-1}$ be the unique solution to~\eqref{eq:nlw} with data $\vec u(0) = (u_0, u_1)$ and maximal interval of existence $I_{\max}(\vec u) := (T_-( \vec u), T_+( \vec u))$.  

Given  $A>0$, set 
\EQ{
\B(A):= \{ (u_0, u_1) \in\dot H^{s_p} \times \dot H^{s_p-1}   \, : \,      \|\vec u(t)\|_{L^{\infty}_t(I_{\max}(\vec u); \dot H^{s_p} \times \dot H^{s_p-1})} \le A\}.
}
\begin{defn} We say that $\mathcal{SC}(A;  \vec u(0))$ holds if $\vec u(0) \in \B(A)$, $I_{\max} (\vec u) = \R$ and $ \|u \|_{S(I)} < \infty$. In addition, we will say that $\mathcal{SC}(A)$ holds if for every $ (u_0, u_1) \in \B(A)$ one has  $I_{\max}(\vec u) = \R$ and $ \|u \|_{S(I)} < \infty$. 
\end{defn} 
\begin{rem}
Recall from Proposition~\ref{small data} that $\| u \|_{S(I)}< \infty$ if and only if $\vec u(t)$ scatters to a free waves as $t \to \pm \I$. Thus,  Theorem~\ref{t:main} is equivalent to the statement that $\mathcal{SC}(A)$ holds for all $A>0$. 
\end{rem} 

Now suppose that Theorem~\ref{t:main} {\em fails to be true}. By Proposition~\ref{small data}, there exists an $A_0>0$ small enough so that $\mathcal{SC}(A_0)$ holds.  Since we are assuming that Theorem~\ref{t:main} fails, we can find a threshold value $A_C$ so that for $A<A_C$, $\mathcal{SC}(A)$ holds, and for $A>A_C$, $\mathcal{SC}(A)$ fails. Note that we must have  $0<A_0<A_C$. The Kenig-Merle concentration compactness argument is now used to produce a \emph{critical element}, namely a minimal non-scattering solution $\vec u_{\textrm{c}}(t)$ to~\eqref{eq:nlw} so that $\mathcal{SC}(A_C, \vec u_{\textrm{c}})$ fails, and which enjoys certain compactness properties.

We state a refined version of this result below, and we refer the reader to~\cite{KM10, Shen, TVZ, TVZ08} for the details. As usual, the deep foundations of the concentration compactness part of the Kenig-Merle framework are profile decompositions of Bahouri and G\'erard \cite{BG} used in conjunction with the nonlinear perturbation theory in Lemma~\ref{l:pert}. 

\begin{prop} \label{p:ce} Suppose Theorem~\ref{t:main} fails to be true. Then, there exists a solution $\vec u(t)$ such that $\mathcal{SC}(A_C;  \vec u)$ fails, which we call a \emph{critical element}.  We can assume that $\vec u(t)$ does not scatter in either time direction, i.e.,  
\EQ{\label{blow up}
\|u\|_{S((T_-(\vec{u}), 0])} = \|u\|_{S([0, T_+(\vec u))} =  \infty
}
and moreover, there exist continuous  functions
 \ant{
 &N: I_{\max}(\vec u) \to (0, \infty) , \quad
 x: I_{\max}(\vec u) \to \R^3
 }
 so that the set 
{\small\EQ{\label{eq:K} 
 \left\{ \left(\tfrac{1}{N(t)^{\frac{2}{p-1}}} u\left( t, x(t) + \tfrac{ \cdot}{N(t)} \right),  \tfrac{1}{N(t)^{\frac{2}{p-1}+1}} u_t\left( t,   x(t) + \tfrac{\cdot}{N(t)} \right) \right) \mid t \in I_{\max}\right\}
} }
is pre-compact in $\dot\H^{s_p}$. 
\end{prop}

We make a few observations and reductions concerning the critical element found in Proposition~\ref{p:ce}. It will be convenient to proceed slightly more generally, starting by giving a name to the compactness property~\eqref{eq:K} satisfied by a critical element.

\begin{defn} \label{d:cp}  Let $I \ni 0$ be an interval and let $\vec u(t)$ be a nonzero solution to~\eqref{eq:nlw} on  $I$. We will say $\vec u(t)$ has the \emph{compactness property on $I$} if there are continuous functions $N: I \to (0, \infty)$ and $x: I \to \R^3$ so that the set 
{\small\ant{
 K_I:= \left\{ \left(\frac{1}{N^{\frac{2}{p-1}}(t)} u\left( t, \, x(t) + \frac{ \cdot}{N(t)} \right), \, \frac{1}{N^{\frac{2}{p-1}+1}(t)} u_t\left( t,  \,  x(t) + \frac{\cdot}{N(t)} \right) \right) \mid t \in I\right\}
}} 
is pre-compact in $\dot\HH^{s_p}$. 
\end{defn} 

We make the following standard remarks about solutions with the compactness property.  We begin with a local constancy property for the modulation parameters.

\begin{lem}\emph{\cite[Lemma 5.18]{KVClay}}\label{l:const} 
Let $\vec u(t)$ have the compactness property on a time interval $I \subset \R$ with parameters $N(t)$ and $x(t)$. Then there exists  constants $\eps_0>0$ and $C_0>0$ so that for every $t_0 \in I$ we have 
\EQ{
&[ t_0 - \frac{\eps_0}{ N(t_0)}, t_0 +  \frac{\eps_0}{N(t_0)}] \subset I,  \\ 
& \frac{1}{C_0}  \le \frac{N(t)}{N(t_0)} \le C_0 \mif \abs{ t- t_0} \le \frac{\eps_0}{N(t_0)}, \\ 
& \abs{ x(t) - x(t_0)} \le \frac{C_0}{N(t_0)} \mif \abs{ t- t_0} \le \frac{\eps_0}{N(t_0)} .
}
\end{lem}

\begin{rem}\label{r:Reta}
For a solution with the  \emph{compactness property} on an interval $I$,  we can, after modulation,  control the $\dot{\HH}^{s_p}$ tails uniformly in $t \in I$. Indeed,  for any $\eta > 0$ there exists $R(\eta) < \infty$ such that
\EQ{\label{3.3}
&\int_{|x -  x(t)| \geq \frac{R(\eta)}{N(t)}}||\nabla|^{s_p} u(t,x)|^{2} \ud x + \int_{|\xi| \geq R(\eta)N(t)} |\xi|^{2s_p} |\hat{u}(t,\xi)|^{2} d\xi  \le \eta, \\
&\int_{|x-x(t)| \geq \frac{R(\eta)}{N(t)}}||\nabla|^{s_p-1} u_t(t,x)|^{2} \ud x + \int_{|\xi| \geq R(\eta)N(t)} |\xi|^{2(s_p-1)}| \hat{u}_{t}(t,\xi)|^{2} d\xi  \le  \eta, 
}
for all $t \in I$.  We call $R(\cdot)$ the \emph{compactness modulus}. 
\end{rem}

We also remark that that any Strichartz norm of the linear part of the evolution of a solution with the compactness property on $I_{\max}$ vanishes  as $t \to T_- $ and as $t \to T_+$. A concentration compactness argument then implies that the linear part of the evolution vanishes weakly in $\dot{\HH}^{s_p}$, that is, 
for each $t_{0} \in I_{\max}$,
\begin{equation}
S(t_{0} - t) u(t) \rightharpoonup 0
\end{equation}
weakly in $\dot{\HH}^{s_p}$ as $t \nearrow \sup I$, $t \searrow \inf I$, 
see  \cite[Section $6$]{TVZ08}\cite[Proposition $3.6$]{Shen}. 
This implies the following lemma, which we use crucially in the  proof of Theorem~\ref{t:main}. 

\begin{lem}\emph{ \cite[Section $6$]{TVZ08}\cite[Proposition $3.6$]{Shen} }\label{l:weak} Let $\vec u(t)$ be a  solution to~\eqref{eq:nlw} with the compactness property on its maximal interval of existence $I = (T_-, T_+)$. Then for any $t_0 \in I$ we can write 
\ant{
\int_{t_0}^T S(t_0 - s) (0, \abs{ u}^{p-1}u )  \, ds \rightharpoonup \vec u(t_0)  \mas T \nearrow T_+ \,  \textrm{weakly in} \,\, \dot{\HH}^{s_p},\\
-\int^{t_0}_T S(t_0 - s) (0, \abs{u}^{p-1} u )  \, ds \rightharpoonup \vec u(t_0)  \mas T \searrow T_- \,  \textrm{weakly in} \,\, \dot{\HH}^{s_p}.
}
\end{lem}

Remark~\ref{r:Reta} indicates that solutions $\vec u(t)$ with the compactness property have uniformly small tails in $\dot{\HH}^{s_p}$, where ``tails'' are taken to be centered at $x(t)$, and relative to the frequency scale $N(t)$ at which the solutions are concentrating. We would like to use this fact to obtain lower bounds for norms of the solution $u(t)$. The immediate issue that arises is that the object that obeys compactness properties is the pair $\vec u(t, x) = (u(t, x), u_t(t, x))$ and, \emph{a priori}, the solution could satisfy $u(t, x) = 0$ a fixed time $t$. Nonetheless, by averaging in time, such a lower bound still holds for the solution itself, $u(t)$. We can quantify this bound in several ways, starting with a result proved in~\cite{KV11b}. 

\begin{lem}\emph{\cite[Lemma 3.4]{KV11b}}\label{l:mass}  Let $\vec u(t)$ be a solution with the compactness property on $I_{\max} = \R$. Then for any $A>0$, there exists $\eta = \eta(A)$ such that 
\EQ{
\abs{ \{ t \in [t_0, t_0+  \frac{A}{N(t_0)}] \mid \| u (t) \|_{L^{\frac{3}{2}(p-1)}_x} \ge \eta \} }  \ge \frac{\eta}{N(t_0)}
}
for all $t_0 \in \R$. 
\end{lem} 

Lemma \ref{l:mass} means that the $L^{\frac{3}{2}(p-1)}_x$ norm of $u(t)$ is nontrivial when averaged over intervals around $t_0$ of length comparable to $N(t_0)^{-1}$ uniformly in $t_0$.  By combining this lemma with Remark~\ref{r:Reta} and Sobolev embedding we obtain the following as an immediate consequence.  
\begin{cor}[Averaged concentration around $x(t)$]\label{C:acax}
Fix any $\de_0>0$.  Let $\vec u(t)$ be a solution with the compactness property on $I_{\max} = \R$. There exists a constant $C>0$ so that 
 \EQ{
N(t_0) \int_{t_0}^{t_0 + \frac{\de}{N(t_0)}}  \int_{\abs{ x- x(t)} \le \frac{C}{N(t)}} \abs{ u(t, x)}^{\frac{3}{2}(p-1)} \, \ud x \, \ud t \gtrsim 1
}
for all $t_0 \in \R$. 
\end{cor} 

One can also deduce the following corollary, also proved in~\cite{KV11b}, which gives an lower bound on the localized $S$ norm of $u(t)$.  

\begin{cor}[$S$-norm concentration around $x(t)$] \label{c:Sloc}   Let $\vec u(t)$ be a solution with the compactness property on $I_{\max} = \R$.
Then there exists constants $c, C>0$ so that 
\EQ{  \label{eq:Sloc}
\int_{t_1}^{t_2} \int_{ \abs{x - x(t)} \le \frac{C}{N(t)}} \abs{ u(t, x) }^{2(p-1)} \, \ud x \, \ud t  \ge c  \int_{t_1}^{t_2} N(t) \, \ud t 
}
for any $t_1, t_2$ such that $t_2 - t_1 \ge \frac{1}{N(t_1)}$. 
\end{cor} 

\begin{proof} 
The proof runs completely parallel to the argument in~\cite[Proof of Corollary 3.5]{KV11b} given for the averaged potential energy. 
\end{proof}

The fact that we have only averaged lower bounds on, e.g., the  $L^{\frac{3}{2}(p-1)}$ norm of a critical element will not be too much trouble. We will often pair the above with the fact that the compactness parameters $N(t), x(t)$ are approximately locally constant; see Lemma~\ref{l:const}. 

Lastly, we also need the following estimate proved in~\cite[Lemma 4.5]{DL1}.

\begin{lem}\emph{\cite[Lemma 4.5]{DL1}}  \label{l:DL1lem} Let $\vec u(t)$ have the compactness property on a time interval $I \subset \R$ with scaling parameter $N(t)$. Let $\eta>0$. Then there exists $\de>0$ such that 
\EQ{
 \| u \|_{L^{2(p-1)}_{t, x}( [t_0 -  \de/ N(t_0), \, t_0 + \de/ N(t_0)] \times \R^3)} \le \eta
}
uniformly in $t_0 \in I$. 
\end{lem}

\subsection{Analysis of solutions with the compactness property}\label{subsec:analysis_of_props}
In the next subsection, we will prove a classification result for solutions with the compactness property. Our goal is to gather together a list of possibilities for the compactness parameters $N(t)$ and $x(t)$ that is exhaustive in the sense that if we rule out the existence of all members of the list, then Theorem~\ref{t:main} is true. Before stating these cases, we need to distinguish between two scenarios based on how fast $x(t)$ is moving relative to the speed of light. To make this distinction precise, we have the following definition. 
\begin{defn} \label{d:sl} 
Let $\vec u(t)$ be a solution to~\eqref{eq:nlw} with the compactness property on $I = \R$ with parameters $x(t)$ and $N(t) \ge 1$.  We will say that $x(t)$ is \emph{subluminal} if there exists a constant $A>1$ so that for all $t_0 \in \R$ there exists $t \in [t_0, t_0 + \frac{A}{N(t_0)}]$ such that 
\EQ{
\abs{x(t)  - x(t_0)} \le \abs{ t - t_0} - \frac{1}{ A N(t_0)}.
}
\end{defn} 

\begin{prop}\label{p:cases}
Suppose $\vec u(t)$ is a solution to~\eqref{eq:nlw} with the compactness property on its maximal interval of existence $I_{\max}$ with compactness parameters $N(t)$ and $x(t)$.  We can assume without loss of generality in the arguments that follow that $I_{\max}$, $N(t)$ and $x(t)$ falls into one of the following four scenarios:
\begin{itemize}
\item[(I)]  \emph{The soliton-like critical element}: $I_{\max} = \R$,  $N(t)  \equiv 1$ for all $t \in\R$ and $x(t)$ is subluminal in the sense of Definition~\ref{d:sl}. 
\item[(II)] \emph{The two-sided concentrating critical element}: $I_{\max} = \R$,  $N(t) \ge 1$ for all $t \in \R$,  $\limsup_{t \to \pm \infty} N(t) = \infty$, and $x(t)$ is subluminal.  
\item[(III)] \emph{The self-similar-like critical element}: $I_{\max} = (0, \infty)$, $N(t) = \frac{1}{t}$, and $x(t)\equiv 0$.  
\item[(IV)] \emph{The traveling wave critical element}: $I_{\max} = \R$, $N(t) \equiv 1$ for all $t \in \R$ and $|x(t) - (t, 0, 0)| \lesssim \sqrt{\abs{t}}$ for all $t \in \R$.   
\end{itemize} 

\end{prop}

\begin{rem} \label{r:p=3b} 
In the case $p=3$, one must take into account the action of the Lorentz group, which will introduce additional cases to the list of critical elements in Proposition \ref{p:cases}. For $p \neq 3$, the hypothesis~\eqref{eq:fantasy} compactifies the action of the Lorentz group in the Bahouri-G\'erard profile decomposition at regularity ~$\dot \H^{s_p}$, which is why only a translations $x(t)$ and scaling $N(t)$ appear in the descriptions of critical elements. However, because $\dot \H^{\frac{1}{2}}$ is invariant under action of the Lorentz group one must confront critical elements with velocity $\ell(t)$ that approaches the speed of light. See the work of Ramos~\cite{Ramos1, Ramos2}, for Bahouri-G\'erard type profile decompositions in this setting. 
\end{rem} 

Before proving Proposition \ref{p:cases}, we note that ruling out Cases (I) - (IV) in the statement of the proposition will prove our main result, Theorem \ref{t:main}. Hence we will now focus on establishing Proposition \ref{p:cases} and proving that such critical elements cannot exist.

\medskip
We will prove this proposition in several steps. First, we will reduce the frequency parameter $N(t)$ to one of three possible cases. We state these reductions for $N(t)$, but we omit the proof as it follows readily from arguments similar to those in \cite[Theorem 5.25]{KVClay}.

\begin{prop}\label{p:3cases} 
Let $\vec u(t)$ denote the critical element found in Proposition~\ref{p:ce}. Passing to subsequences, taking limits, using scaling considerations and time reversal, we can assume, without loss of generality, that $T_+(\vec u) = + \infty$, and that the frequency scale $N(t)$ and maximal interval of existence $I_{\max} =I_{\max}(\vec u)$ satisfy one of the following three possibilities:  
\begin{itemize}
\item \emph{(Soliton-like scale)} $I_{\max} = \R$ and 
\ant{
N(t) \equiv 1 \quad \, \forall \, t \in \R.
}
\item \emph{(Doubly concentrating scale)} $I_{\max} = (-\infty, \infty)$ and  
\ant{
 \limsup_{t \to T_-} N(t) =\infty, \quad    \limsup_{t \to  \infty} N(t)  = \infty, \mand N(t) \ge 1 \quad \, \forall t \in \R.
 }  
\item \emph{(Self-similar scale)} $I_{\max} = (0, \infty)$ and $N(t) = t^{-1}$.
\end{itemize}
 
\end{prop}

We will now make a few further reductions, mostly concerning the spatial center $x(t)$ of a critical element that is global in time. 

We will show that in all cases where we have a solution with the compactness property with translation parameter $x(t)$ that fails to be subluminal, we may extract a traveling wave solution. To prove this, we will need to analyze the properties of solutions with the compactness property and more specifically, properties of their spatial centers, $x(t)$. We turn to this analysis now. First, we note that in the case that $x(t)$ is subluminal (see Definition~\ref{d:sl}) we can derive the following consequence. 

\begin{lem}\emph{\cite[Proposition 4.3]{KV11b}}  \label{l:sl} 
Let $\vec u(t)$ be a solution to~\eqref{eq:nlw} with the compactness property on $I = \R$ with parameters $x(t)$ and $N(t) \ge 1$. Suppose $x(0)=0$ and that $x(t)$ is subluminal in the sense of Definition~\ref{d:sl}. Then there exists a $\de_0>0$ so that 
\EQ{
\abs{x(t) - x(\tau)} \le (1- \de_0) \abs{t - \tau}
}
for all $t, \tau$ with 
\EQ{
\abs{t- \tau} \ge \frac{1}{\de_0 N_{t, \tau}},
} 
where  $N_{t, \tau} := \inf_{s \in [t, \tau]} N(s) $. 
\end{lem} 
\begin{proof} See the proof of Proposition 4.3 in \cite{KV11b}.
\end{proof} 
Using Lemma~\ref{l:const} together with Lemma~\ref{l:mass} and a  domain of dependence argument based on the finite speed of propagation, we obtain a preliminary bound on how fast  $x(t)$ can grow.  (See e.g. \cite[Proposition~4.1]{KV11b}.)
\begin{lem} 
Let $\vec u(t)$ have the compactness property on a time interval $I \subset \R$ with parameters $N(t)$ and $x(t)$. Then there exists a constant $C>0$ so that for any $t_1, t_2 \in I$ we have 
\EQ{ \label{eq:dxt} 
\abs{x(t_1) - x(t_2)} \le \abs{t_1 - t_2} + \frac{C}{N(t_1)} + \frac{C}{N(t_2)}.
}
In fact, if $\vec u(t)$ is global in time, we have $N(t) \abs{t} \to \infty$ as $\abs{t } \to \infty$ and we normalize so that $x(0) = 0$, from which the above yields 
\EQ{ \label{eq:fsp2} 
\lim_{t \to \pm \infty} \frac{\abs{ x(t)}}{\abs{t}} \le 1.
}
\end{lem}

\begin{rem} 
We  remark that by finite speed of propagation and compactness, we can assume that 
\EQ{
\lim_{t \to T_{\pm}(\vec u)} \abs{t} N(t) \in [1, \infty) .
}
\end{rem}

Note that according to the definition of the compactness property, the function $x(t)$ is not uniquely defined;  indeed, one can always modify $x(t)$ up to a radius of $\mathcal{O}(N(t)^{-1})$, provided one also modifies the compactness modulus appropriately.  
Note, however, that the compactness property, together with monotone convergence, prevents $\vec{u}$ from concentrating on very narrow strips, as measured in units of $N(t)^{-1}$. See \cite[Lemma~4.2]{KV11b}.
\begin{lem}\label{L:strips} Let $\vec u$ be an solution to \eqref{eq:nlw} with the compactness property on an interval $I$.  Then for any $\eta>0$, there exists $c(\eta)>0$ so that
\[
\sup_{\omega\in\mathbb{S}^2} \int_{|\omega\cdot[x-x(t)]|\leq c(\eta)N(t)^{-1}} |\nsp u|^2 + |\nspt u_t|^2\,\ud x <\eta. 
\]
\end{lem}

To deal with ambiguity in the definition of $x(t)$, we use the notion of a `centered' spatial center as in \cite{KV11b}, that is, a choice of $x(t)$ such that each plane through $x(t)$ partitions $\vec u(t)$ into two non-trivial pieces. 

\begin{defn}\label{D:centered} Let $\vec u$ be a solution to \eqref{eq:nlw} with the compactness property on an interval $I$ with spatial center $x(t)$. We call $x(t)$ \emph{centered} if there exists $C(u)>0$ such that for all $\omega\in\mathbb{S}^2$ and $t\in I$, 
\[
\int_{\omega\cdot[x-x(t)]>0} |\nsp u(t,x)|^2 + |\nspt u_t(t,x)|^2 \,\ud x \geq C(u).  
\]
\end{defn}

\begin{prop}\label{P:centered} Let $\vec u$ be a global solution to \eqref{eq:nlw} with the compactness property.  Then there exists a centered spatial center for $\vec u$.
\end{prop}

\begin{proof} The argument is similar to the proof of \cite[Proposition~4.1]{KV11b}. Let $x(t)$ be any spatial center for $\vec{u}$.  To shorten formulas, we introduce the notation
\[
\varphi(t,x)=|\nsp u(t,x)|^2 + |\nspt u_t(t,x)|^2.
\]
By compactness, there exists $C=C(u)$ large enough that 
\[
\inf_{t\in\R}\int_{B(t)} \varphi(t,x)\,\ud x \gtrsim_u 1,\qtq{where}B(t):=\{x:|x-x(t)|\leq CN(t)^{-1}\}. 
\]
Now set 
\[
\tilde x(t) = x(t) + \frac{\int_{B(t)}[x-x(t)]\varphi(t,x)\,\ud x}{\int_{B(t)} \varphi(t,x)\,\ud x}. 
\]

By definition, $|x(t)-\tilde x(t)| \leq CN(t)^{-1}$, and hence $\tilde x(t)$ is a valid spatial center for $\vec u$ (one only needs to add $C$ to the compactness modulus).  We now claim that $\tilde x(t)$ is centered. To see this, first note that by construction one has
\[
\int_{B(t)} \omega\cdot[x-\tilde x(t)] \varphi(t,x)\,\ud x = 0.
\]
On the other hand, combining non-triviality on $B(t)$ together with Lemma~\ref{L:strips}, we have
\[
\int_{B(t)\cap |\omega\cdot[x-\tilde x(t)]|>cN(t)^{-1}} \varphi(t,x)\,\ud x\gtrsim_u 1
\]
for some $c=c(u)>0$. Thus
\[
\int_{B(t)}|\omega\cdot[x-\tilde x(t)]|\varphi(t,x)\,\ud x \gtrsim_u N(t)^{-1},
\]
and so
\[
\int_{B(t)} \{\omega\cdot[x-\tilde x(t)]\}_+ \varphi(t,x)\,\ud x \gtrsim_u N(t)^{-1},
\]
where $+$ denotes the positive part.  As $|x-\tilde x(t)|\leq 2CN(t)^{-1}$ for $x\in B(t)$, we finally deduce
\[
1\lesssim_u \int_{B(t)} \frac{\{\omega\cdot[x-\tilde x(t)]\}_+}{2CN(t)^{-1}} \varphi(t,x)\,\ud x\lesssim_u \int_{\omega\cdot[x-\tilde x(t)]>0}\varphi(t,x)\,\ud x 
\]
for all $\omega\in\mathbb{S}^2$, as needed.\end{proof}

\begin{prop}\label{p:hans} Suppose that $\vec u(t)$ is a solution with the compactness property on $\R$ with parameters $N(t)$ and $x(t)$. Suppose in addition that $N(t) = 1$ for all $t \in \R$, and that $x(t)$ fails to be subluminal in the sense of Definition~\ref{d:sl}. Then there exists a (possibly different) solution $\vec w(s)$ to~\eqref{eq:nlw} with the compactness property on $\R$ with parameters $N(s)$ and $x(s)$ satisfying
\EQ{
N(s) \equiv 1, \quad  |x(s) - (s, 0, 0)| \lesssim \sqrt{{\abs{s}}} \quad  \forall s \in \R. 
}
\end{prop} 

\begin{proof}
Let $\vec u(t)$ be a solution to~\eqref{eq:nlw} with the compactness property on $\R$ with parameters $N(t)  \equiv 1$ and $x(t)$ failing to be subluminal. This means we can find a sequence $t_m$ and intervals 
\ant{
I_m:= [t_m, t_m + m]
}
such that 
\EQ{\label{eq:nsl1} 
\abs{x(t_m) - x(t)} \ge \abs{t_m - t} - \frac{1}{m} \quad \forall t \in I_m.
}
We construct a sequence as follows. Set 
\EQ{
\vec u_m(0) :=  \vec u(t_m,  \cdot - x(t_m)).
}
Using the pre-compactness of the trajectory of $\vec u$ modulo the translations by $x(t)$ we can (passing to a subsequence) extract a strong limit 
\EQ{
\vec u_m(0) \to \vec u_\I(0) \in \dot \HH^{s_p} \mas m \to \infty.
}
Let $\vec u_\infty(\tau)$ be the solution to~\eqref{eq:nlw} with initial data $\vec u_{\infty}(0)$. One can show that we must have $[0, \infty) \subset I_{\max}(\vec u_{\infty})$ and that $\vec u_{\infty}$ satisfies the following compactness property on $[0, \infty)$: the set 
\EQ{
 K_\infty  := \{ \vec u_\I(  \tau,  \cdot - x_\infty(\tau)) \mid \tau \in [0, \infty)\}
 }
 is pre-compact in $\dot \HH^{s_p}(\R^3)$, where for each $\tau>0$ the function $x_\infty(\tau) $ is defined by 
 \EQ{
 x_\infty(\tau):= \lim_{m \to \infty} ( x(t_m + \tau) - x( t_m)).
 }
 Note that  for each $\eps>0$ and for all  $\tau \in [0, \infty)$ we can choose $M>0$ large enough so that for all $m  \ge M$ we have 
 \ant{
 \abs{ x_\I (\tau)} + \eps \ge \abs{ x(t_m + \tau) - x( t_m)} \ge  \abs{\tau} - \frac{1}{m},
 }
 where the last inequality follows from~\eqref{eq:nsl1}.  Letting $m \to \infty$ above, we conclude that in fact 
 \EQ{
 \abs{x_\I(\tau)} \ge  \tau \quad \forall  \tau \in [0, \infty),
 }
 By finite speed of propagation (see~\eqref{eq:fsp2}) we can conclude that in fact
 \ant{
 \lim_{ \tau \to \infty} \frac{ \abs{x_\I(\tau)}}{\tau}    = 1,
 }
 We now refine our solution again, this time constructing a suitable limit from $\vec u_{\infty}(\tau)$. Indeed, using the previous two lines, choose a sequence $\tau_m \to \infty$ so that
 \EQ{
 \tau \le  \abs{x_\I(\tau)} \le \tau + 2^{-m} \quad \forall \tau \in J_m:= [\tau_m- m, \tau_m +m],
 }
 It is clear from elementary geometry that for large enough $m$, 
 \EQ{\label{eq:dxtau}
 \abs{ x_\I(\tau) - x_\I(t)} \ge \abs{ t - \tau} - \frac{1}{m} \quad \forall t, \tau \in J_m,
 }
 As before we extract a limit from the sequence 
 \ant{
 \vec u_{\I,m}(0) :=  \vec u_\I(\tau_m,  \cdot - x_\I(\tau_m)) \to \vec v(0) \in \HH^{s_p}
 }
 and we note that the solution $\vec v(s)$ to~\eqref{eq:nlw} with data $\vec v(0)$ has the compactness property on $\R$ with parameters $\ti N(s) \equiv 1$ and $\ti x(s)$ defined by 
 \ant{
 \ti x(s):= \lim_{m \to \infty} ( x_{\infty}(\tau_m + s) - x_{\I}(\tau_m))
 }
 Using~\eqref{eq:dxtau} along with~\eqref{eq:dxt}  we see that for all $s_1, s_2 \in \R$ we have 
 \EQ{ \label{eq:tiC} 
 \abs{s_1 - s_2} \le \abs{ \ti x(s_1) - \ti x(s_2)} \le \abs{s_1 -s_2} + \ti C
 }
 for some absolute constant $\ti C>0$ and moreover we still have that 
 \EQ{ \label{xbigs} 
 1 \le \frac{\abs{\ti x(s)}}{\abs{s}} \to 1 \mas s \to \pm \infty,
 }

 Now we express $\ti x(s)$ in polar coordinates, finding $r(s) \ge 0$ and $\om(s) \in \Sp^2$ so that 
 \EQ{
 \ti x(s) = r(s) \om (s) \quad \forall s \in [0, \infty).
 }
Note that by~\eqref{xbigs} we have 
 \EQ{
 \frac{r(s)}{s} \to 1 \mas s \to \infty.
 }
 Since $\om(s) \in \Sp^2$ we can find a sequence $s_m \to \infty$ and we can (up to passing to a subsequence) find a limit $\om_0$ so that 
 \EQ{
 \om(s_m) \to  \om_0 \mas m \to \infty.
 }
To prove the claim, it suffices to verify that
\[
|\tilde{x}(s) - s\omega_0| \leq C \sqrt{s},
\]
since then we obtain the desired result applying a fixed spatial rotation. Note that
\[
|s_2\omega(s_2)-s_1\omega(s_1)|^2 = |s_1-s_2|^2 + s_1 s_2 |\omega(s_2)-\omega(s_1)|^2,
\]
which by finite speed of propagation yields
\[
|\omega(s_2)-\omega(s_1)| \leq \sqrt{\frac{2C|s_1-s_2|+C^2}{s_1s_2}}. 
\] 
Then
\begin{align*}
|&(s_n+s)\omega(s_n+s)-s_n\omega(s_n)-s\omega_0| \\
&\leq |s_n+s|\bigl| \omega(s_n+s)-\omega(s_n)\bigr| + s|\omega(s_n)-\omega_0| \\
& \leq \sqrt{(2Cs+C^2)(1+\tfrac{s}{s_n})} + s|\omega(s_n)-\omega_0|, 
\end{align*} 
which implies
\[
|\tilde{x}(s)-s\omega_0| \leq \sqrt{2Cs+C^2},
\]
as required.

  \end{proof}

In the case that $N(t)  \ge 1$  and $x(t)$ is not subluminal, we will now show that we can also reduce to the case when $N(t) = 1$ for all $t \in \R$ and $x(t) = (t, 0, 0) + O(\sqrt{{\abs{t}}})$. We will need the following lemma.
\begin{lem}\label{l:mon} 
Let $\vec u(t)$ have the compactness property on $I \subset \R$ with parameters $N(t)$ and $x(t)$. Let $t_1, t_2 \in I$ be any times so that $N(t_1) \le  N(t_2)$. Then there exists a uniform constant $c \in (0, 1)$ such that 
\EQ{ \label{eq:mon} 
 \abs{x(t_1) - x(t_2)} \ge \abs{ t_1 - t_2} - \frac{c}{ N(t_1)}  \Longrightarrow N(t_2) \le \frac{1}{c^2} N(t_1)
}
\end{lem} 
\begin{proof}[Proof of Lemma~\ref{l:mon}]
The argument adapts readily from \cite[Lemma~4.4]{KV11b}, using the arguments from \S\ref{subsec:analysis_of_props}. Exploiting time-reversal symmetry, space-translation symmetry, and rotation symmetry, we may assume $t_1<t_2$, $x(t_1)=0$, and $x(t_2)=(x_1(t_2),0,0)$ with $x_1(t_2)\geq 0$.  Further, we may choose $x(t)$ to be centered by Proposition \ref{P:centered}.

Suppose for contradiction that for times $t_1, t_2$ as in the statement of the lemma,
\[
 \abs{x(t_1) - x(t_2)} \ge \abs{ t_1 - t_2} - \frac{c}{ N(t_1)} 
 \]
but $cN(t_1)^{-1} \geq c^{-1} N(t_2)^{-1}$, where $c=c(u)$ will be chosen sufficiently small below.   

Let $\psi:\R\to[0,\infty)$ be a cutoff so that $\psi=1$ for $x\leq -1$ and $\psi =0$ for $x\geq -\frac12$.  Set $\psi_2(x_1)=\psi(\frac{x_1-x_1(t_2)}{cN(t_1)^{-1}})$.  Then, given $\eta>0$ and choosing $c=c(\eta)$ sufficiently small, we have
\[
\|(\psi_2u(t_2),\psi_2u_t(t_2))\|_{\H^{s_p}} < \eta. 
\]
Choosing $\eta$ small enough, the small-data theory and finite speed of propagation for \eqref{eq:nlw} imply 
\[
\int_{\Omega} |\nsp u(t_1,x)|^2 + |\nspt u_t(t_1,x)|^2 \,\ud x \lesssim \eta^2,
\]
where
\[
\Omega = \{x:x_1 \leq x_1(t_2)-(t_2-t_1)-cN(t_1)^{-1}\}. 
\]
Using the assumption on $|x(t_2)-x(t_1)|$ and the normalizations above, one finds
\[
\Omega\supset\{x:-e_1\cdot[x-x(t_1)]\geq 2cN(t_1)^{-1}\}, 
\]
so that
\[
\int_{-e_1\cdot[x-x(t_1)]\geq 2cN(t_1)^{-1}} |\nsp u(t_1,x)|^2 + |\nspt u_t(t_1,x)|^2 \,\ud x \lesssim \eta^2. 
\]
On the other hand, choosing $c=c(\eta)$ sufficiently small, Lemma~\ref{L:strips} implies
\[
\int_{0<-e_1\cdot[x-x(t_1)]< 2cN(t_1)^{-1}} |\nsp u(t_1,x)|^2 + |\nspt u_t(t_1,x)|^2 \,\ud x < \eta^2.
\]
We now choose $\eta^2\ll C(u)$, where $C(u)$ is as in Definition~\ref{D:centered}, to reach a contradiction to Proposition \ref{P:centered}.\end{proof}

We are now in a position to prove that we can extract a traveling wave solution from any solution with compactness property  with translation parameter $x(t)$ that fails to be subluminal.
\begin{prop} \label{p:nsl}
 Suppose that $\vec u(t)$ is a solution with the compactness property on $\R$ with parameters $N(t)$ and $x(t)$. Suppose that either $N(t)$ is soliton-like or doubly concentrating in the sense of Proposition~\ref{p:3cases}  and that $x(t)$ fails to be subluminal in the sense of Definition~\ref{d:sl}. Then there exists a (possibly different) solution $\vec w(s)$ to~\eqref{eq:nlw} with the compactness property 
on $\R$ with parameters $N(s)$ and $x(s)$ satisfying
\EQ{
N(s) \equiv 1, \quad  |x(s) - (s, 0, 0)| \lesssim \sqrt{s} \quad  \forall s \in \R. 
}
\end{prop}

\begin{proof}[Proof of Proposition~\ref{p:nsl}]
Note that by Proposition~\ref{p:hans} it suffices to show that we can extract a solution with the compactness property on $\R$ with parameters $N(t)  =1$ and $x(t)$ failing to be subluminal.  By our assumption that $x(t)$ fails to be subluminal, for each $m \in \N$ there exists $t_m \in \R$ so that 
\EQ{\label{eq:nsl2} 
 \abs{ x(t_m) - x(t)}  \ge \abs{t - t_m} - \frac{1}{m N(t_m)} \quad  \forall t \in I_m:= [t_m, \, t_m + \frac{m}{N(t_m)}].
 }
 We will show that $N(t) \simeq N(t_m)$ for all $t \in I_m$ with constants independent of $m$. First assume that $$N(t_m) \le N(t).$$ Then 
 by Lemma~\ref{l:mon} we can  find a constant $c>0$ so that 
 \EQ{\label{eq:ub2} 
 c^2 N(t) \le  N(t_m) \le  N(t) \quad \forall t \in I_m.
 }
 Next assume that 
\ant{
 N(t) \le N(t_m).
 }
 This means that 
 \ant{
 -\frac{1}{N(t_m)} \ge - \frac{1}{N(t)}
 }
 and thus from~\eqref{eq:nsl2} we see that 
 \EQ{
  \abs{ x(t_m) - x(t)}  \ge \abs{t - t_m} - \frac{1}{m N(t_m)}  \ge \abs{t - t_m} -\frac{1}{M N(t)}.
 }
 Another application of Lemma~\ref{l:mon} then gives
 \EQ{
 N(t) \le N(t_m) \le \frac{1}{c^2} N(t).
 }
 As we can assume in Lemma~\ref{l:mon} that $c<1$ we deduce that 
 \EQ{\label{eq:cnc} 
 c^2 N(t)  \le N(t_m) \le \frac{1}{c^2} N(t) \quad \forall t \in I_m.
 }
We can then extract, in the usual manner a new solution $\vec w(s)$ with the compactness property on $[0, \infty)$ with 
\EQ{
&\ti N(s):= \lim_{m \to \infty} \frac{N( t_m + \frac{s}{N(t_m)})}{N(t_m)}, \\ 
&\ti x(s) := \lim_{m \to \infty} N(t_m) ( x(t_m + \frac{s}{N(t_m)}) - x(t_m)). 
}
Note that by~\eqref{eq:cnc} we must have 
\EQ{
c_1 \le \ti N(s) \le C_1  \quad \forall \, s \in[0, \infty).
}
 Moreover, using~\eqref{eq:nsl2}, for each $\eps>0$ we can find $M>0$ large enough so that for each $m \ge M$ we have
\ant{
\abs{ \ti x(s)}  + \eps  &\ge \abs{N(t_m) ( x(t_m + \frac{s}{N(t_m)}) - x(t_m)) }\\
& \ge N(t_m) \abs{ \frac{s}{N(t_m)} - \frac{1}{m N(t_m))}}  \ge  s - \frac{1}{m}.
}
Letting $m \to \infty$ we obtain
\EQ{
\abs{\ti x(s)} \ge s \quad  \forall s \in [0, \infty).
}
Noting that $\ti x(0) = 0$ and combining the above with~\eqref{eq:fsp2} we conclude that 
\EQ{
1 \le \frac{ \abs{\ti x(s)}}{s}  \to 1 \mas s \to \infty.
}
From here it is straightforward to obtain a new solution $\vec w(s)$ with the compactness property on all of $\R$ with parameters $N(s) \equiv 1$ and $x(s)$ failing to be subluminal in the sense of Definition~\ref{d:sl}, and we apply Proposition \ref{p:hans} to conclude. 
\end{proof} 

Finally, we now have the ingredients necessary to prove Proposition~\ref{p:cases}.

\begin{proof}[Proof of Proposition~\ref{p:cases}]
Suppose $\vec u(t)$ is a solution to~\eqref{eq:nlw} with the compactness property on its maximal interval of existence $I_{\max}$ with compactness parameters $N(t)$ and $x(t)$. By Proposition~\ref{p:3cases}, if the solution has the compactness property with $N(t) = t^{-1}$,  then we may also assume without loss of generality that it has the compactness property with translation parameter $x(t) = 0$: by finite speed of propagation, $x(t)$ must remain bounded, and hence we may, up to passing to a subsequence, obtain a pre-compact solution with $x(t) = 0$ by applying a fixed translation. Thus, in the case that $N(t) = t^{-1}$ we obtain a self-similar solution, i.e. we have reduced to case (III). 

In the remaining cases we must address different scenarios depending on whether or not $x(t)$ is subluminal in the sense of Definition~\ref{d:sl}. If $x(t)$ is subluminal, then we have reduced ourselves to cases (I) and (II). If $x(t)$ fails to be subluminal, then by Proposition~\ref{p:nsl} we can find a critical element as in the traveling wave scenario, i.e. case (IV). This concludes the proof.
\end{proof}

\section{The soliton-like critical element} \label{s:soliton} 

In this section we show that the soliton-like critical element, that is, case (I) from Proposition~\ref{p:cases}, cannot exist. The main result is the following proposition: 
 

 \begin{prop}\label{prop:soliton_zero}
There are no soliton-like critical elements for~\eqref{eq:nlw}, in the sense of Case (I) of Proposition \ref{p:cases}.
 \end{prop}
 We recall that soliton-like means that $\vec u(t)$ is a global solution to~\eqref{eq:nlw} with the compactness property on $\R$ as defined in~\eqref{d:cp} with parameters $N(t)  \equiv 1$, and $x(t)$ subluminal in the sense of Definition~\ref{d:sl}. We will show that any such solution with the compactness property is necessarily $\equiv 0$.
 
 The proof will be accomplished in two main steps. We are ultimately aiming to employ a rigidity argument based on a virial identity, which will show that any such critical element must then be identically 0. The key point here is that in order to access the virial identity, which is at $\dot \HH^1$ regularity, and to use it to prove Proposition \ref{prop:soliton_zero}, we first must prove that our critical element actually lies in a precompact subset of ~$\dot \HH^1$. Thus, we must first show that a soliton-like critical element must be more regular than expected. In fact, we will prove that the trajectory $K$ of any soliton-like critical element (see Definition \ref{d:cp})  must be pre-compact in $\dot \HH^1 \cap \dot \HH^{s_p}$.

Throughout this section, we assume towards a contradiction that $\vec u(t)$ is a critical element with $x(t)$ subluminal in the sense of Definition~\ref{d:sl} and $N(t)\equiv 1$.  In particular, by Lemma \ref{l:sl}  there exists $\delta_0>0$ so that
\[
|x(t)-x(\tau)|<(1-\delta_0)|t-\tau|\qtq{for all}|t-\tau|>\frac1{\delta_0}. 
\]

\subsection{Additional regularity} \label{s:soliton-reg}  We first prove that if the soliton-like critical element $\vec u$ has some additional regularity to begin with, then we can achieve $\dot \H^1$ regularity.  The key ingredient in our proof will be a double Duhamel argument, which will enable us to gain the requisite regularity for critical elements, while our main technical tool will be the use of a frequency envelope which controls the $\dot \H^1$ norm (see Definition~\ref{d:fren}). In order to exploit the sharp Huygens principle, we will use the following modified frequency projection operators:  let $\psi\geq 0$ be a smooth function supported on $|x|\leq 2$ satisfying $\psi=1$ on $|x|\leq 1$.  For $k\geq 0$, let
\begin{align}\label{qk_def}
Q_{<k} f(x) = \int_{\R^3} 2^{3k} \psi(2^k(x-y))f(y)\,dy. 
\end{align}
These satisfy the same estimates as the usual Littlewood--Paley projections (which instead use sharp cutoffs in frequency space), e.g. the Bernstein estimates in Lemma~\ref{l:bern}. 

We summarize the main ingredient in Proposition \ref{prop:soliton_zero}, the aforementioned additional regularity result, in the following proposition.

\begin{prop} \label{p:sol-reg} 
Suppose $\vec u$ is a soliton-like critical element. Then
\[
\vec u\in L_t^\infty\dot\H^{s_p} \quad \Longrightarrow \quad \vec u\in L_t^\infty  \dot\H^{s}. 
\]
for some $s > 1$. In particular, the set 
\EQ{
K := \{  \vec u(t, \cdot - x(t)) \mid t \in \R\}  \subset \dot{\HH^{s_p}} \cap \dot \HH^{1} 
}
is pre-compact in $\dot{\HH^{s_p}} \cap \dot \HH^{1} $. 
\end{prop} 

We will prove Proposition \ref{p:sol-reg} in several steps. To make this precise, we define the parameter
\EQ{ \label{eq:s0def} 
s_0=s_p+\frac{5-p}{2p(p-1)}=\frac32-\frac5{2p}.
}
This exponent is chosen so that $\dot \H^{s_0}$ has the same scaling as $L_t^p L_x^{2p}$, and we note that crucially, $s_p < s_0 < 1$.  

\subsection{The jump from \texorpdfstring{$\dot \cH^{s_0}(\bR^3)$}{H} regularity to \texorpdfstring{$\dot \cH^{1}(\bR^3)$}{H} regularity}  We begin with the first, easier gain in regularity, namely passing from $\dot \H^{s_0}(\R^3)$ to $\dot \H^1(R^3)$.

\begin{prop}\label{P:reg-jump} Suppose $\vec u$ is a soliton-like critical element.  Let $s_0>s_p$ be defined as in~\eqref{eq:s0def}. Then
\[
\vec u\in L_t^\infty\dot\H^{s_0} \quad \Longrightarrow \quad \vec u\in L_t^\infty  \dot\H^{1}. 
\]
\end{prop}

\begin{proof}  
 By time-translation symmetry, it suffices to estimate the $\dot \H^{1}$-norm at time $t=0$.  We complexify the solution, letting 
 \[
 w = u+ \frac{i}{\sqrt{-\Delta}} u_t.
 \]
 Then
 \[
 \|w(t) \|_{\dot H^1} \simeq \|\vec u(t) \|_{\dot H^1 \times L^2},
 \]
 and if $\vec u(t)$ solves \eqref{eq:nlw}, then $w(t)$ is a solution to
 \begin{align}
 w_t = -i \sqrt{-\Delta} w \pm \frac{i}{\sqrt{-\Delta}} |u|^{p-1} u.
 \end{align}
 By Duhamel's principle, for any $T$, we have
 \[
 w(0) = e^{iT \sqrt{-\Delta}} w(T) \pm \frac{i}{\sqrt{-\Delta}} \int_T^0 e^{i \tau \sqrt{-\Delta}} F(u)(\tau) d\tau
 \]
 where $F(u) = |u|^{p-1} u$.  By compactness (see Lemma~\ref{l:weak}),
\begin{equation}\label{weakH10}
\lim_{T\to\infty} Q_{<k} e^{-i T\sqrt{-\Delta}} w (T) = \lim_{T\to\infty} Q_{<k} e^{iT \sqrt{-\Delta} } w(-T) = 0
\end{equation}
as weak limits in $\dot H^{1}$ for any $k\geq 0$.  We next write
\begin{align*}
Q_{<k} w (0) & = e^{-iT\sqrt{-\Delta} }Q_{<k} w (T)  {\mp \frac{1}{\sqrt{-\De}} }\int_0^T e^{-it\sqrt{-\Delta} }Q_{<k} F (u(t))\,\ud t \\
& = e^{iT\sqrt{-\Delta} }Q_{<k} w (-T)  {\mp \frac{1}{\sqrt{-\De}} } \int_{-T}^0 e^{-it\sqrt{-\Delta} }Q_{<k} F(u(t))\,\ud t.
\end{align*}
Using \eqref{weakH10}, and arguing as in~{\cite[Section 4]{DL1}} we can deduce 
\begin{align}
&\langle Q_{<k} w(0),Q_{<k} w (0)\rangle_{\dot H^{1}} \nonumber  \\
& = \lim_{T\to\infty}  \big\langle\int_0^T e^{-it\sqrt{-\Delta}}Q_{<k} F(u(t))\,\ud t,  \int_{-T}^0 \hspace{-2mm} e^{-i\tau \sqrt{-\Delta} }Q_{<k} F(u(s))\,\ud t\big\rangle_{L^{2}}. \label{scdd1}
\end{align}

We fix a large parameter $R>0$ to be determined below.  Let $\de_0$ be as in the statement of Lemma \ref{l:sl} and take $T=2 R\delta_0^{-1}$.  We define
\begin{equation}\label{eq:regions}
\begin{split}
\textup{Region }\textbf{A} &:= \{(t,x) \,:\, 0 \leq t \leq T\}, \\  
\textup{Region }\textbf{B} &:=\{(t,x) \,:\, |x- x(T)| \geq R + |t - T|\},\\
\textup{Region }\textbf{C} &:= \{(t,x) \,:\, |x- x(T)| < R + |t - T|\}.
\end{split}
\end{equation}
See Figure~\ref{f:bowtie}.

\begin{figure}[h] \label{f:bowtie} 
  \centering
  \includegraphics[width=14cm]{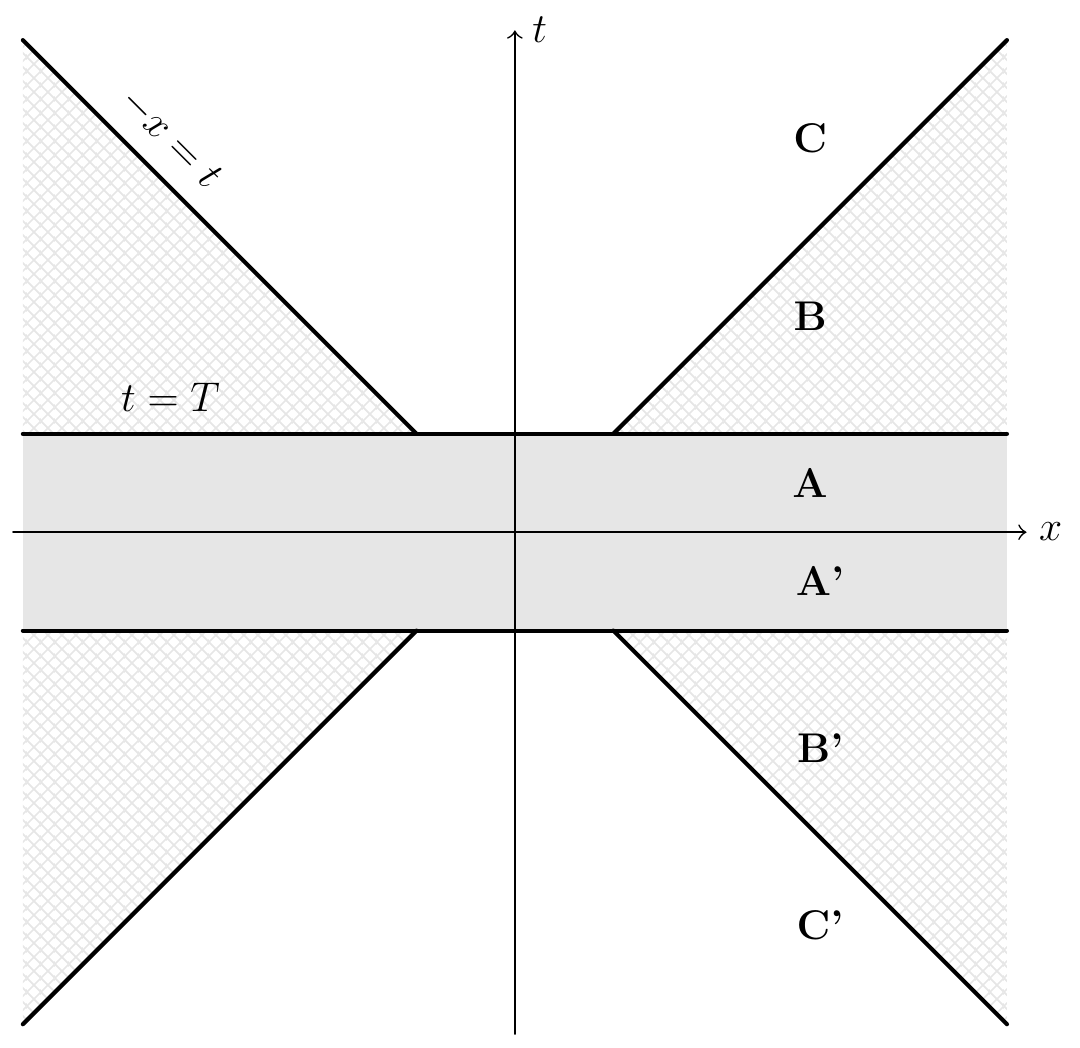}
  \caption{A depiction of the spacetime regions \textbf{A, A', B, B'} and \textbf{C, C'} in the case that $x(t) = 0$.} \label{fig:regions_soliton}
\end{figure}

We will treat these regions separately. Our goal is to bound $u$ on region \textbf{A} using that we are estimating the solution on a compact time interval, and on region \textbf{B} using the small data theory and finite speed of propagation. We will then use the double Duhamel trick, together with the sharp Huygens principle on region \textbf{C}, to conclude the proof.

Let $\chi_R$ denote a smooth cutoff to the set 
\[
\{|x-x(T)| > R\} \subseteq \mathbb{R}^3.
\]
Now fix a small parameter $\eta>0$.  By compactness of $\vec u$, if $R = R(\eta)$ is sufficiently large then we have
\begin{align}\label{equ:small_data_soliton}
\| \chi_R \vec u(T)\|_{\H^{s_p}} \le \eta. 
\end{align}
We let $\vec v = (v,v_t)$ be the solution to \eqref{eq:nlw} with initial data 
\[
\vec v(T)=\chi_R \vec u (T).
\]
By finite speed of propagation, we have that
\[
u \equiv v \qtq{for} |x-x(T)|\geq R+|t-T|,
\]
We now rewrite \eqref{scdd1}, and abusing notation slightly, we define
\begin{equation}\label{scdd-decomp1}
\begin{split}
&\int_0^\infty e^{-it\sqrt{-\Delta}}Q_{<k} F(u(t))\,\ud t = A+B+C, \\
&A=\int_0^T e^{-it\sqrt{-\Delta} } Q_{<k} F(u(t))\,\ud t, \\
&B=\int_T^\infty e^{-it \sqrt{-\Delta} } Q_{<k} F(v(t))\,\ud t, \\
&C = \int_T^\infty e^{-it\sqrt{-\Delta}} Q_{<k}[ F(u(t))- F(v(t))]\,\ud t.
\end{split}
\end{equation}
Note that the notation in \eqref{scdd-decomp1} is such that each term relates to an estimate for the solution on the correspondingly named region from \eqref{eq:regions}. 

We can carry out a similar construction at time $-T$, yielding a small solution $\widetilde v$ that agrees with $u$ whenever $|x-x(-T)|\geq R+|t+T|$, and we obtain three terms in the negative time direction
\begin{equation}\label{scdd-decomp2}
\begin{aligned}
&\int_{-\infty}^0 e^{-i \tau \sqrt{-\Delta}}Q_{<k}  F(u(\tau))\,d\tau = A'+B'+C', \\
&A'=\int_{-T}^0 e^{-i\tau\sqrt{-\Delta}} Q_{<k} F(u(\tau))\,d\tau, \\
& B'=\int_{-\infty}^{-T} e^{-i\tau\sqrt{-\Delta}} Q_{<k} F(\widetilde v(\tau))\,d\tau, \\
&C' = \int_{-\infty}^{-T} e^{-i\tau\sqrt{-\Delta}} Q_{<k}[ F(u(\tau))-F(\widetilde v(\tau))]\,d\tau. 
\end{aligned}
\end{equation}
Using the elementary linear algebra estimate
\begin{equation}\label{linal}
|\langle A+B+C,A'+B'+C'\rangle| \lesssim |A|^2 + |A'|^2 + |B|^2 + |B'|^2 + |\langle C,C'\rangle|
\end{equation}
whenever $A+B+C=A'+B'+C'$, we may estimate
\[
 \langle Q_{<k} w(0),Q_{<k} w(0)\rangle_{\dot H^{1}}
\]
by obtaining bounds for $A, A'$ and $B, B'$ and $\langle C, C' \rangle$.

\subsection*{Region A} To estimate the $A$ and $A'$ terms, first we establish the bound
\begin{align}\label{equ:a_est1}
\|u\|_{L_{t,x}^{2(p-1)}([-T,T]\times\R^3)}\lesssim \left(\frac{T}{\eps} \right)^{\frac{1}{2(p-1)}}
\end{align}
for some suitably small $\eps>0$.  To prove this, we rely on the fact that $\vec u$ is a soliton-like critical element.  Fix $\eta >0$. Since $N(t)= 1$, there exists $\eps >0$ small enough that the $L_{t,x}^{2(p-1)}$-norm is bounded by $\eta$ on any interval of length $\eps$; see Lemma~\ref{l:DL1lem}.   Thus to obtain the desired bound, we divide $[-T,T]$ into $\sim \lceil T / \eps \rceil$ intervals $J_k$ of length $\eps$,  and
\[
\|u\|_{L_{t,x}^{2(p-1)}([-T,T]\times\R^3)}^{2(p-1)} \simeq \sum_{k=1}^{\lceil T / \eps \rceil}  \|u\|_{L_{t,x}^{2(p-1)}(J_k \times\R^3)}^{2(p-1)} \lesssim \frac{T}{\eps}.
\]

Using a similar argument together with Strichartz estimates and the hypothesis
\[
\|u\|_{L_t^\infty \H^{s_0}} \lesssim 1,
\]
we obtain that
\begin{align}\label{equ:a_est2}
\| u\|_{L_t^p L_x^{2p}([-T,T]\times\R^3)}\lesssim \left(\frac{T}{\eps} \right)^{\frac{1}{p}}\|u\|_{L_t^\infty \H^{s_0}}.
\end{align}
Thus, using \eqref{equ:a_est1}, \eqref{equ:a_est2} and Strichartz estimates, we can estimate
\begin{align*}
|A|^2 + |A'|^2 
&\lesssim \|u\|_{L_t^p L_x^{2p}([-T,T]\times\R^3)}^{p} \lesssim \left(\frac{T}{\eps} \right) \|u\|_{L_t^\infty \H^{s_0}}^p. 
\end{align*}

\subsection*{Region B} For the estimates of $B$ and $B'$, we use the small data theory to bound the solutions $v$ and $\widetilde v$. We argue only for $v$ as the estimates for $\widetilde{v}$ are identical. By the small data theory, for
$\eta$ chosen sufficiently small in \eqref{equ:small_data_soliton}, we have 
\[
\|v\|_{L_{t,x}^{2(p-1)}(\R^{1+3})} \lesssim \eta.
\]
Using Strichartz estimates, we bound
\begin{align*}
\| |\nabla|^{\frac{3(p-3)}{2p}}v\|_{L_t^p L_x^{\frac{2p}{p-2}}} & \lesssim \|v(T)\|_{\H^{s_0}}+\| |\nabla|^{\frac{3(p-3)}{2p}}(|v|^{p-1}v)\|_{L_t^{\frac{2p}{p+2}}L_x^{\frac{p}{p-1}}} \\
& \lesssim \|u(T)\|_{\H^{s_0}} + \|v\|_{L_{t,x}^{2(p-1)}}^{p-1}\| |\nabla|^{\frac{3(p-3)}{2p}}v\|_{L_t^p L_x^{\frac{2p}{p-2}}},
\end{align*}
with all space-time norms over $\R^{1+3}$. Note that $(p,\frac{2p}{p-2})$ is wave admissible.  Thus, for $\eta$ sufficiently small, we deduce
\[
\| |\nabla|^{\frac{3(p-3)}{2p}}v\|_{L_t^p L_x^{\frac{2p}{p-2}}} \lesssim \|u\|_{L_t^\infty \H^{s_0}(\R^{1+3})}, 
\]
and hence it follows from Sobolev embedding that
\[
 \quad \| v\|_{L_t^p L_x^{2p}(\R^{1+3})}\lesssim \|u\|_{L_t^\infty \H^{s_0}}. 
 \]
Thus, we have shown that
\[
|B|^2+|B'|^2 \lesssim \|v\|_{L_t^p L_x^{2p}(\R^{1+3})}^{p} \lesssim \|u\|_{L_t^\infty \H^{s_0}}^p. 
\]

\subsection*{Region C} Finally, we claim that
\begin{align}\label{equ:c_vanish}
\langle C,C'\rangle \equiv 0. 
\end{align}
To see this, write
\begin{align*}
\langle C,C'\rangle = \int_T^\infty\hspace{-1.5mm}\int_{-\infty}^{-T} &\langle e^{i(\tau-t)\sqrt{-\Delta}}Q_{<k}[F(u(t))- F(v(t))], \\
&\quad Q_{<k}[F(u(\tau))- F(\widetilde v(\tau))]\rangle\,d\tau\,\ud t,
\end{align*}
and note that by subluminality and the fact that $x(0)=0$, we have for $T=2 R\delta_0^{-1}$ the inclusion
\begin{equation}\label{scsl}
\{|x-x(\pm T)|\leq R\}\subset \{|x|\leq (1- 2^{-1} \delta_0) T\}.
\end{equation}
We recall that the operator $Q_{<k}$ defined in \eqref{qk_def}, is given by convolution with the function $2^{3k} \psi(2^k x)$ for a fixed function $\psi \in C_0^\infty(\R^3)$. Hence for $k \geq k_0$, a sufficiently large, fixed constant depending on the support of $\psi$, $\delta_0$ and $T$, we can ensure
\[
\textup{supp } \left( Q_{<k}[F(u(\tau))- F(\widetilde v(\tau))] \right) \subseteq \bigl\{|x|\leq |\tau| - 4^{-1} \delta_0 T \bigr\}.
\]
Similarly, using the properties of the $Q_{<k}$ and the sharp Huygens principle, we can ensure that for $k$ sufficiently large,
\[
\textup{supp } \left(e^{i(\tau-t)\sqrt{-\Delta}}Q_{<k}[F(u(t))- F(v(t))] \right) \subseteq \bigl\{ |x| > |t - \tau| - 4^{-1} \delta_0 T \bigr\}.
\]
Since $t > 0$ and $\tau < 0$, we have that $|t-\tau| > |\tau|$, this yields \eqref{equ:c_vanish}, as required.

 \medskip 
Collecting these estimates, we obtain that
\[
\|Q_{<k} w(0)\|_{\dot H^{1}}^2 = \langle Q_{<k} w(0),Q_{<k} w(0)\rangle_{\dot H^{1}} \lesssim 1
\]
uniformly in $k\geq 0$.  The desired result then follows. \end{proof}

\subsection{The jump from \texorpdfstring{$\dot \cH^{s_p}(\bR^3)$}{H} regularity to \texorpdfstring{$\dot \cH^{s_0}(\bR^3)$}{H} regularity}  
Now we turn to the more difficult estimates. Here, we will need a finer analysis based on frequency envelope machinery. We prove the following.

\begin{prop}\label{P:improve-reg}
Suppose $\vec u$ is a soliton-like critical element. Then
\[
\vec u\in L_t^\infty\dot\H^{s_p} \quad \Longrightarrow \quad \vec u\in L_t^\infty  \dot\H^{s}. 
\]
for any $s_p \leq s < 1$.
\end{prop}

\begin{proof}
  Once again, we define
\begin{align}
\textup{Region }\textbf{A} &:= \{(t,x) \,:\, |t| \leq T\}, \\  
\textup{Region }\textbf{B} &:=\{(t,x) \,:\, |x- x(T)| \geq R + |t - T|\},\\
\textup{Region }\textbf{C} &:= \{(t,x) \,:\, |x- x(T)| < R + |t - T|\}.
\end{align}
with corresponding regions $A', B' , C'$ in the negative time direction. We further introduce
\[
Q_k = Q_{<2k}-Q_{<k} \qtq{for}k >0, \qquad Q_0 = Q_{<0}.
\]
By Schur's test, we can conclude that these frequency projections are a good partition of frequency space, in the sense that
\[
\|f\|_{\dot H^s}^2 \sim \|Q_0 f\|_{\dot H^s}^2 + \sum_{k\geq0} 2^{ks} \|Q_k f\|_{L^2}^2. 
\]
We will also need to introduce an exponent $q$ satisfying
\[
2<q<\frac{2}{s_p}. 
\]

\subsection*{Region \textbf{A}}

We begin by defining suitable frequency envelopes with a parameter $\sigma>0$ to be determined shortly.  We set
\begin{equation}\label{sc-fe}
\begin{aligned}
\gamma_k(t_0)&=\sum_j 2^{-\sigma |j-k|}\bigl[ 2^{s_p j}\|Q_j u(t_0)\|_{L^2}+2^{j(s_p-1)}\|Q_j\partial_t u(t_0)\|_{L^2}\bigr],\\
\alpha_k (J) &= \sum_j  2^{-\sigma |j-k|}\ \bigl[ 2^{-j(\frac{2}{q}-s_p)}\|Q_j u\|_{L_t^q L_x^{\frac{2q}{q-2}}(J\times\R^3)} \\
& \hspace{24mm} + 2^{j\left(\frac{2}{q}-1+s_p\right)}\|Q_ju\|_{L_t^{\frac{2q}{q-2}} L_x^q(J\times\R^3)}\bigr]
\end{aligned}
\end{equation}
for $k \geq 0$. Note $(q,\frac{2q}{q-2})$ is sharp admissible and that each of the quantities appearing in the definition of $\beta_k$ has the same scaling as $\dot \H^{s_p}$.  We will choose 
\[
0<\sigma<\frac2q-s_p. 
\]
Our goal is to prove that
\begin{equation}\label{alphak}
\alpha_k([-T, T]) \lesssim \gamma_k(0) + C_02^{-k \sigma},
\end{equation}
where $C_0=C_0(T)$. 

We begin by recording the some space-time estimates for $\vec u$ that are consequences of the pre-compactness of the set $K$, see Definition \ref{d:cp}.  We fix $\eta >0$. Since $N(t)= 1$, there exists $\eps >0$ small enough that the $L_{t,x}^{2(p-1)}$ norm is $< \eta$ on any interval of length $\eps$; see Lemma~\ref{l:DL1lem}. Furthermore, we can find $k_0 = k_0(\eta)$ such  so that for any $k >k_0$, 
\[
\|Q_{>k}u\|_{L_{t,x}^{2(p-1)}([-T,T]\times\R^3)} < \eta T^{\frac{1}{2(p - 1)}}. 
\]
With these bounds in hand, we turn to the proof of \eqref{alphak}. In the following, all space-time norms will be taken over $[-T,T]\times\R^3$. For any $j$, we decompose the nonlinearity as follows.  Writing $u_{\leq j}=Q_{\leq j}u$ (and similarly for $u_{>j}$), we write
\[
F(u)=F(u_{>k_0}) +F(u)-F(u_{>k_0}),
\]
where $k_0(\eta)$ is as above.  By Taylor's theorem, we have that
\[
F(u)= F(u_{>k_0}) +  u_{\leq k_0 } \int_0^1 F'(\theta u_{\leq  k_0 } +  u_{ > k_0 }) d\theta,
\]
and hence to estimate the nonlinearity, it suffices to estimate three type of terms
\begin{equation}\label{sc-nonlinear-decomp}
u_{>k_0}^{p-1}u_{>j}, \qquad  u_{>k_0}^{p-1}u_{\leq j}, \qquad u_{\leq k_0}u^{p-1}.
\end{equation}
Using the inhomogeneous Strichartz estimates, we obtain
\begin{align} \label{equ:strichartz_duhamel}
&2^{-j(\frac2q-s_p)}\biggl\|\int_0^t e^{i(t-s)\sqrt{-\Delta}} Q_j F(u(s))\,ds\biggr\|_{L_t^q L_x^{\frac{2q}{q-2}}} \\ 
& \quad + 2^{j \left(\frac2q-1+s_p\right)}\biggl\|\int_0^t e^{i(t-s)\sqrt{-\Delta} } Q_j F(u(s))\,ds\biggr\|_{L_t^{\frac{2q}{q-2}}L_x^q} \\
& \lesssim \min\bigl\{ 2^{j (\frac2q-1+s_p)}\|F(u)\|_{L_t^{\frac{q}{q-1}}L_x^{\frac{2q}{q+2}}}, 2^{-j(\frac2q-s_p)}\|F(u)\|_{L_t^{\frac{2q}{q+2}}L_x^{\frac{q}{q-1}}}\bigr\}. 
\end{align}
Now let $J$ be an interval with $|J| < \eps$, and let $t_0 = \inf J$. In the next estimates, all norms will be taken over $J \times \mathbb{R}^3$. Using Strichartz estimates, we estimate
\begin{align*}
2^{-j(\frac2q-s_p)}&\|Q_ju\|_{L_t^q L_x^{\frac{2q}{q-2}}(J \times \mathbb{R}^3)} + 2^{j (\frac2q-1+s_p)}\|Q_j u\|_{L_t^{\frac{2q}{q-2}} L_x^q(J \times \mathbb{R}^3) } \nonumber\\
&\lesssim 2^{js_p}\|u_j(t_0)\|_{L_x^2} + 2^{j(s_p-1)}\|\partial_t u_j(t_0)\|_{L_x^2} \\
& \quad + 2^{j \left(\frac2q-1+s_p\right)}\|u_{>k_0}^{p-1}u_{>j}\|_{L_t^{\frac{q}{q-1}}L_x^{\frac{2q}{q+2}}} \\
& \quad + 2^{-j(\frac2q-s_p)}\|u_{>k_0}^{p-1}u_{\leq j}\|_{L_t^{\frac{2q}{q+2}}L_x^{\frac{q}{q-1}}} \\
& \quad + 2^{-j(\frac2q-s_p)}\|u_{\leq k_0} u^{p-1}\|_{L_t^{\frac{2q}{q+2}}L_x^{\frac{q}{q-1}}}. \\
& \lesssim 2^{js_p}\|u_j(t_0)\|_{L_x^2} + 2^{j(s_p-1)}\|\partial_t u_j(t_0)\|_{L_x^2} + I + II + III.
\end{align*}
We first estimate Term I. We obtain
\begin{align*}
 &2^{j (\frac2q-1+s_p)}\|u_{>k_0}^{p-1}u_{>j}\|_{L_t^{\frac{q}{q-1}}L_x^{\frac{2q}{q+2}}} \\
 & \lesssim 2^{j(\frac2q-1+s_p)}\|u_{>k_0}\|_{L_{t,x}^{2(p-1)}}^{p-1}\|u_{>j}\|_{L_t^{\frac{2q}{q-2}} L_x^q} \\
& \lesssim \eta^{(p-1)}  T^{1/2} 2^{j (\frac2q-1+s_p)}\sum_{\ell>j} 2^{-\ell(\frac2q-1+s_p)} \bigl[2^{\ell(\frac2q-1+s_p)}\|Q_\ell u\|_{L_t^{\frac{2q}{q-2}} L_x^q}\bigr].
\end{align*}
Similarly, for Term II we obtain
\begin{align*}
&2^{-j(\frac2q-s_p)}\|u_{>k_0}^{p-1}u_{\leq j}\|_{L_t^{\frac{2q}{q+2}}L_x^{\frac{q}{q-1}}} \\
  &\lesssim 2^{-j(\frac2q-s_p)}\|u_{>k_0}\|_{L_{t,x}^{2(p-1)}}^{p-1}\|u_{\leq j}\|_{L_t^q L_x^{\frac{2q}{q-2}}} \\
    &\lesssim 2^{-j(\frac2q-s_p)}\|u_{>k_0}\|_{L_{t,x}^{2(p-1)}}^{p-1} \sum_{ \ell \leq j} 2^{\ell (\frac2q-s_p)} \bigl[2^{-\ell(\frac2q-s_p)} \|u_{\ell}\|_{L_t^q L_x^{\frac{2q}{q-2}}}\bigr] .
\end{align*}
Finally, for Term III, using smallness of the interval and we obtain
\begin{align}
2^{-j(\frac2q-s_p)}&\|u_{\leq k_0} u^{p-1}\|_{L_t^{\frac{2q}{q+2}}L_x^{\frac{q}{q-1}}}  \\
&\lesssim 2^{-j(\frac2q-s_p)}\|u\|_{L_{t,x}^{2(p-1)}}^{p-1}\|u_{\leq k_0}\|_{L_t^q L_x^{\frac{2q}{q-2}}} \\
&\lesssim 2^{-j(\frac2q-s_p)}2^{k_0(\frac2q-s_p)} \eta^{p-1} T^{1/2}. 
\end{align}
Multiplying by $2^{-\sigma |j-k|}$ and summing in the above bounds, recalling that $\sigma<\frac2q-s_p$ in our definition of the frequency envelopes in \eqref{sc-fe}, it follows that (for $t_0 = \inf J$) we have
\[
\gamma_k(t_1) + \alpha_{k}(J) \lesssim \gamma_k(t_0) +  T^{1/2} \eta^{p-1} \alpha_{k}(J) + C_0 (T) 2^{-k\sigma} . 
\]
For $\eta \equiv \eta (T)$ small enough so that
\[
C \eta^{p-1} T^{1/2} < \frac{1}{2}
\]
with $C$ the implicit constant in Strichartz estimates, this implies
\[
\alpha_k(J) \lesssim \gamma_k(t_0) + C_0  2^{-k\sigma}.
\]
Iterating this procedure $\lceil T/\eps \rceil$ times on $[-T,T]$, we may also conclude that
\[
\gamma_k(t_0) \lesssim \gamma_k(0),
\]
from which \eqref{alphak} follows by summing up these estimates. 

\subsection*{Region \textbf{B}}

To implement the double Duhamel argument, we will again consider the solution $v$ to \eqref{eq:nlw} with data $\vec v(T)=\chi_R \vec u(T)$.  To control this solution, we define the frequency envelopes 
\[
\widetilde\gamma_k(t_0)\qtq{and}  \beta_k
\]
analogously to \eqref{sc-fe}, but with space-time norms over $\R^{1+3}$.  We will prove
\begin{equation}\label{betak}
\beta_k \lesssim \gamma_k(0) + C_0 2^{-k\sigma}.
\end{equation}
First observe that
\[
\|v\|_{L_{t,x}^{2(p-1)}(\R^{1+3})}\lesssim \eta.
\]
Thus
\begin{align*}
\| |\nabla|^{-(\frac2q-s_p)} &v\|_{L_t^q L_x^{\frac{2q}{q-2}}} + \| |\nabla|^{\frac2q-1+s_p} v\|_{L_t^{\frac{2q}{q-2}} L_x^q} \\
& \lesssim \|v(T)\|_{\H^{s_p}} + \| |\nabla|^{\frac2q-1+s_p}(v|v|^{p-1})\|_{L_t^{\frac{q}{q-1}}L_x^{\frac{2q}{q+2}}} \\
& \lesssim \eta + \|v\|_{L_{t,x}^{2(p-1)}}^{p-1}\||\nabla|^{\frac2q-1+s_p} v\|_{L_t^{\frac{2q}{q-2}} L_x^q} \\
& \lesssim \eta + \eta^{p-1}||\nabla|^{\frac2q-1+s_p} v\|_{L_t^{\frac{2q}{q-2}} L_x^q},
\end{align*}
which implies in particular that
\begin{equation}\label{bk-vgd}
\|v_{\leq 1}\|_{L_t^q L_x^{\frac{2q}{q-2}}(\R^{1+3})} \lesssim\eta. 
\end{equation}

We now estimate $\beta_k$ in essentially the same manner as $\alpha_k$.  The main difference is that we split at frequency $1$ instead of at frequency $k_0$ as above.  Estimating as above, but using \eqref{bk-vgd}, we deduce
\[
\beta_k \lesssim \widetilde\gamma_k(T) + \eta^{p-1} \beta_k + C_0 2^{-k\sigma},
\]
which implies
\begin{align} \label{equ:first_est}
\beta_k \lesssim \widetilde \gamma_k(T) + C_0 2^{-k\sigma}.
\end{align}
In order to prove \eqref{betak}, we need to relate $\widetilde\gamma_k(T)$ to $\gamma_k(0)$.  Similar arguments as in \eqref{betak} yield
\[
\gamma_k(T) \lesssim \gamma_k(0) + \eta^{p-1}\beta_k + C_0 2^{-k\sigma} \lesssim \gamma_k(0) + C_0 2^{-k\sigma},
\]
so it therefore suffices to relate $\widetilde \gamma_k(T)$ to $\gamma_k(T)$.  Using that $\vec v(T)=\chi_R \vec u(T)$, we apply the commutator estimate Lemma~\ref{L:commutator} to deduce
\begin{align*}
2^{ks_p}\|Q_k v(T)\|_{L^2} & \lesssim 2^{ks_p}\|Q_k u(T)\|_{L^2} + (2^k R)^{-(1-s_p)}\|u\|_{L_t^\infty \dot H^{s_p}}, \\
2^{k(s_p-1)}\|Q_k\partial_t v(T)\|_{L^2} & \lesssim 2^{k(s_p-1)}\|Q_k\partial_t u(T)\|_{L^2} + 2^{-k}R^{-1}\|\partial_t u\|_{L_t^\infty \dot H^{s_p-1}}. 
\end{align*}
In particular, since $\sigma<\frac2q-s_p<1-s_p$, we deduce that
\[
\widetilde\gamma_k(T) \lesssim \gamma_k(T) + C_0 2^{-k\sigma}.
\]
Putting this together with \eqref{equ:first_est} above, we conclude
\[
\beta_k \lesssim \gamma_k(0)+C_02^{-k\sigma},
\]
which completes the proof of \eqref{betak}. 

\medskip
We will now carry out the double Duhamel argument with the complexified solutions $w$.  
%
  We write
\begin{align}\label{equ:double_duh_term}
&\langle Q_jw(0),Q_jw(0)\bigl\rangle_{\dot H^1}\\
& = \lim_{T \to\infty}\langle \int_0^T e^{-i t\sqrt{-\Delta} }Q_j F(u(t))\,\ud t,\int_{-T}^0 e^{-i \tau \sqrt{-\Delta}}Q_j F(u(\tau))\,d\tau\bigr\rangle_{\dot H^1}
\end{align}
and (as before) decompose
\begin{align}
\int_0^\infty e^{-i t\sqrt{-\Delta} }Q_j F(u(t))\,\ud t &= A+B+C\\
&=A'+B'+C'=\int_{-\infty}^0 e^{-i \tau \sqrt{-\Delta}}Q_jF(u(\tau))\,d\tau,
\end{align}
for components as in  \eqref{scdd-decomp1} and \eqref{scdd-decomp2}.  Once again, we rely on the algebraic inequality
\begin{align} \label{regularity1}
\langle Q_jw(0),Q_jw(0)\bigl\rangle_{\dot H^1} \lesssim |A|^2 + |A'|^2 + |B|^2 + |B'|^2 + |\langle C,C'\rangle|
\end{align}
 and we note that by construction and the argument above relying on the sharp Huygens principle, $\langle C,C'\rangle_{\dot \H^1}\equiv 0$. 
 
 To treat the other terms, we recall the definition of the frequency envelope $\alpha_k$ in \eqref{sc-fe}, and we use \eqref{alphak} and \eqref{betak}. To this end, we multiply the left-hand side of \eqref{equ:double_duh_term} by $2^{-\sigma| j-k|}$ and we sum over $j \geq 0$ to obtain
  \begin{align}\label{equ:outcome0}
 \gamma_k(0) \lesssim \eta^{p-1} \gamma_k(0) + C_0 2^{-k \sigma},
 \end{align}
 which, choosing $\eta$ sufficiently small depending only on the implicit constant, implies
  \begin{align}\label{equ:outcome1}
 \gamma_k(0) \lesssim C_0 2^{-k \sigma},
 \end{align}
which yields $\vec u\in \dot \H^{s}$ for any $s_p\leq s<s_p+\sigma$.  Since that we may choose any $\sigma<\frac2q-s_p$ and $q$ arbitrarily close to $2$, we deduce $\vec u\in L_t^\infty \dot \H^{s}$ for any $s_p\leq s<1$.  This completes the proof of Proposition \ref{P:improve-reg}. \end{proof}

Propositions \ref{P:reg-jump} and \ref{P:improve-reg} immediately yield the following corollary.

\begin{cor}
Suppose $\vec u$ is a soliton-like critical element. Then
\[
\vec u\in L_t^\infty\dot\H^{s_p} \quad \Longrightarrow \quad \vec u\in L_t^\infty  \dot\H^{1}. 
\]
\end{cor}

\subsection{The jump from \texorpdfstring{$\dot \cH^{1}(\bR^3)$}{H} regularity to \texorpdfstring{$\dot \cH^{s}(\bR^3)$}{H} regularity}  
As mentioned above, in order to employ the rigidity argument based on a certain virial identity, we also need to prove that the trajectory of a critical element in fact lies in a pre-compact subset of $\dot \H^1$. We will achieve this by proving that in fact, we can gain a bit more regularity, specifically we can place the solution in $\dot \H^s$ for some $s > 1$. The key idea here is that we actually have a bit of room in the previous estimates given the additional assumption of $\dot \cH^1$ regularity, and this will provide some extra decay which we can use to establish the additional regularity.

\begin{prop}\label{P:improve-reg2}
Suppose $\vec u$ is a soliton-like critical element. Then $\vec u\in L_t^\infty\dot\H^{s} $ for some $s > 1$.
\end{prop}

\begin{proof}
Let $v$ and $\widetilde{v}$ be the solutions to the small data Cauchy problems defined above. By small data arguments $v(T) \in \dot \cH^{1}(\bR^3)$ and $\| v(T) \|_{\dot \cH^{s_p}}$ small implies that
\begin{align}
\| v \|_{ L_t^{\frac{2q}{q-2}}  L_x^q(\bR \times \bR^{3})} + \| |\nabla|^{1- s_p} v \|_{L_{t,x}^{2(p-1)}(\bR \times \bR^{3})} &\lesssim_T 1\\
\| \tilde{v} \|_{ L_t^{\frac{2q}{q-2}}  L_x^q(\bR \times \bR^{3})} + \| |\nabla|^{1 - s_p} \tilde{v} \|_{L_{t,x}^{2(p-1)}(\bR \times \bR^{3})} &\lesssim_{T} 1.
\end{align}
Furthermore, arguing as above and partitioning $[-T, T]$ into sufficiently small intervals, we obtain
\[
\| u \|_{ L_t^{\frac{2q}{q-2}}  L_x^q ([-T, T] \times \bR^{3})} + \| |\nabla|^{1- s_p} u \|_{L_{t,x}^{2(p-1)}([-T,T] \times \bR^{3})} \lesssim_{T} 1.
\]
These inequalities, together with the argument used to prove Proposition \ref{P:improve-reg}, as well as \eqref{regularity1} and \eqref{equ:strichartz_duhamel} establishes that
\begin{align}
\| Q_{k} u(0) \|_{\mathcal H^{1}}^{2} &\lesssim_{T} 2^{-k\sigma} 2^{-k \left(\frac{2}{q} - 1+ s_p\right)}.
\end{align}
Since we may choose any
\[
\sigma < \frac{2}{q} - s_p,
\]
and $q$ arbitrarily close to $2$, we have then shown that
\[
\sum_{k} 2^{2 \alpha k} \| Q_{k} u(0) \|_{\mathcal H^{1}}^{2}  <  \infty
\]
for any $\alpha < \frac{1}{2}$, which concludes the proof.
\end{proof}

\subsection{Rigidity for the soliton-like critical element} 

Now we may prove that the soliton-like critical element is identically zero. We summarize this in the following proposition.

\begin{prop} \label{p:srig} 
Let $\vec u(t) \in \dot\H^1$ be a global-in-time solution to~\eqref{eq:nlw} such that for subluminal $x(t)$ the set 
\EQ{
K = \{ u( t,  \cdot - x(t)), \p_t u (t,  \cdot - x(t)) \mid t \in \R\} \subset \dot\H^1 \cap \dot \H^{s_p} 
} 
is a pre-compact subset of $\dot \H^1 \cap \dot \H^{s_p}$. Then $\vec u(t) \equiv 0$. 
\end{prop} 

As mentioned in the introduction, we include a proof of rigidity for the soliton-like critical element in the \textit{focusing setting} as well. The arguments that we use are similar to the ones given in~\cite[Section 3]{CKLS1} and~\cite{DL1, CR5d} but with a modification. The key new ingredient here is that the subluminality  of $x(t)$ compactifies the subset of the Lorentz group taking $(t, x(t))$ to $(t', 0)$; see also~\cite{KM06, NakS} for a somewhat different approach that uses the Lorentz transform to show that critical elements must have zero momentum.    The main ingredients in the proof are the following virial identities. 

In what follows we let $r=|x|$ and denote $\partial_r u = \nabla u \cdot \tfrac{x}{|x|}$.

\begin{lem}[Virial Identities]
Let $\chi \in C^\infty_0$ be a smooth radial function such that $\chi(r) = 1$ if $r \le 1$ and $\supp \chi \in \{ r \le 2 \}$.  For any $R>0$ we define $\chi_R(r) = \chi(r/ R)$ and let $\vec u(t)$ be a solution to \eqref{eq:nlw}. Denoting 
\EQ{\label{eq:Omdef} 
\Om_{u(t)}(R):= \int_{\abs{ x} \ge R}  | \na u |^2 + | \p_t u|^2 + \frac{\abs{u}^2}{\abs{x}^2} + \abs{ u}^{p+1} \, \ud x,
}
we have
\EQ{\label{eq:vir3d} 
 \frac{\ud }{ \ud t} \ang{ \p_t u  \mid \chi_R  ( r \p_r u + u) } & = -E( \vec u)  \pm  \big( \tfrac{p-3}{p+1} \big) \| u \|_{L^{p+1}}^{p+1} \\
 & \quad + O( \Om_{u(t)} (R)), 
}
where the $+$ sign above corresponds to the focusing equation and the $-$ sign corresponds to the defocusing equation. 

If $\vec  u(t)$ solves the focusing equation, we have 
\EQ{ \label{eq:vfoc} 
 \frac{\ud }{ \ud t} \ang{ \p_t u  \mid \chi_R  ( r \p_r u +  \frac{1}{2}u) } &= -\frac{1}{2} \int \abs{\p_t u}^2  - 3 \left( \frac{1}{p+1} - \frac{1}{6} \right) \int \abs{u}^{p+1}  \\
 & \quad + O (\Om_{u(t)} (R)).  
}
\end{lem} 

\begin{proof}[Proof of Proposition~\ref{p:srig} for the focusing equation]

We may assume that $x(0) = 0$. Since $x(t)$ is subluminal we can find $\de >0$ so that 
\EQ{ \label{eq:sl1} 
\abs{ x (t) - x( \tau) } \le (1-\de) \abs{t - \tau}, \quad \abs{x(t)} \le (1-\de) \abs{t} 
}
for all $t, \tau \in \R$. 

For convenience, we consider only the special case where 
\EQ{
x(t) = (x_1(t), 0, 0) \quad  \forall t >0,
}
as this contains the essential difficulties and the general argument is an easy modification of the one presented below. Recall that for each $\nu \in (-1, 1)$ we have a Lorentz transform $L_\nu$  defined by 
\EQ{
L_{\nu}(t, x_1, x_1, x_3) = \Big( \frac{ t -\nu x_1}{\sqrt{1- \nu^2}}, \frac{ x_1 - \nu t}{\sqrt{1- \nu^2}}, x_2, x_3 \Big)=:(t', x').
}
For any $T >0$, set 
\EQ{
\nu(T) := \frac{x_1(T)}{T}.
}
Then
\EQ{ \label{eq:alde} 
-(1-\de) \le \nu(T) \le 1-\de
}
and the Lorentz transform $L_{\nu(T)}$ gives 
\EQ{
L_{\nu(T)} (T, x_1(T), 0, 0) = (T', 0, 0, 0),
}
where 
\EQ{ \label{eq:T'} 
T' =  \sqrt{T^2 - x_1(T)^2}.
}  
Since  $x(t)$ satisfies~\eqref{eq:sl1}, we have the bounds 
\EQ{
c_\de T   \le T'  \le T
}
for $c_\de:= \sqrt{1 - (1-\de)^2}>0$, which means that $T'$ is comparable to $T$. For each $T>0$ define 
\EQ{
v_{\nu(T)}(t', x') := u \circ L_{\nu(T)} (t, x) 
}
Then, since $K$ above is pre-compact for $x(t)$ subluminal  and since  $\vec v_{\nu(T)}(t')$ as above is a fixed Lorentz transform of $\vec u(t, x)$, we can find a subluminal translation parameter $x'(t')$ with 
\EQ{
x'(T') = 0
} 
such that the trajectory 
\EQ{\label{eq:K'}
K' := \{ \vec v_{\nu(T)}(t', x - x'(t')) \mid t' \in \R\}
}
is pre-compact in $\dot \H^1 \cap \dot \H^{s_p}$. We will now establish the following.

\begin{claim} \label{c:Tn} 
Consider a critical element for the focusing equation with $3 \le p <5$. For each $n$ there exists a time $T_n>0$ such that for $T_n'$ as in~\eqref{eq:T'} we have 
\EQ{
\frac{1}{T'_n} \int_0^{T_n'} \int_{\R^3} \abs{\p_t v_{\nu(T_n)}(t, x) }^2 + \abs{v_{\nu(T_n)}(t, x)}^{p+1} \, \ud x \, \ud t < \frac{1}{4n} 
}
\end{claim} 
\begin{proof}[Proof of Claim~\ref{c:Tn}] 
Let $T>0$. Since $\vec v_{\nu(T)}$ solves the focusing equation we average~\eqref{eq:vfoc} with $R = C_\de T$  over the time interval $[0, T']$ for some constant $C_\de$ to be specified below, yielding
\EQ{ \label{eq:vir1v} 
\frac{1}{T'} \int_0^{T'} \int_{\R^3} &\abs{\p_t v_{\nu(T)}(t, x) }^2 + \abs{v_{\nu(T)}(t, x)}^{p+1} \, \ud x \, \ud t \\
& \lesssim \frac{1}{T'} \abs{ \ang{ \p_t v_{\nu(T)}(t) \mid \chi_{2T}  r \p_r v_{\nu(T)} } \vert^{T'}_0}  \\
& \quad + \frac{1}{T'} \abs{ \ang{ \p_t v_{\nu(T)}(t) \mid \chi_{2T} v_{\nu(T)} } \vert^{T'}_0}\\
& \quad +\frac{1}{T'} \int_0^{T'}  \Om_{v_{\nu(T)}(t)}(C_\de T) \ud t 
}
where $\Om_{v_{\nu(T)}}(C_\de T)$ is defined as in~\eqref{eq:Omdef}. Given $n>0$, by~\eqref{eq:K'}, the subluminality of $x'(t')$, and the fact that 
\EQ{
x'(0) = 0, \quad x'(T') = 0
}
we can choose $C_\de$ and $T = T_n$ large enough so that 
\EQ{
\frac{1}{T_n'} \int_0^{T_n'} \Om_{v_{\nu(T_n)}(t)}(C_\de T) \ud t  \ll \frac{1}{n}.
}
Note that $C_\de$ can be chosen  independently of $n$. Next we estimate the first term on the right-hand side of~\eqref{eq:vir1v}.  We treat only the case where the inner product is evaluated at $t = T'$, as the case when it is evaluated at $t =0$ is similar. We have 
\EQ{
\frac{1}{T'} &\abs{ \ang{ \p_t v_{\nu(T)}(T') \mid \chi_{2T}\cdot   r \, \p_r v_{\nu(T)}(T') }} \\
&\lesssim  \frac{T^{\frac{1}{2}}}{T'} \|\p_t v_{\nu(T)}(T') \|_{L^2}  \| \na v_{\nu(T)}(T') \|_{L^2( \abs{x} \le T^{\frac{1}{2}})} \\
& \quad + \frac{C_\de}{c_\de} \|\p_t v_{\nu(T)}(T') \|_{L^2}  \| \na v_{\nu(T)}(T') \|_{L^2( T^{\frac{1}{2}} \le \abs{x} \le C_\de T)}.
}
Since $T' \simeq_{\de} T$, the first term on the right-hand side above can be made as small as we like by choosing $T_n$ large enough so that 
\EQ{
\frac{T_n^{\frac{1}{2}}}{T_n'} \ll \frac{1}{n}.
}
Similarly, for the second term on the right, we rely on the pre-compactness of $K'$  in $\dot \H^1 \cap \dot \H^{s_p}$ and the fact that $x'(T'_n) = 0$ which yields,  
\EQ{
\| \na v_{\nu(T_n)}(T'_n) \|_{L^2( \abs{x} \ge T_n^{\frac{1}{2}}) } \ll \frac{1}{n} 
}
for $T_n$ large enough. The second term on the right-hand-side of~\eqref{eq:vir1v} is estimated in a similar fashion. This completes the proof of Claim~\ref{c:Tn}. 
\end{proof} 
Now, given this sequence of times $T_n$ guaranteed by Claim~\ref{c:Tn} consider the sequence $\nu(T_n):= x_1(T_n)/ T_n$. By~\eqref{eq:alde} we can, passing to subsequence that we still denote by $\nu(T_n)$,  find a fixed $\nu \in [-1-\de, 1-\de]$ with 
\EQ{ \label{eq:allim} 
\nu(T_n) \to \nu_0 \mas n \to \infty.
}
Define 
\EQ{
v_{\nu_0}(t', x'): = u \circ L_{\nu_0}(t, x)
}
and note that this is a \emph{fixed} Lorentz transform of $u$. 
It follows from Claim~\ref{c:Tn}, ~\eqref{eq:allim}, and a continuity argument that in fact, 
\EQ{
\frac{1}{T'_n} \int_0^{T_n'} \int_{\R^3} \abs{\p_t v_{\nu_0}(t, x) }^2 + \abs{v_{\nu_0}(t, x)}^{p+1} \, \ud x \, \ud t < \frac{1}{2n} 
}
after passing to a further subsequence. Using yet another continuity argument we can assume without loss of generality that the $T_n' = M_n \in \N$, i.e., 
\EQ{\label{eq:Mn} 
\frac{1}{M_n} \int_0^{M_n} \int_{\R^3} \abs{\p_t v_{\nu_0}(t, x) }^2 + \abs{v_{\nu_0}(t, x)}^{p+1} \, \ud x \, \ud t < \frac{1}{n} 
}
for some sequence $\{M_n \} \subset \N$ with $M_n  \to \infty$. Now we claim that there exists a sequence of positive integers $m_n \to \infty$ such that 
\EQ{ \label{eq:mn}
\int_{m_n}^{m_n +1} \int_{\R^3} \abs{\p_t v_{\nu_0}(t, x) }^2 + \abs{v_{\nu_0}(t, x)}^{p+1} \, \ud x \, \ud t  \to 0 \mas n  \to \infty.
}
If not, we could find $\eps_0>0$ such that for all $n \in \Z$ we have 
\EQ{
\int_{m}^{m +1} \int_{\R^3} \abs{\p_t v_{\nu_0}(t, x) }^2 + \abs{v_{\nu_0}(t, x)}^{p+1} \, \ud x \, \ud t  \ge \eps_0.
}
However, summing up from $0$ to $M_n -1$ we would then have 
\EQ{
\int^{M_n}_{0} \int_{\R^3} \abs{\p_t v_{\nu_0}(t, x) }^2 + \abs{v_{\nu_0}(t, x)}^{p+1} \, \ud x \, \ud t  \ge \eps_0 M_n,
}
which contradicts~\eqref{eq:Mn}. Now, by~\eqref{eq:mn} we have 
\EQ{ \label{eq:0lim} 
\int_{0}^{1} \int_{\R^3} \abs{\p_t v_{\nu_0}(m_n+t, x) }^2 + \abs{v_{\nu_0}(m_n+t, x)}^{p+1} \, \ud x \, \ud t  \to 0
}
as $n\to\infty$. On the other hand, passing to a further subsequence, we can find $(V_0, V_1) \in \dot \H^1 \cap \dot \H^{s_p}$ such that 
\EQ{
\vec v_{\nu_0}( m_n,  \cdot - x'(m_n)) \to (V_0, V_1)\in \dot \H^1 \cap \dot \H^{s_p} \mas n \to \infty.
}
Let $\vec V(t)$ be the solution to~\eqref{eq:nlw} with data $(V_0, V_1)$. Then for some $t_0>0$ sufficiently small we have 
\begin{align}\label{val_conv}
\lim_{n \to \infty} \sup_{t \in [0, t_0]} \| \vec v_{\nu_0}(m_n + t, \cdot - x'(m_n)) - \vec V(t) \|_{\dot \H^1 \cap \dot \H^{s_p}} = 0 
\end{align}
However, from~\eqref{eq:0lim} we can then conclude that 
\EQ{
\vec V \equiv 0,
}
from which we conclude from \eqref{val_conv} and small data arguments that
\EQ{
\vec v_{\nu_0} \equiv 0.
}
This means $\vec u \equiv 0$ as well, which finishes the proof. 
\end{proof} 

\begin{proof}[Proof of Proposition~\ref{p:srig} for the defocusing equation]
The argument is much easier if either $p =3$ or if the equation is defocusing since~\eqref{eq:vir3d} gives us coercive control over the energy. Indeed, arguing as in the proof of Claim~\ref{c:Tn}, but using~\eqref{eq:vir3d} instead of~\eqref{eq:vfoc} we see that 
\EQ{
E(\vec v_{\nu(T)})  =  \frac{1}{T'} \int_0^{T'} E(\vec v_{\nu(T)} ) \, \ud t = o(1) \mas T \to \infty
}
since for each fixed $T$, the energy of $v_{\nu(T)}(t)$ is constant in time. However, since 
\EQ{
v_{\nu(T)}(t', x') = u \circ L_{\nu(T)} (t, x), 
}
we must have either $\limsup_{T \to \infty} \abs{\nu(T)}  = 1$,  or $E(\vec u) =0$. The former is impossible by~\eqref{eq:alde}. Hence $E(\vec u) = 0$. Therefore $\vec u \equiv 0$. 
\end{proof} 
\begin{rem} 
Note the argument given above for the defocusing equation also works for the cubic focusing equation since~\eqref{eq:vir3d} yields  control of the full energy for $p=3$. Arguing as above one can conclude that $E(u) =0$. Since the only nonzero solutions with zero energy must blow up in both time directions~\cite{KSV} we conclude that the global-in-time solution $\vec u \equiv 0$; see~\cite{DL1} where a version of this argument was carried out in detail. 
 \end{rem}

\section{The self-similar critical element}  \label{s:ss}

In this section, we assume towards a contradiction that $\vec u$ is a self-similar-like critical element as in Proposition~\ref{p:cases}, case (III).  We will prove that any such $\vec u$ has finite energy, and in fact that $E(\vec u)=0$.  Since we are treating the defocusing equation, this implies $\vec u \equiv 0$. The arguments in this section can be readily adapted to the focusing setting as well.

More precisely, we will prove the following result.
  \begin{prop}\label{ssim_0}
There are no self-similar-like critical elements, in the sense of Case (II) of Proposition~\ref{p:cases}.
 \end{prop}

 As in Section \ref{s:soliton}, we will prove this proposition via two additional regularity arguments. We fix the following notation: let 
\EQ{ \label{eq:s0} 
s_0=s_p+\frac{5-p}{2p(p-1)}=\frac32-\frac5{2p},
}

  \begin{prop} \label{ad_ref_ss}
Let $\vec u$ be a self-similar-like critical element as in Proposition~\ref{p:cases}. Then, 
\begin{equation}
\|\vec u(T)\|_{\H^{s_0}}\lesssim T^{-(s_0-s_p)}
\end{equation}
uniformly in $T>0$.
\end{prop}

 \begin{prop}\label{P:reg-jump-ss} Let $\vec u$ be a self-similar-like critical element as in Proposition~\ref{p:cases}. Let $s_0$ be as in~\eqref{eq:s0} and suppose that
\begin{equation}\label{s0pluseps}
\|\vec u(T)\|_{\dot \cH^{s_0}}\lesssim T^{-(s_0-s_p)}
\end{equation}
uniformly in $T>0$, then
\[
\|\vec u(T)\|_{\dot \cH^{1}}\lesssim T^{-p(s_0-s_p)}
\]
uniformly in $T>0$.
\end{prop}
 
Proposition \ref{P:reg-jump-ss} will immediately imply Proposition \ref{ssim_0}.
 
 \begin{proof}[Proof of Proposition \ref{ssim_0} assuming Proposition \ref{P:reg-jump-ss}]
Note that the nonlinear component of the energy is controlled by the $ \dot H^{\frac{3}{2} - \frac{3}{p+1}}(\bR^3)$ norm  by Sobolev embedding, and by interpolation we have
 \[
 \dot H^{\frac{3}{2} - \frac{3}{p+1}}(\bR^3) \subseteq \dot H^{s_p} \cap \dot H^1.
 \]
Thus the conserved energy $E(\vec u)$ must be zero by sending $T \to \infty$ in Proposition~\ref{P:reg-jump-ss}. Then $E[\vec u]\equiv 0$, which implies that $\vec u \equiv 0$, which is impossible. 
 \end{proof}
 
Proposition~\ref{P:reg-jump-ss} is the easier of the two additional regularity arguments, so we turn to this first.
 
 \subsection{The jump from \texorpdfstring{$\dot \cH^{s_0}(\bR^3)$}{H} to \texorpdfstring{$\dot \cH^{1}(\bR^3)$}{H} regularity.}
 We first prove that if $\vec{u}$ has some additional regularity, then we can achieve $\dot \H^1$ regularity, and hence reach the desired contradiction. 

\begin{proof}[Proof of \protect{Proposition~\ref{P:reg-jump-ss}}]
Using $N(t)=t^{-1}$, we have
\[
\|u\|_{L_{t,x}^{2(p-1)}([2^k,2^{k+1}]\times\R^3)} \lesssim 1
\]
uniformly in $k$. Thus for any $0<\eta\ll 1$, we can partition $[2^{k},2^{k+1}]$ into $C(\eta)$ intervals $I_j$ so that
\[
\|u\|_{L_{t,x}^{2(p-1)}(I_j\times\R^3)} < \eta. 
\]
On each such interval, we may argue using Strichartz estimates and a continuity argument together with \eqref{s0pluseps} to deduce that
\[
\| u\|_{L_t^p L_x^{2p}(I_j\times\R^3)} \lesssim 2^{-k(s_0-s_p)} 
\]
for each $j$. This implies
\[
\|u\|_{L_t^p L_x^{2p}([2^k,2^{k+1}]\times\R^3)}\lesssim 2^{-k(s_0-s_p)}
\]
uniformly in $k$. We once again complexify the solution. We let
 \[
 w = u+ \frac{i}{\sqrt{-\Delta}} u_t.
 \]
Once again, if $\vec u(t)$ solves \eqref{eq:nlw}, then $w(t)$ is a solution to
 \begin{align}
 w_t = -i \sqrt{-\Delta} w \pm \frac{i}{\sqrt{-\Delta}} |u|^{p-1} u.
 \end{align} 
 By compactness,
\begin{equation}\label{weakH1}
\lim_{T\to\infty} P_{\leq k} e^{iT \sqrt{-\Delta} } w(-T) = 0
\end{equation}
as weak limits in $\dot H^{s_0}$ for any $k \geq 0$.  By Strichartz estimates, we have
\begin{align*}
\|P_{\leq k} w(T)\|_{\dot H^{1}}  \lesssim \|u|u|^{p-1}\|_{L_t^1 L_x^2([T,\infty)\times\R^3)} & \lesssim \sum_{2^k \geq \frac{T}{2}} \|u|u|^{p-1}\|_{L_t^1 L_x^2([2^k, 2^{k+1}]\times\R^3)} \\
& \lesssim \sum_{2^k \geq \frac{T}{2}} 2^{-k p(s_0-s_p)} \lesssim T^{-p(s_0-s_p)},
\end{align*}
which completes the proof.
\end{proof}

\subsection*{The jump from \texorpdfstring{$\dot \cH^{s_p}$}{H} to \texorpdfstring{$\dot \cH^{s_0}$}{H} regularity.}

It remains to prove Proposition~\ref{ad_ref_ss}.  The main technical ingredient in the proof of Proposition~\ref{ad_ref_ss} is a long-time Strichartz estimate. 
\begin{prop}[Long-time Strichartz estimate]\label{prop-ss-lts} Let $\alpha \geq 1$ and
\[
2<q<\frac{2}{s_p}.
\]
Suppose $\vec u$ is a self-similar-like critical element as in Proposition \ref{p:cases} with compactness modulus function $R(\cdot)$.  For any $\eta_0>0$, there exists $k_0=k_0(R(\eta_0),\alpha)$ so that for every $k > k_0$,
\begin{align*}
 &\| |\nabla|^{\frac{3(p-3)}{2(p-1)}} u_{>k}\|_{L_t^{2(p-1)}L_x^{\frac{2(p-1)}{p-2}}([1,2^{\alpha(k-k_0)}]\times\R^3)} \\
 &\hspace{15mm}+ \| |\nabla|^{-(\frac{2}{q}-s_p)}u_{>k}\|_{L_t^q L_x^{\frac{2q}{q-2}}([1,2^{\alpha(k-k_0)}]\times\R^3)}<\eta_0. 
\end{align*}
\end{prop}

\begin{proof} We proceed by induction on $k > k_0$. Let $\eta_0>0$.  Using compactness and the fact that $N(t)=t^{-1}$, we may find $k$ large enough that
\begin{align*}
\| & |\nabla|^{\frac{3(p-3)}{2(p-1)}} u_{>k_0}\|_{L_t^{2(p-1)}L_x^{\frac{2(p-1)}{p-2}}([1,2^{3\alpha}]\times\R^3)} \\
& \quad + \| |\nabla|^{-(\frac2q-s_p)}u_{>k_0}\|_{L_t^q L_x^{\frac{2q}{q-2}}([1,2^{3\alpha}]\times\R^3)} <\frac12\eta_0.
\end{align*}
This implies the result for $k_0\leq k\leq 8k_0$.  

To establish the induction step, by Taylor's theorem, we may decompose 
\begin{align}
F(u) &= F(u_{> k-3 }) +  u_{\leq k-3 } \int_0^1 F'(\theta u_{\leq  k-3 } +  u_{ > k-3 })\\
&= F(u_{> k-3 }) +  u_{\leq  k-3} F'(u_{> k-3 }) \\
& \hspace{24mm}+ u_{\leq k-3 }^2 \int_0^1 \int_0^1 F''(\theta_1 \theta_2 u_{< k-3} +  u_{> k-3}) d\theta_1 d\theta_2.
\end{align}
Hence, we can write the nonlinearity $F(u)$ as a sum of terms
\[
P_{>k} F(u) = |P_{> k - 3}u|^{p-1}P_{> k - 3}u + P_{>k}(u_{\leq k-3 }F'(u_{> k-3 }) ) + P_{>k}(u_{\leq k-3 }^2 P_{>k-3} F_2)
\]
where
\[
F_2 =  \iint_0^1 F''(\theta_1 \theta_2 u_{< k-3} +  u_{> k-3}) d\theta_1 d\theta_2,
\]
and we have used in the last term that 
\[
P_{>k}(u_{\leq k-3 }^2 F_2) = P_{>k}(u_{\leq k-3 }^2 P_{>k-3} F_2).
\]
Note that $|F'(u_{> k-3 })| \lesssim |u_{> k-3}|^{p-1}$ and $|F_2| \lesssim |u|^{p-2}$, and so we may replace these terms with $|u|^{p-1}$ and $|u|^{p-2}$ respectively once we have a chosen a dual space.

Fix exponents
\begin{equation}\label{ss-ab}
\gamma=\frac{q}{2},\quad \rho=\frac{6(-q+pq)}{12-12p-21q+13pq},
\end{equation}
and note that $\gamma\in(1,2)$ for $q\in(2,4)$, while for $q=2$, we have $\rho=\frac{6(p-1)}{7p-15}\in(\frac65,2)$.  In particular, by choosing $q$ close to $2$ we can guarantee that $\gamma,\rho\in(1,2)$.  Furthermore,
\[
\frac{1}{\alpha}+\frac{1}{\beta}-\frac32 = \frac{2(p-3)}{3(p-1)}\geq 0,
\]
which guarantees that the conjugate exponent pair $(\gamma',\rho')$ is wave-admissible. 

By Strichartz estimates, 
\begin{align}
\| &|\nabla|^{\frac{3(p-3)}{2(p-1)}} u_{>k} \|_{L_t^{2(p-1)}L_x^{\frac{2(p-1)}{p-2}}} + \| |\nabla|^{-(\frac2q-s_p)}u_{>k}\|_{L_t^q L_x^{\frac{2q}{q-2}}} \nonumber \\
& \lesssim \|u_{>k}(1)\|_{\H^{s_p}} + \| |\nabla|^{\frac{3(p-3)}{2(p-1)}}[u_{> k-3}]^p\|_{L_t^{\frac{2(p-1)}{p}}L_x^{\frac{2(p-1)}{2p-3}}} \label{J-ss-lts1}\\
& \quad + \||\nabla|^{-(\frac{2}{q}-s_p)}P_{>k}(u_{\leq k-3} u_{>k-3}^{p-1})\|_{L_t^{\frac{2q}{q+2}}L_x^{\frac{q}{q-1}}} \label{J-ss-lts2} \\
& \quad + \| |\nabla|^{-\frac4q+3s_p}P_{>k}(u_{\leq k-3}^2P_{>k}(u^{p-2}))\|_{L_t^\gamma L_x^\rho}\label{J-ss-lts3}\\
&:= I + II + II,
\end{align}
where all space-time norms are over $[1,2^{\alpha(k-k_0)} ]\times\R^3$. We estimate Term I as follows:
\begin{align}
\| |\nabla|^{\frac{3(p-3)}{2(p-1)}}&[u_{> k-3}]^p\|_{L_t^{\frac{2(p-1)}{p}}L_x^{\frac{2(p-1)}{2p-3}}} \nonumber \\
& \lesssim \| u_{> k-3}\|_{L_{t,x}^{2(p-1)}}^{p-1} \| |\nabla|^{\frac{3(p-3)}{2(p-1)}} u_{> k-3} \|_{L_t^{2(p-1)} L_x^{\frac{2(p-1)}{p-2)}}} \nonumber \\
& \lesssim \| |\nabla|^{\frac{3(p-3)}{2(p-1)}} u_{> k-3} \|_{L_t^{2(p-1)} L_x^{\frac{2(p-1)}{p-2)}}}^{p}. \label{equ:ss-term1}
\end{align}
By induction, we have
\[
 \| |\nabla|^{\frac{3(p-3)}{2(p-1)}} u_{> k-3} \|_{L_t^{2(p-1)} L_x^{\frac{2(p-1)}{p-2}}([1,2^{\alpha(k-k_0)}/8]\times\R^3)}\leq \eta_0. 
\] 
Thus, using $N(t)=t^{-1}$ and that
\[
\int_{2^{\alpha(k-k_0)}/8}^{2^{\alpha(k-k_0 )}} t^{-1}\,\ud t = \log 2^{3\alpha} \sim 1,
\]
and the fact that $N(t)\leq 1$ on $[2^{\alpha(k-k_0)}/8,\,2^{\alpha(k-k_0 )}]$ for $k >  k_0 \gg 1$, we can deduce
\[
 \| |\nabla|^{\frac{3(p-3)}{2(p-1)}} u_{> k-3} \|_{L_t^{2(p-1)} L_x^{\frac{2(p-1)}{p-2}}([2^{\alpha(k-k_0)}/8,\,2^{\alpha(k-k_0 )}]\times\R^3)}\leq \eta_0. 
\]
In particular, using \eqref{equ:ss-term1}, we obtain that
\[
\| |\nabla|^{\frac{3(p-3)}{2(p-1)}}[u_{> k-3}]^p\|_{L_t^{\frac{2(p-1)}{p}}L_x^{\frac{2(p-1)}{2p-3}}([1,2^{\alpha(k-k_0)}]\times\R^3)} \lesssim \eta_0^p.
\]

For Term II, we estimate
\begin{align*}
2^{-k(\frac{2}{q}-s_p)} \|u_{\leq k-3}\|_{L_t^q L_x^{\frac{2q}{q-2}}} \|u_{> k-3}\|_{L_{t,x}^{2(p-1)}}^{p-1} \lesssim \eta_0^{p-1} 2^{-k(\frac2q-s_p)}\|u_{\leq k-3 }\|_{L_t^q L_x^{\frac{2q}{q-2}}}.
\end{align*}
Fix $C_0\geq 1$ to be determined below.  We write
\begin{align} 
\|u_{\leq k-3}\|_{L_t^q L_x^{\frac{2q}{q-2}}([1,2^{\alpha(k-k_0)}]\times\R^3)} & \lesssim \|u_{\leq C_0} \|_{L_t^q L_x^{\frac{2q}{q-2}}([1,2^{\alpha(k-k_0)}] \times\R^3)} \label{ss-lts-1} \\
& \quad + \| u_{C_0<\cdot\leq k_0} \|_{L_t^q L_x^{\frac{2q}{q-2}}([1,2^{\alpha(k-k_0)} ] \times\R^3)} \label{ss-lts-2} \\
& \quad + \sum_{k_0\leq j \leq k-3 } \|u_j\|_{L_t^q L_x^{\frac{2q}{q-2}}([1,2^{\alpha(k-k_0)}] \times\R^3)} \label{ss-lts-3}.
\end{align}
For \eqref{ss-lts-2}, we have
\[
\| u_{C_0<\cdot\leq k_0} \|_{L_t^q L_x^{\frac{2q}{q-2}}([1,2^{\alpha(k-k_0)}] \times\R^3)}  \lesssim C_0^{\frac2q-s_p}\log(2^{k-k_0}).
\]
On the other hand, for $C_0=C_0(\eta_0)$ large enough, we can estimate \eqref{ss-lts-3} by
\[
 \sum_{k_0\leq j \leq k-3 } \|u_j\|_{L_t^q L_x^{\frac{2q}{q-2}}([1,2^{\alpha(k-k_0)}] \times\R^3)} \lesssim \eta_0 2^{k_0(\frac2q-s_p)}\log(2^{k-k_0}).
\]
Finally, for $k_0\leq j \leq k-3$ we first use the inductive hypothesis to write
\[
\|P_j u\|_{L_t^q L_x^{\frac{2q}{q-2}}([1,2^{\alpha(j-k_0)} ]\times\R^3)} \lesssim 2^{j(\frac2q-s_p)}\eta_0. 
\]
Arguing as we did for the high-frequency piece,
\[
\|P_M u\|_{L_t^q L_x^{\frac{2q}{q-2}}([2^{\alpha(j-k_0)},2^{\alpha(k-k_0)}]\times\R^3)} \lesssim 2^{j(\frac2q-s_p)}\eta_0\log(2^{k-j}). 
\]
Thus
\begin{align*}
 \sum_{k_0\leq j \leq k-3 }& \|u_j\|_{L_t^q L_x^{\frac{2q}{q-2}}([1,2^{\alpha(k-k_0)} ]\times\R^3)}  \\
 &\lesssim \sum_{k_0\leq j\leq k-3 }\eta_0 2^{j(\frac2q-s_p)}[1+\log(2^{k-j})] \lesssim \eta_0 2^{k(\frac2q-s_p)},
\end{align*}
where we have used
\[
\sum_{L>1} L^{-(\frac2q-s_p)}\log(L) \lesssim 1. 
\]
Collecting these estimates, we find
\begin{align}\label{J-ss-lts-low}
&\|u_{\leq k-3}\|_{L_t^q L_x^{\frac{2q}{q-2}}([1,2^{\alpha(k-k_0)} ] \times\R^3)} \\
& \lesssim [C_0+\eta_0 2^{k_0(\frac2q-s_p)}]\log(2^{k-k_0}) + \eta_0 2^{k(\frac2q-s_p)},
\end{align}
which yields
\begin{align}
\||\nabla|^{-(\frac{2}{q}-s_p)}&P_{>k}(u_{\leq k-3} u_{>k-3}^{p-1})\|_{L_t^{\frac{2q}{q+2}}L_x^{\frac{q}{q-1}}}  \\
& \lesssim  \eta_0^{p-1} 2^{-k(\frac2q-s_p)} [C_0^{\frac2q-s_p}+\eta_0 2^{k_0(\frac2q-s_p)}]\log(2^{k-k_0}) + \eta_0^p.
\end{align}
Choosing $k_0$ possibly even larger, we deduce
\[
\||\nabla|^{-(\frac{2}{q}-s_p)}P_{>k}(u_{\leq k-3} u_{>k-3}^{p-1})\|_{L_t^{\frac{2q}{q+2}}L_x^{\frac{q}{q-1}}} \lesssim \eta_0^p. 
\]

Finally, we estimate Term III. Since $-\tfrac{2}{q}+s_p<0$, we use the fractional chain rule and Bernstein estimates to obtain
\begin{align*}
\| &|\nabla|^{-\frac4q+3s_p}P_{>k}(u_{\leq k-3 }^2P_{>k}(u^{p-2}))\|_{L_t^\gamma L_x^\rho} \\
& \lesssim 2^{-k\frac4q+2s_p}\|u_{\leq k-3}\|_{L_t^q L_x^{\frac{2q}{q-2}}}^2\| |\nabla|^{s_p}(u^{p-2})\|_{L_t^\infty L_x^{\frac{6(p-1)}{7p-15}}} \\
& \lesssim 2^{-2k(\frac{2}{q}-s_p)} \|u_{\leq k-3}\|_{L_t^q L_x^{\frac{2q}{q-2}}}^2 \| u\|_{L_t^\infty L_x^{\frac{3(p-1)}{2}}}^{p-3} \| |\nabla|^{s_c} u\|_{L_t^\infty L_x^2}. 
\end{align*}
Using \eqref{J-ss-lts-low} (and the conditions on $k_0,C_0$ given above), we conclude
\[
 \| |\nabla|^{-\frac4q+3s_p}P_{>k}(u_{\leq k-3}^2P_{>k}(u^{p-2}))\|_{L_t^\gamma L_x^\rho} \lesssim \eta_0^2.
\]
Combining our estimates for Terms I, II and III and choosing $\eta_0$ small, we conclude that 
\[
\| |\nabla|^{\frac{3(p-3)}{2(p-1)}} u_{>k} \|_{L_t^{2(p-1)}L_x^{\frac{2(p-1)}{p-2}}} + \| |\nabla|^{-(\frac2q-s_p)}u_{>k}\|_{L_t^q L_x^{\frac{2q}{q-2}}} \leq \eta_0
\]
on $[1,2^{\alpha(k-k_0)}]\times\R^3$, thus closing the induction and completing the proof. \end{proof}

Finally, we arrive at the proof of the additional regularity Proposition \ref{ad_ref_ss}.
\begin{proof}[Proof of \protect{Proposition~\ref{ad_ref_ss}}]
We compute
\begin{align}
\|\vec u(1)\|^2_{\H^{s_0}} \simeq \| w (1) \|^2_{\dot H^{s_0}} \lesssim \sum_{k \gg 1} 2^{2ks_0} \langle P_k w(1), P_k  w(1)\rangle .
\end{align}
We use the double Duhamel argument based at $t=1$. For some $k\gg 1$, we write
\[
\langle P_k w(1), P_k  w(1)\rangle = \int_0^1 \int_1^\infty \langle e^{i(1-t)\sqrt{-\Delta}}P_k F(u(t)), e^{i(1-\tau)\sqrt{-\Delta}} P_k F(u(\tau))\rangle \,d\tau \,\ud t.
\]
We fix $\alpha \geq 1$, to be determined below, and split
\[
\int_1^\infty e^{i(1-t)\sqrt{-\Delta}}P_k F(u(t))\,\ud t  = A_k + B_k,
\]
where
\begin{align*}
A_k =\int_1^{2^{k\alpha }} e^{i(1-t)\sqrt{-\Delta}}P_k F(u(t))\,\ud t,\quad B_k&=\int_{2^{k\alpha}}^\infty e^{i(1-t)\sqrt{-\Delta}}P_k F(u(t))\,\ud t.
\end{align*}
We also write
\[
 \int_0^1 e^{i(1-\tau)\sqrt{-\Delta} }P_k F(u(s))\,ds  =  Z_k,
\]
We will use the estimate
\[
 |\langle A_k+B_k, Z_k\rangle| \leq |A_k|^2+2|\langle B_k,Z_k\rangle|,
\]
which follows from the fact that $A_k+B_k=Z_k$. 

\medskip
We first estimate the $\langle B_k,Z_k\rangle$ term. We expand
\[
|\langle B_k,Z_k\rangle| \leq \sum_{\ell \leq 0} \sum_{ j \geq k\alpha} \int_{2^{\ell}}^{2^{\ell+1}} \hspace{-2mm}  \int_{2^j}^{2^{j+1}} \bigl| \langle e^{-i(t-\tau)\sqrt{-\Delta} }P_k F(u(t)), P_k F(u(s))\rangle\bigr| \,d\tau\,\ud t.
\]
 We claim that
\begin{equation}\label{01071}
\|P_k(u|u|^{p-1})\|_{L_t^2 L_x^1([2^{\ell},2^{\ell+1}]\times\R^3)} \lesssim 2^{-ks_p}
\end{equation}
uniformly in $\ell \geq 0$. Indeed, arguing as above, we can decompose the nonlinearity into two types of terms
\[
u_{> k-1}u^{p-1} \qquad \textup{and} \qquad uP_{> k-1}(u^{p-1}) ,
\]
since if both $u$ and $|u|^{p-1}$ are projected to low frequencies, the product vanishes when projected to high frequencies. 

We thus have by Bernstein's inequality, H\"older's inequality, and the fractional chain rule that 
\begin{align*}
\|&P_k(u|u|^{p-1})\|_{L_t^2 L_x^1} \\ & \lesssim \|u\|_{L_{t,x}^{2(p-1)}}^{p-1}\|u_{>k-1}\|_{L_t^\infty L_x^2} + \|u\|_{L_{t,x}^{2(p-1)}}\|P_{>k-1}(u^{p-1})\|_{L_t^{\frac{2(p-1)}{p-2}} L_x^{\frac{2(p-1)}{2p-3}}} \\
& \lesssim 2^{-ks_p} \|u\|_{L_{t,x}^{2(p-1)}}^{p-1}\| |\nabla|^{s_p}u\|_{L_t^\infty L_x^2} \lesssim 2^{-ks_p},
\end{align*}
where all spacetime norms are over $[2^{\ell}, 2^{\ell+1}]\times\R^3$. 

Using dispersive estimates, we have for any $j \geq k \alpha$ and $\ell \leq 0$,
\begin{align*}
 \int_{2^{\ell}}^{2^{\ell+1}} & \int_{2^j}^{2^{j+1}} \bigl| \langle e^{-i(t-\tau)\sqrt{-\Delta} }P_k F(u(t)), P_k F(u(s))\rangle\bigr| \,d\tau\,\ud t \\
& \lesssim  \int_{2^{\ell}}^{2^{\ell+1}} \int_{2^j}^{2^{j+1}} t^{-1} 2^k \|P_k(u|u|^{p-1})(t)\|_{L_x^1} \|P_k(u|u|^{p-1)}(\tau)\|_{L_x^1}\,d\tau\,\ud t \\
& \lesssim 2^{\frac \ell 2} 2^{-\frac j 2} \|P_k(u|u|^{p-1})\|_{L_t^2 L_x^1([2^{\ell},2^{\ell+1}]\times\R^3)}\|P_k(u|u|^{p-1})\|_{L_t^2 L_x^1([2^k,2^{k+1}]\times\R^3)} \\
& \lesssim 2^{\frac \ell 2} 2^{-\frac j2} 2^{k(1-2s_p)}. 
\end{align*}
 Summing over $\ell \leq 0$ and $j \geq  k\alpha $, we deduce that
 \begin{equation}\label{ss-yz}
 |\langle B_k , Z_k \rangle_{\dot H^{s_p}_x}| \lesssim 2^{k(1-\frac{\alpha}{2})}.
 \end{equation}

We now turn to estimating the $|A_k|^2$ term.  We will use a frequency envelope argument to establish the required bounds. Once again, we fix an exponent $q$ satisfying
\[
2 < q < \frac{s_p}{2}.
\]
Let
\begin{align}\label{sigma}
\sigma < \min\left\{s_p,\frac2q-s_p,\frac4q-1-s_p \right\}. 
\end{align}
and define
\[
\gamma_k = \sum_{j} 2^{-\sigma|j-k|} \|w_j\|_{L_t^\infty \dot H^{s_p}([1,\infty)\times\R^3)}.
\]
We will establish the following: Let $\eta_0>0$ and let $R(\cdot)$ denote the compactness modulus function of $\vec u$. Then there exists $k_0\equiv k_0(\eta_0,R(\eta_0))$ sufficiently large that
\begin{align}\label{ss-ak}
\|A_k\|_{\dot H^{s_p}_x}\lesssim C(k_0)2^{-k (\frac2q-s_p)+} + \eta_0\sum_{j} 2^{-\sigma |j-k|} \|u_j\|_{L_t^\infty \dot H_x^{s_p}([1,\infty)\times\R^3)}
\end{align}
for all $k \geq k_0$. For $p > 3$, we write the nonlinearity as $F(u)=|u|^{p-3}u^3$ and then decompose $u^3$ by writing $u=u_{\leq k}+u_{>k}$.  By further decomposing $u_{\leq k}=u_{\leq k_0}+u_{k_0<\cdot\leq k}$, we are led to terms of the form
\begin{align}
F(u)&=|u|^{p-3}u_{>k}^3 \label{HHH}\\
& \quad + 3|u|^{p-3}u_{>k}^2 u_{\leq k_{0}} \label{HHL}  \\
& \quad + 3|u|^{p-3}u_{>k}^2 u_{k_0<\cdot\leq k} \label{HHM} \\
& \quad + |u|^{p-3} u_{\leq k_{0}} F_3 \label{MML}  \\
& \quad +   |u|^{p-3}u_{k_0<\cdot\leq k}^3 \label{MMM}\\
& \quad + 3|u|^{p-3}u_{>k} u_{\leq k}u_{\leq k_0} \label{HML} \\
& \quad + 3|u|^{p-3}u_{>k} u_{\leq k}u_{k_0<\cdot\leq k}, \label{HMM} 
\end{align}
where we have written
\[
F_3 = u_{\leq k_0}^2 + 2u_{\leq k_0}u_{k_0<\cdot\leq k}+ u_{k_0<\cdot\leq k}^2. 
\]
 
 By Proposition~\ref{prop-ss-lts}, for any $\beta \geq 1$ there exists $k_0\equiv k_0(R(\eta_0),\beta)$ so that  for every $k > k_0$ we have
\begin{align}
 \| |\nabla|^{\frac{3(p-3)}{2(p-1)}}& u_{>k}\|_{L_t^{2(p-1)}L_x^{\frac{2(p-1)}{p-2}}([1,2^{\beta(k-k_0)}]\times\R^3)}\\
 &+ \| |\nabla|^{-(\frac2q-s_p)}u_{>k}\|_{L_t^q L_x^{\frac{2q}{q-2}}([1,2^{\beta(k-k_0)}]\times\R^3)}<\eta_0.
\end{align}
Fix $\beta>\alpha$ and $k_1=k_1(R(\eta_0),\alpha,\beta)\geq  k_0$ which satisfies
\[
2^{k_1(\beta - \alpha)}\geq 2^{k_0 \beta}. 
\]
Then $2^{\beta(k- k_0)} \geq 2^{k \alpha}$ for $k \geq k_1$, and hence for every $k \geq k_1$, we have
\begin{equation}\label{cor-ss-lts}
\begin{split}
\| |\nabla|^{\frac{3(p-3)}{2(p-1)}} &u_{>k}\|_{L_t^{2(p-1)}L_x^{\frac{2(p-1)}{p-2}}([1,2^{\alpha k}]\times\R^3)}\\
&+\| |\nabla|^{-(\frac2q-s_p)}u_{>k}\|_{L_t^q L_x^{\frac{2q}{q-2}}([1,2^{\alpha k}]\times\R^3)} <\eta_0. 
\end{split}
\end{equation}
We will use this estimate repeatedly below. Furthermore, we may also establish identical long-time Strichartz estimates for
\[
\| |\nabla|^{s_p-\frac2r} u_{>k}\|_{L_t^r L_x^{\frac{2r}{r-2}}},
\]
where $\tfrac2{s_p}<r<4$.


To estimate \eqref{HHH}, we use the dual Strichartz pair
\[
\left(\frac{r}{2},\frac{6r(p-1)}{12-12p-21r+13pr}\right),
\]
with $\frac2{s_p}<r<4$.  We note that this pair is dual admissible: writing the pair as $(A,B)$, we have $\tfrac{1}{A}+\tfrac{1}{B}=\tfrac{13p-21}{6p-6}\geq \tfrac32$ for $p\geq 3$.  Note that $A\in(1,2)$ since $r\in(2,4)$ and $B>1$ for $r<\tfrac{12(p-1)}{7p-15}$.  This is compatible with $r>\tfrac{2}{s_p}$ when $p\in[3,5)$. We can thus bound
\begin{align*}
2^{k(3s_p-\frac4r)}&\|u\|_{L_t^\infty L_x^{\frac{3(p-1)}{2}}}^{p-3}\|u_{>k}\|_{L_t^r L_x^{\frac{2r}{r-2}}}^2 \sum_{k\leq j} \|u_{j}\|_{L_t^\infty L_x^2} \\
& \lesssim \| |\nabla|^{s_p-\frac2r}u_{>k}\|_{L_t^r L_x^{\frac{2r}{r-2}}}^2\sum_{k\leq j} 2^{(k-j)s_p}\|u_j\|_{L_t^\infty \dot H_x^{s_p}} \\
& \lesssim \eta_0 \sum_{j>k}2^{(k-j) s_p}\|u_j\|_{L_t^\infty \dot H_x^{s_p}}. 
\end{align*}

For \eqref{HHL}, we use the dual Strichartz pair
\begin{align}\label{eq:hhl_pair}
\left(\frac{2q(p-1)}{2p+q-2},\frac{6q(p-1)}{6-15q+2p(5q-3)}\right).
\end{align}
We bound the contribution of this term by 
\begin{align*}
2^{k(2s_p-\frac2q)}&\|u\|_{L_t^\infty L_x^{\frac{3(p-1)}{2}}}^{p-3}\|u_{>k}\|_{L_{t,x}^{2(p-1)}} \|u_{\leq k_0}\|_{L_t^q L_x^{\frac{2q}{q-2}}} \|u_{>k}\|_{L_t^\infty L_x^2} \\
& \lesssim \eta_0 2^{-k(\frac2q-s_p)}2^{k_0(\frac2q-s_p)}\log 2^k \|u\|_{L_t^\infty \dot H_x^{s_p}} \\
& \lesssim \eta_0 2^{-k(\frac2q-s_p)}2^{k_0(\frac2q-s_p)}\log 2^k.
\end{align*}
For \eqref{HHM}, we use the same dual pair as in \eqref{eq:hhl_pair}, and we obtain
\begin{align*}
2^{k(2s_p-\frac2q)}&\|u\|_{L_t^\infty L_x^{\frac{3(p-1)}{2}}}^{p-3}\|u_{>k}\|_{L_{t,x}^{2(p-1)}} \sum_{k_0\leq j_1\leq k\leq j_2} \|u_{j_1}\|_{L_t^q L_x^{\frac{2q}{q-2}}} \|u_{j_2}\|_{L_t^\infty L_{x}^{2}}\\
& \lesssim \eta_0 \sum_{k_0\leq j_1\leq k\leq j_2} 2^{j_1(\frac2q-s_p)}\log(2^{k-j_1}) 2^{-j_2s_p}\|u_{j_2}\|_{L_t^\infty \dot H_x^{s_p}} \\
& \lesssim \eta_0\sum_{k\leq j}2^{(k-j)s_p}\|u_j\|_{L_t^\infty \dot H_x^{s_p}}. 
\end{align*}

To estimate \eqref{MML}, we use the admissible dual pair $(\frac{q}{2},\frac{6q}{7q-8}+)$. We choose $\rho$ so that
\[
\frac{3}{\rho}=\frac{2}{q}+\frac{4}{p-1}-\frac{3}{2}.
\]
We bound the contribution of this term by
\begin{align*}
2^{-k(\frac2q-s_p)}&\|u\|_{L_t^\infty L_x^{\frac{3(p-1)}{2}}}^{p-3}\|u_{\leq k_0}\|_{L_t^q L_x^{\frac{2q}{q-2}}} \sum_{j_1\leq j_2\leq k} \|u_{j_1}\|_{L_t^q L_x^{\frac{2q}{q-2}}}\|u_{j_2}\|_{L_t^\infty L_x^{\rho+}} \\
& \lesssim 2^{-k(\frac2q-s_p)}2^{k_0(\frac2q-s_p)}\log 2^k \sum_{j_1\leq j_2\leq k} 2^{j_1(\frac2q-s_p)}2^{j_2(s_p-\frac2q+)}\log(2^{k-j_1}) \\
& \lesssim 2^{-k(\frac2q-s_p)+}.
\end{align*}

Now we estimate \eqref{MMM} as follows: we apply Strichartz estimates with the dual (sharp) admissible pair $(\tfrac{q}{2},\tfrac{2q}{3q-4})$. 	Then we obtain
\begin{align*}
2^{k(1-\frac4q+s_p)}\|u\|_{L_t^\infty L_x^{\frac{3(p-1)}{2}}}^{p-3} \sum \|u_{j_1}\|_{L_t^q L_x^{\frac{2q}{q-2}}} \|u_{j_2}\|_{L_t^q  L_x^{\frac{2q}{q-2}}}\|u_{j_3}\|_{L_t^\infty L_x^{\frac{6(p-1)}{9-p}}},
\end{align*}
where the sum is over $k_0\leq j_1\leq j_2\leq j_3\leq k$. 

Now, for $k_0\leq j\leq k$, we can estimate
\begin{align*}
\|u_{j}\|_{L_t^q L_x^{\frac{2q}{q-2}}([1,2^{\alpha k}]\times\R^3)} & \lesssim\|u_{j}\|_{L_t^q L_x^{\frac{2q}{q-2}}([1,2^{\alpha j}]\times\R^3)}+ \|u_j\|_{L_t^q L_x^{\frac{2q}{q-2}}([2^{\alpha k},2^{\alpha k}]\times\R^3)} \\
& \lesssim \eta_0 \log(2^{k-j}) 2^{j(\frac2q-s_p)},
\end{align*}
using the long-time Strichartz estimate of Proposition \ref{prop-ss-lts}  and we note the $\log$ comes from the second term. We also have
\[
\|u_j\|_{L_t^\infty L_x^{\frac{6(p-1)}{9-p}}} \lesssim 2^{-j(1-s_p)}\|u_j\|_{L_t^\infty \dot H_x^{s_p}}. 
\] 
This yields
\begin{align*}
\eta_0 2^{k(1-\frac4q+s_p)}& \sum_{k_0\leq j_1\leq j_2\leq j_3\leq k} 2^{j_1(\frac2q-s_p)}\log(2^{k-j_1}) 2^{j_2(\frac2q-s_p)}\log(2^{k-j_2})2^{-j_3(1-s_p)} \\
& \lesssim \eta_0 \sum_{k_0\leq j\leq k} \log(2^{j-k}) 2^{(j-k)(\frac4q-1-s_p)}\|u_j\|_{L_t^\infty \dot H_x^{s_p}}. 
\end{align*}
Note that for this estimate, we need 
\[
q<\frac{4}{1+s_p} = \frac{8(p-1)}{5p-9},
\]
which is compatible with $q>2$ for $p\in[3,5)$.

For \eqref{HML}, we use the dual Strichartz pair
\begin{align}\label{equ:hml_pair}
\left(\frac{q}{2},\frac{6(pq-q)}{12-12p-21q+13pq}\right).
\end{align}
We bound the contribution of this term by 
\begin{align*}
 2^{k(3s_p-\frac4q)}&\|u_{\leq k_0}\|_{L_t^q L_x^{\frac{2q}{q-2}}} \sum_{j_1\leq k \leq j_2} \|u_{j_1}\|_{L_t^q L_x^{\frac{2q}{q-2}}} \|u_{j_2}\|_{L_t^\infty L_x^2} \\
& \lesssim 2^{k(3s_p-\frac4q)} 2^{k_0(\frac2q-s_p)} \sum_{j_1\leq k\leq j_2} 2^{j_1(\frac2q-s_p)}\log(2^{k-j_1}) 2^{-j_2s_p} \|u_{j_2}\|_{L_t^\infty \dot H_x^{s_p}} \\
& \lesssim 2^{-k(\frac2q-s_p)}2^{k_0(\frac2q-s_p)}. 
\end{align*}

Finally, for \eqref{HMM}, using the same dual pair as in \eqref{equ:hml_pair}, and we estimate the contribution of this term by
\begin{align*}
2^{k(3s_p-\frac4q)}& \|u\|_{L_t^\infty L_x^{\frac{3(p-1)}{2}}}^{p-3} \sum_{j_1, k_0\leq j_2\leq k\leq j_3} \hspace{-4mm}\|u_{j_1}\|_{L_t^q L_x^{\frac{2q}{q-2}}} \|u_{j_2}\|_{L_t^q L_x^{\frac{2q}{q-2}}} \|u_{j_3}\|_{L_t^\infty L_x^2} \\
& \lesssim \eta_0 2^{k(3s_p-\frac4q)} \sum_{j_1\leq j_2\leq k \leq j_3} \hspace{-4mm} 2^{k_1(\frac2q-s_p)}\log(2^{2k-j_1 - j_2})2^{j_2(\frac2q-s_p)}2^{-j_3 s_p}\|u_{j_3}\|_{L_t^\infty \dot H_x^{s_p}} \\
& \lesssim \eta_0\sum_{k\leq j} 2^{(k-j)s_p}\|u_j\|_{L_t^\infty \dot H_x^{s_p}}. 
\end{align*}

Putting together all the estimates, we establish \eqref{ss-ak}, which, together with \eqref{ss-yz} and the conditions on $\sigma$ from \eqref{sigma} yields
\[
\|w_k(1)\|_{\dot H^{s_p}_x}  \lesssim 2^{k(\frac12-\frac{\alpha}{4})}+ 2^{-k(\frac2q-s_p)+} + \eta_0\sum_{j} 2^{-\sigma|j-k|} \|w_j\|_{L_t^\infty \dot H_x^{s_p}([1,\infty)\times\R^3)} 
\]
for all $k \gg 1$.  For $\alpha$ large enough, we can guarantee that the the second term dominates the first, and hence
\[
\|w_k(1)\|_{\dot H^{s_p}_x}  \lesssim 2^{-k(\frac2q-s_p)+} + \eta_0\sum_{j} 2^{-\sigma|j-k|} \|w_j\|_{L_t^\infty \dot H_x^{s_p}([1,\infty)\times\R^3)}
\]
for all $k \gg 1$. We now rescale the solution $u$ and use the fact that the rescaled solution $Tu(Tt,Tx)$ is also a self-similar solution for any $T>1$ (with the {same} compactness modulus function as $u$). This yields
\begin{equation}\label{01072}
\|w_k\|_{L_t^\infty \dot H^{s_p}([1,\infty)\times\R^3)} \lesssim 2^{-k(\frac2q-s_p)+} + \eta_0 \sum_{j} 2^{-\sigma|j-k|} \|w_j\|_{L_t^\infty \dot H_x^{s_p}([1,\infty)\times\R^3)}.
\end{equation}
Let $0<\eta< \sigma$. Then \eqref{01072} implies that for $k \gg k_0$,
\[
\gamma_k \lesssim 2^{-k\eta} +\eta_0\alpha_k,
\]
and hence, we may conclude that 
\[
\| w(1)\|_{\dot H^{s_p+\eta}} \lesssim 1 \qtq{for any}0<\eta<\sigma. 
\]
Using the same rescaling argument as above, and the relation between $w$ and $u$, we ultimately deduce that
\[
\| u(T)\|_{\H^{s_p+\delta}} \lesssim T^{-\eta},
\]
which yields \eqref{s0pluseps} provided we can choose $\eta=\frac{5-p}{2p(p-1)}$.  Combining with the constraint $\eta<\tfrac{2}{q}-s_p$, this requires that we choose
\[
2<q<\frac{4p}{3p-5},
\]
which is possible whenever $p\in[3,5)$. For the other term appearing in the definition of $\sigma$,  we find that we can choose $\eta=\frac{5-p}{2p(p-1)}$ provided we take 
\[
q<\frac{8p}{5(p-1)},
\]
which is similarly allowable by the requirement that $q>2$ for $p\in[3,5)$. This completes the proof of Proposition~\ref{ad_ref_ss} and hence completes our treatment of the self-similar scenario.
\end{proof}

\section{Doubly concentrating critical element: the sword and shield} \label{s:sword} 
We now consider the case of the doubly concentrating critical element, that is, $N(t) \geq 1$ on $\mathbb{R} = I_{\textup{max}}$ and
\EQ{ \label{eq:Ntoinf} 
\limsup_{t \to \pm \infty} N(t) = \infty.
}

By Proposition~\ref{p:cases} we may assume in this case that  $x(t)$ is subluminal in the sense of Definition~\ref{d:sl}. By Lemma~\ref{l:sl} there exists $\de_0>0$ so that 
\begin{align} \label{eq:sl}
|x(t) - x(\tau)| \leq (1- \delta_0)|t - \tau| , 
\end{align}
for all $t, \tau$ with 
\EQ{
|t - \tau| \geq \frac{1}{\delta_0 \inf_{s \in [t, \tau]} N(s)}. 
} 

The goal of this section is to prove the following proposition:
\begin{prop}
There are no doubly-concentrating critical elements, in the sense of Case (III) of Proposition~\ref{p:cases}.
\end{prop}

To prove this proposition, we establish the following dichotomy: either additional regularity for the critical element can be established using essentially the same arguments used in Section~\ref{s:soliton}, or a self-similar-like critical element can be extracted by passing to a suitable limit. To this end we define function 
$\tau : \mathbb{R} \rightarrow \mathbb{R}$ by
\[
\tau(t) = \int_{0}^{t} N(s) \ud s.
\]
Since $N(t) > 0$ and $\lim_{t \to \pm \infty} \tau(t) = \infty$, the function $\tau : [0,\infty) \to [0,\infty)$ is bijective. Hence for any $t_0 > 0 $ and any  $C_+ > 0$, there exists a unique $\ka_+ =  \ka_+( t_0, C_+) > 0$ such that
\[
t_0 + \frac{\ka_{+}( t_0,  C_{+})}{N(t_0)}= \tau^{-1}(\tau(t_{0}) +  C_{+}).
\]
Similarly, for $t_0 < 0$ and any  $ C_- > 0$, we can define
\[
t_0 - \frac{\ka_{-}( t_0,  C_{-})}{N(t_0)} = \tau^{-1}(\tau(t_{0}) - C_{-}).
\]
Fix $\eta > 0$ as in the small data theory of Proposition~\ref{small data}, and let $R = R(\eta)$ be such that for all $t \in \mathbb{R}$,
\begin{align}\label{equ:compactness_small_data}
&\int_{|x -  x(t)| \geq \frac{R(\eta)}{N(t)}}||\nabla|^{s_p} u(t,x)|^{2} \ud x + \int_{|x-x(t)| \geq \frac{R(\eta)}{N(t)}}||\nabla|^{s_p-1} u_t(t,x)|^{2} \ud x   \leq  \eta,
\end{align} 
see Remark \ref{r:Reta}. Now let $\chi(t) = \chi_{R, N}(t)$ be a smooth cutoff to the set
\[
\{ |x - x(t) | \geq R(\eta) / N(t) \}.
\]
By our choice of $R(\eta)$ we have 
\[
\| \chi(t) \vec u \|_{\dot \cH^{s_p}}^2 \lesssim \eta.
\]

Since $N(t) \geq 1$ and by~\eqref{eq:sl}, for any $t_{0}$, there exists $C_{+}(t_{0}) \ge 1$ sufficiently large so that
\begin{align} \label{eq:Cpmdef} 
\biggl|x\Big(t_{0} +& \frac{\ka_{+}(t_{0}, C_{+}(t_{0}))}{ N(t_{0})}\Big) - x(t_{0})\biggr|\\
 &\leq \abs{\frac{\ka_{+}(t_{0}, C_{+}(t_{0}))}{N(t_{0})}} - \frac{R(\eta)}{N(t_0 + \ka_+(t_0, C_+(t_0)) N(t_0)^{-1})},
\end{align}
and similarly for $C_-(t_0)$. By continuity we may assume that $C_{\pm}(t_0)$ are minimal with this property. 
Furthermore, for every $t_{0}$ there exists $C(t_{0})$ such that for some $t_{1} \in \R$ satisfying
\begin{equation}
\tau(t_{1}) - \tau(t_{0}) \leq C(t_{0}),
\end{equation}
there exist $t_{-} < t_{1} < t_{+}$ with
\begin{equation}
\tau(t_{1}) - \tau(t_{-}) \leq 2 C(t_{0}), \quad \tau(t_{+}) - \tau(t_{1}) \leq 2 C(t_{0}),
\end{equation}
which satisfies
\begin{equation}
|x(t_{-}) - x(t_{1})| \leq |t_{-} - t_{1}| - \frac{R(\eta)}{N(t_{-})}, \quad \text{and} \quad |x(t_{+}) - x(t_{1})| \leq |t_{+} - t_{1}| - \frac{R(\eta)}{N(t_{+})}.
\end{equation}
It is clear from the definition that $C(t_{0}) \leq \sup(C_{+}(t_{0}), C_{-}(t_{0}))$, and thus is finite. However, $C_{\pm}(t_{0})$ need not be uniformly bounded for $t_{0} \in \R$, and hence neither does $C(t_0)$.

We will now analyze several cases based on whether $C(t_0)$ are uniformly bounded for $t_0 \in \mathbb{R}$. 

\subsection{Case 1: \texorpdfstring{$C(t_0)$}{C} are uniformly bounded}  Here we work under the assumption that there exists a constant $C > 0$ such that $C(t_0) \leq C$ for all $t_0 \in \mathbb{R}$. 

We show that essentially the same argument used in Section~\ref{s:soliton-reg} can be used to show that such a critical element necessarily has the compactness property in $\dot \HH^{s_p} \cap \dot \HH^1$.

\begin{prop}[Additional regularity] \label{p:casc-reg}  Let $\vec u(t) \in \dot \HH^{s_p}$ be a solution with the compactness property that is subluminal and doubly concentrating, as in Case (III) of Proposition~\ref{p:cases}. Assume in addition that $C(t)$ is uniformly bounded as a function of  $t \in \R$. Then $\vec u(t) \in \dot \HH^1$ and satisfies the bound,  
\EQ{ \label{eq:h1-regN} 
  \| \vec u(t) \|_{\dot \HH^1} \lesssim N(t)^{\frac{5-p}{2(p-1)}}
}
uniformly in $t \in \R$. 
\end{prop} 

%

For the moment, we will assume Proposition \ref{p:casc-reg}, and we will use it to prove the following corollary.

\begin{cor} \label{c:casc-reg} 
Let $\vec u(t)$ satisfy the hypotheses of Proposition~\ref{p:casc-reg}. Then $\vec u(t) \equiv 0$. 
\end{cor} 

\begin{proof}[Proof of Corollary~\ref{c:casc-reg} assuming Proposition~\ref{p:casc-reg}] 
We begin by extracting from $\vec u(t)$ another solution with the compactness property on a half-infinite time interval $[0, \infty)$ but with new scaling parameter $\ti N(s)   \to 0$ as $s \to \infty$. Let $\ti t_m$ be any sequence of times with 
\EQ{
\ti t_m \to - \infty, \quad N(\ti t_m) \to \infty \mas m \to \infty.
}
Next choose another sequence $t_m \to - \infty$ by choosing $t_m$ such that 
\EQ{
N(t_m) := \max_{t \in [\ti t_m, 0]} N(t).
}
Now define a sequence as follows: set 
\EQ{
&w_m(s, y) :=    \frac{1}{N( t_m)^{\frac{2}{p-1}}} u (  t_m + \frac{s}{N(t_m)}, x( t_m) + \frac{y}{N( t_m)}),   \\
&\p_t w_m(s, y):=  \frac{1}{N( t_m)^{\frac{2}{p-1} +1}} \p_t u ( t_m + \frac{s}{N(t_m)}, x( t_m) + \frac{y}{N(t_m)}), 
}
and set 
\EQ{
\vec w_m := ( w_m(0, y),  \p_t w_m(0, y)).
}
Then by the pre-compactness in $ \dot{\HH}^{s_p} $, there exists (after passing to a subsequence)  $\vec w_{\infty}(y)  \neq 0$ so that 
\EQ{ \label{eq:strongw} 
 \vec{w}_m \to \vec{w}_{\infty} \in \dot{\HH}^{s_p}.
} 
It is standard to show  that  $\vec w(s)$ (the evolution of $\vec w_{\infty} = \vec w(0)$)  has the compactness property on $I = [0, \infty)$  with frequency parameter  $ \ti N(s)$  defined by 
 \EQ{
 \ti N(s):= \lim_{m \to \infty} \frac{N( t_m + \frac{s}{N(t_m)})}{ N( t_m)},
}
and moreover that 
\EQ{ \label{bigN}
&\ti N(s)  \le 1   \quad \forall s \in [0, \infty), \\
&\liminf_{s \to \pm \infty} \ti N(s) = 0.
} 
By the uniform bounds of~\eqref{eq:h1-regN}, we see that 
\EQ{
\| \vec w(s) \|_{\dot\HH^1} \lesssim \ti N(s)^{\frac{5-p}{2(p-1)}} \quad \forall \, s \in [0, \infty),
}
and hence there exists a sequence of times $s_n \to \infty$ along which 
\EQ{
\| \vec w(s_n) \|_{\dot\HH^1} \lesssim \ti N(s_n)^{\frac{5-p}{2(p-1)}} \to 0 \mas n \to \infty.
}
Using the above, Sobolev embedding, and interpolation,  along the same sequence of times we have 
\EQ{
\| w(s_n) \|_{L^{p+1}} \lesssim \| w(s_n) \|_{\dot H^{\frac{3(p-1)}{2(p+1)}}}   \to 0 \mas n \to \infty.
}
But then, since the energy of $\vec w(s)$ is well-defined and conserved, we must have 
\EQ{
E( \vec w) = 0.
}
For the defocusing equation we may immediately conclude that $\vec w(s) \equiv 0$. \end{proof} 

\begin{rem}
As in Section \ref{s:soliton}, these arguments readily adapt to the focusing setting.
\end{rem}

\begin{proof}[Sketch of the proof of Proposition~\ref{p:casc-reg}] The argument is nearly identical to the proof of Proposition~\ref{p:sol-reg} in Section~\ref{s:soliton}, hence rather than repeat the entire proof, we instead summarize how the uniform boundedness of the numbers $C(t_0)$ allow us to proceed as in Section~\ref{s:soliton-reg}. The main idea is that the boundedness of these constants  means that for each $t_0 \in \R$ we only have to  wait a uniformly bounded amount of time,  where time is measured relative to the scale $N(t)$), for the forward and backwards light cones based at $(t_0, x(t_0))$) to capture the bulk of the solution.cConsequently, we can apply the same techniques that were developed in Section~\ref{s:soliton-reg} directly and implement a double Duhamel argument. In order to estimate the norm at a time $t=t_0$, we recall the definitions of $t_1,t_\pm$ above and decompose space-time into three regions:
\begin{enumerate}
\item[A)] \textbf{Region A}: $[t_-, t_+] \times \mathbb{R}^3$,
\item[B)] \textbf{Region B}: the forward (resp. backward) light-cones from 
\[
\{t_+\} \times \{x \,:\, |x - x(t_1)| \geq |t_+ - t_1|\},
\]
and
\[
  \{t_-\} \times \{x \,:\, |x - x(t_1)| \geq |t_- - t_1|\},
\]
\item[C)] \textbf{Region C}: $\mathbb{R} \times \mathbb{R}^4 \,\setminus$ (\textbf{Region A} $\cup$ \textbf{Region B}).
\end{enumerate}

\begin{figure}[h]
  \centering
  \includegraphics[width=14cm]{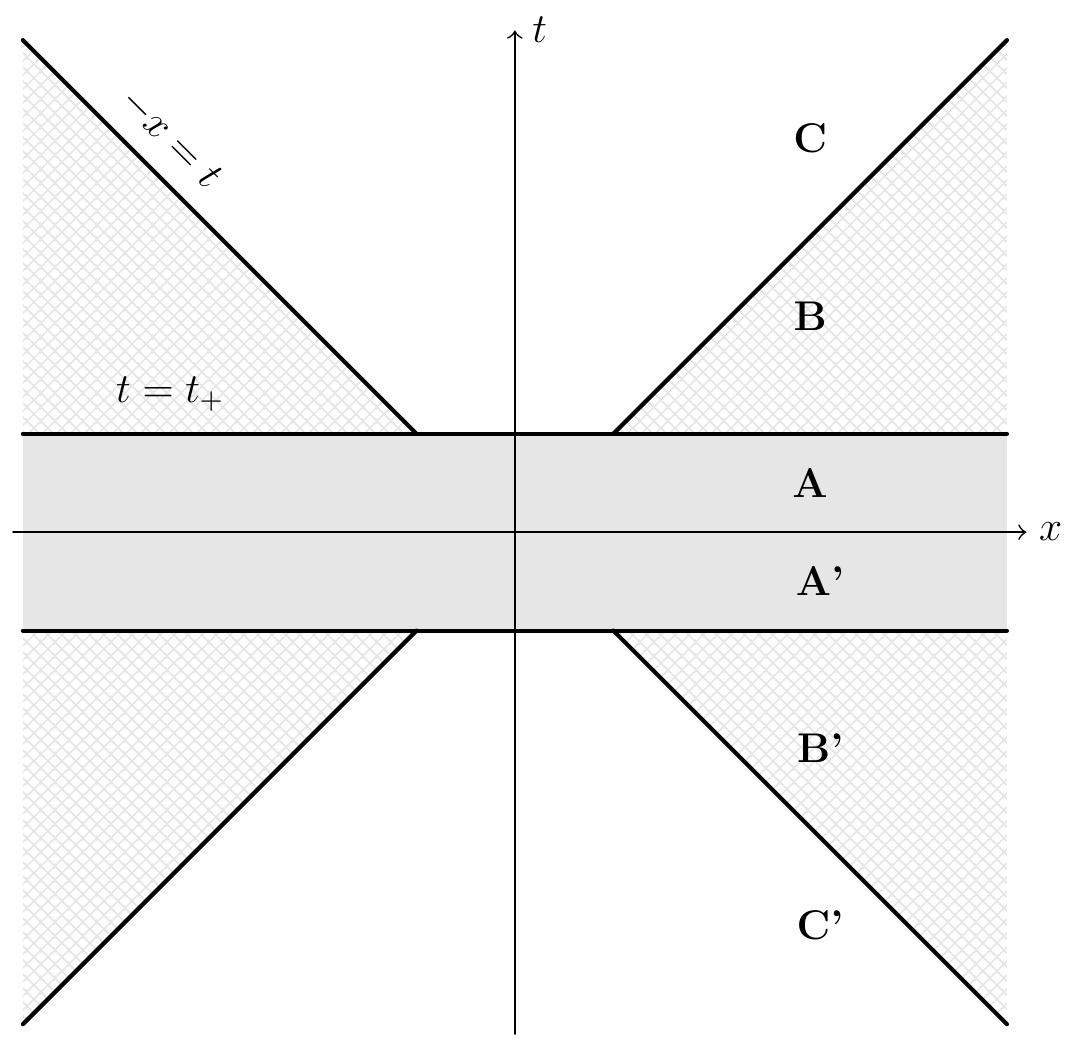}
  \caption{A depiction of the regions \textbf{A, B} and \textbf{C} in the case that $x(t) = 0$.} \label{fig:regions_swordandshield}
\end{figure}

On \textbf{Region A}, we control the solution by dividing the time interval $[t_-, t_+]$ into finitely many sufficiently small time strips on which we can use Lemma~\ref{l:DL1lem}.

The main difficulty is that we need to ensure that we can uniformly control the number of small strips we will need to accomplish this (this type of uniform control was guaranteed  in the Section~\ref{s:soliton}  because we had $N(t) \equiv 1$ there). Here, the boundedness of the constants $C$ is used to achieve this uniformity. 

  From Lemma~\ref{l:DL1lem} we know that for each $\eta>0$ there exists $\delta>0$ such that for all $t \in \R$
 \EQ{
 \|u\|_{L^{2(p-1)}_{t,x}([t - \frac{\delta}{N(t)}, t+ \frac{\delta}{N(t)}] \times \mathbb{R}^3)}  \leq \eta\quad \forall t \in \R.
 }
 Fix this $\delta > 0$. Examining the proof of the estimates used to control the solution on \textbf{Region A} in Section~\ref{s:soliton}, see \eqref{eq:regions}, we need to show that there exists a uniformly (in $t_0$) bounded number $M>0$ of times $t_m$, $-M \le m \le M$ with  $t_- \le t_{m}\le t_+$,  and such that the corresponding intervals $I_{-M}, \dots I_{M}$ with 
 \EQ{
 I_M := [t_m - \frac{\delta}{N(t_m)}, t_m+ \frac{
 \delta}{N(t_m)}]
 }
 satisfy \EQ{
 [t_-, t_+] \subset \bigcup_{m = -M}^M I_m.
 }
 In this case we obtain
\[
\|u\|^{2(p-1)}_{L^{2(p-1)}_{t,x}([t_-, t_+] \times \mathbb{R}^3)} \lesssim \sum_{i=1}^M \|u\|^{2(p-1)}_{L^{2(p-1)}_{t,x}(I_i \times \mathbb{R}^3)} \lesssim \int_{t_-}^{t_+} N(t)\,\ud t,
\]
and, since
\begin{align}\label{equ:int_bds}
 \int_{t_-}^{t_+} N(t) \, \ud t  = \tau(t_+) - \tau(t_-) \le 4C 
\end{align}
by construction, this would yield the desired upper bound. 

Hence, we now turn to the argument that intervals on which we can control the $L^{2(p-1)}_{t,x}$ will exhaust the time interval $[t_-, t_+]$ after finitely many steps.  Since $|N'(t)| \lesssim N(t)^2$ on an interval of length $\delta / N(t)$, for any $t_1, t_2 \in [t_-, t_+]$, which satisfy $|t_1 - t_2| \leq \delta / N(t_1)$, we have that
\[
N(t_1) - \delta N(t_1) \lesssim N(t_2) \lesssim N(t_1) + \delta N(t_1).
\]
Consequently, for any $t_1 \in [t_-, t_+]$ we must have
\[
\int_{t_1 - \delta / N(t_1)}^{t_1 + \delta / N(t_1)} N(t) \ud t \geq (2\delta - 2\delta^2) ,
\]
which for any $0 < \delta < \frac{1}{2}$ yields
\begin{align}\label{equ:lwr_bds}
\int_{t_1 - \delta / N(t_1)}^{t_1 + \delta / N(t_1)} N(t) \, \ud t \geq \delta.
\end{align}
By \eqref{equ:int_bds} and \eqref{equ:lwr_bds},
\begin{align}
4C \geq  \int_{t_-}^{t_+} N(t) \ud t  &=  \int_{t_-}^{t_- + \delta / N(t_-)} N(t) \ud t + \int_{t_- + \delta / N(t_-)}^{t_+} N(t ) \ud t \\
&=  \int_{t_- + \delta / N(t_-)}^{t_+ - \delta / N(t_+)} N(t) \ud t + \delta,
\end{align}
hence the positivity of $N(t)$ implies that by iterating this procedure, we be able to cover the whole interval $[t_-, t_+]$ in at most $4C / \delta$ many intervals of length $\delta / N(t)$ where we can control the $L^{2(p-1)}_{t,x}$ norm of the critical element. 

On \textbf{Region B}, we use \eqref{equ:compactness_small_data} to apply the small data theory at times $t_\pm$, which, together with finite speed of propagation, yields a uniform bound on the solution. Finally, on \textbf{Region C}, we may use the sharp Huygens principle exactly as in Section~\ref{s:soliton-reg}.

All together, using arguments from Section~\ref{s:soliton}, this will yield that
\begin{equation}
\| u(t_{1}) \|_{\dot \cH^1} \lesssim N(t_{1})^{\frac{5 - p}{2(p - 1)}}.
\end{equation}
For more details, we refer the reader to \cite{DL2}. By continuation of regularity and $(\ref{equ:lwr_bds})$, this implies
\begin{equation}
\| u(t_{0}) \|_{\dot \cH^1} \lesssim N(t_{1})^{\frac{5 - p}{2(p - 1)}},
\end{equation}
where the implicit constant again depends on $C$. Finally, since $|N'(t)| \lesssim N(t)^{2}$,
\begin{equation}
N(t_{0}) \sim_{C} N(t_{1}),
\end{equation}
which completes the proof.

\end{proof}

\subsection{Case 2: \texorpdfstring{$C(t)$}{C} is not uniformly bounded}  In this case we will show how to extract a self-similar-like critical element by taking an appropriate limit. The arguments from Section~\ref{s:ss}, specifically Proposition \ref{P:reg-jump-ss}, then allow us to conclude that any such solution must be $\equiv 0$, which is a contradiction. 

By assumption, there exists sequences $\{t_n\}$ such that
\[
C(t_{n}) \geq 2n.
\]

Now define 
\[
I_{n} = [t_{n} - \kappa_{-}(t_{n}, n) N(t_{n})^{-1}, t_{n} + \kappa_{+}(t_{n}, n) N(t_{n})^{-1}].
\]
Borrowing language from \cite{TVZ}, since $C(t_{n}) \geq 2n$, we show that all sufficient late times $t \in I_n$ are future focusing, that is,
\begin{equation}
\forall \tau \in I_{n} : \tau > t, \quad |x(\tau) - x(t)| \geq |\tau - t| - \frac{R(\eta)}{N(\tau)},
\end{equation}
or all sufficiently early times $t \in I_{n}$ are past focusing, that is
\begin{equation}
\forall \tau \in I_{n} : \tau < t, \quad |x(t) - x(\tau)| \geq |t - \tau| - \frac{R(\eta)}{N(\tau)}.
\end{equation}
Indeed, suppose that there exist $t_{-}^{n}, t_{+}^{n} \in I_{n}$ such that $\tau(t_{+}^{n}) - \tau(t_{-}^{n}) \geq C_{n}$ for some $C_{n} \nearrow \infty$ as $n \nearrow \infty$, $t_{-}^{n}$ is future focusing, $t_{+}^{n}$ is past focusing, and $t_{-}^{n} < t_{+}^{n}$. In that case,
\begin{equation}
N(t) \sim N(\tau), \quad \forall t, \tau \in [t_{-}^{n}, t_{+}^{n}],
\end{equation}
with constant independent of $n$. For $n$ sufficiently large this violates subluminality. 

Therefore, suppose without loss of generality that for $n$ sufficiently large, all sufficiently late times, say all
\begin{equation}
t \in [t_{n} + \kappa_{+}(t_{n}, n/2) N(t_{n})^{-1}, t_{n} + \kappa_{+}(t_{n}, n) N(t_{n})^{-1}] = I_{n}',
\end{equation}
are future focusing. 
First, we note that if $t \in I_{n}$ is future focusing, then for any $\tau \in I_{n}$, $\tau > t$,
\begin{equation}\label{monotone}
N(\tau) \leq \frac{R(\eta)}{c} \inf_{t < s < \tau} N(s).
\end{equation}
Indeed, for any $\tau \in I$, $\tau > t$, $|x(\tau) - x(t)| \geq |\tau - t| - \frac{R(\eta)}{N(\tau)}$. Then if $N(t) \leq \frac{c N(\tau)}{R(\eta)}$,
\begin{equation}
|x(t) - x(\tau)| \geq |t - \tau| - \frac{c}{N(t)},
\end{equation}
and therefore, we conclude that $N(\tau) \leq \frac{1}{c^{2}} N(t) \leq \frac{N(\tau)}{c R(\eta)}$, which is a contradiction for $R(\eta)$ sufficiently large. Note that in the case of past focusing times, a similar argument yields a lower bound in place of \eqref{monotone}.

Consequently, for any $t \in I_{n}'$,
\begin{equation}
N(t) \leq \inf_{\tau < t : \tau \in I_{n}'} N(\tau).
\end{equation}
In particular, modifying by a constant, $N(t)$ may be replaced by $\widetilde{N}(t)$ on $I_{n}'$, where
\begin{equation}
\widetilde{N}(t) = \inf_{t_{n} + \kappa_{+}(t_{n}, \frac{n}{2})N(t_{n})^{-1} < \tau < t} N(\tau).
\end{equation}
Clearly, $\widetilde{N}(t)$ is monotone decreasing. Furthermore, $\widetilde{N}(t)$ must converge to a self-similar solution taking $n \to \infty$.  
The main idea is that forward in time, on longer and longer time intervals,  the pre-compact solution expands to fill the light cone. This observation will enable us to extract a solution which ``looks self-similar'' on $[1, \infty)$ and we can then rescale that solution to extract a true self-similar solution on $[0,\infty)$. We proceed with this argument now. 

We begin by simplifying our notation, setting 
\begin{align}
t_{-}^{n} &= t_{n} + \kappa_{+}(t_{n}, \frac{n}{2}) N(t_{n})^{-1}\\
t_{+}^{n} &= t_{n} + \kappa_{+}(t_{n}, n) N(t_{n})^{-1}. 
\end{align}
By definition of subluminality (see Definition \ref{d:sl}), it holds that uniformly for all $t \in I_{n}'$,
\begin{equation}
\widetilde{N}(t) (t - t_{-}^{n}) \lesssim 1,
\end{equation}
 independent of $n$. We further have that
\begin{equation}
\widetilde{N}(t) (t - t_{-}^{n}) \gtrsim 1
\end{equation}
is also uniformly bounded for all $t \in I_{n}'$ such that $t - t_{-}^{n} \geq \frac{\delta}{N(t_{-}^{n})}$ by finite propagation speed. 

Now set 
\[
K_n := [\kappa_{+}(t_{n}, n) - \kappa_{+}(t_{n}, n/2)] N(t_{n})^{-1} \cdot N(t_{-}^{n}).
\]
Since
\begin{align}\label{equ:cn_unbdd}
\int_{t_{-}^{n}}^{t_{+}^{n}} \widetilde{N}(t) \, \ud t \sim \frac{n}{2} \to \infty,
\end{align}
and $\widetilde{N}(t) \leq \widetilde{N}(t_{-}^{n})$ for all $t \in I_{n}'$, we see that if $K_n \leq C$ for all $n \in \mathbb{N}$, then
\[
\int_{t_{-}^{n}}^{t_{+}^{n} } \widetilde{N}(t) \lesssim 1,
\]
which contradicts \eqref{equ:cn_unbdd}.  Hence we may conclude that $K_n$ is unbounded.  We can then define a rescaled sequence as follows: set
\[
u_n(0,x) = \frac{1}{\widetilde{N}(t_{-}^{n})^\frac{2}{p-1}} u \biggl(t_{-}^{n}, x(t_{-}^{n}) +  \frac{x}{\widetilde{N}(t_{-}^{n})}\biggr),
\]
\[
\partial_t u_n(0,x) =  \frac{1}{\widetilde{N}(t_-^n)^{\frac{2}{p-1} + 1}} u \biggl(t_-^n, x(t_-^n) +  \frac{x}{\widetilde{N}(t_-^n)}\biggr)
\]
and let
\[
\vec w_n(1) = \left( u_n(0, x), \partial_t u_n(0,x) \right).
\]
By pre-compactness of the trajectory of $\vec u$ in $\dot \cH^{s_p}$ (modulo symmetries), the rescaled initial data converges, that is $\vec w_n(1) \to w_\infty$ in $\dot \cH^{s_p}$. We let $\vec w(s)$ be the evolution of $\vec w_\infty =: \vec w(1)$, then $\vec w_\infty$ has the compactness property with a new scaling parameter $\widehat{N}(s)$, given by
\[
\widehat{N}(s) = \lim_{n \to \infty} \frac{\widetilde{N}(t_{-}^{n} + \frac{s}{\widetilde{N}(t_{-}^{n})})}{\widetilde{N}(t_{-}^{n})}.
\]
Hence we have
\[
cs \leq \frac{1}{\widehat{N}(s)} \leq s, \quad \textup{for all }s > 1.
\]

We may also assume without loss of generality that $\vec w_\infty$ has the compactness property with translation parameter $\widetilde{x}(s) = 0$: by finite speed of propagation, $\widetilde{x}(s)$ must remain bounded, and hence we may, up to passing to a subsequence, obtain a pre-compact solution with $\widetilde{x}(s) = 0$ by applying a fixed translation. Finally, we consider one last sequence of times $\{s_n\}$ with $s_n \to \infty$ and we define
\[
w_n(1,x) = \frac{1}{(s_n)^\frac{2}{p-1}} w \biggl(s_n, \frac{x}{s_n}\biggr), \quad \partial_t u_n(1,x) =  \frac{1}{(s_n)^{\frac{2}{p-1} + 1}} w \biggl(s_n, \frac{x}{s_n}\biggr).
\]
We set
\[
\vec v_n(1) = \left( w_n(1, x), \partial_t w_n(1,x) \right),
\]
which gives rise to a corresponding solution $\vec v_n(\tilde{s})$ with $\widehat{N}(\tilde{s}) = \tilde{s}^{-1}$ on $[\frac{1}{s_n}, \infty)$. We then can take the limit $n \to \infty$, which yields convergence $\vec v_n \to \vec v_\infty$ in $\dot \cH^{s_p}$, and a solution $\vec v$ with initial data $\vec v_\infty$ which is self-similar on $[0, \infty)$.

\section{The traveling wave critical element} \label{s:hans} 

In this section we preclude the possibility of the existence of a `traveling wave' critical element.  

Recall the definition of a traveling wave critical element. 

\begin{defn}[Traveling wave]
We say $\vec u(t) \neq 0$ is a \emph{traveling wave critical element} if $\vec u(t)$ is  a global-in-time solution to~\eqref{eq:nlw} such that the set 
\EQ{
K:=  \left\{ \left( u\left( t, \, x(t) +  \cdot \right), \, \p_t u\left( t,  \,  x(t) + \cdot \right) \right) \mid t \in \R \right\}
} 
is pre-compact in $\dot{H}^{s_p} \times \dot H^{s_p-1} (\R^3)$, where the function $x: \R \to \R^3$ satisfies, 
\begin{align} 
 x(0)&= 0,  \label{eq:xhans0}\\ 
\abs{t} - C_1 &\le \abs{x(t)} \le \abs{t} + C_1 \label{eq:xhans1} \\ 
\abs{ x(t) - (t, 0, 0)} &\le C_1  \abs{t}^{\frac{1}{2}} \label{eq:xhans2}
\end{align}
for some uniform constant $C_1>0$. 
\end{defn}

The main result of this section is the following theorem.

\begin{prop} \label{p:hans_zero} There are no traveling wave critical elements, in the sense of Case (IV) of Proposition~\ref{p:cases}.
\end{prop} 

To prove Propostion~\ref{p:hans_zero}, we will show that any traveling waves critical element would enjoy additional regularity in the $x_2$ and $x_3$ directions. This will allow us to utilize a direction-specific Morawetz-type estimate to reach a contradiction.  We will require an additional technical ingredient, namely, a long-time Strichartz estimate in the spirit of~\cite{D-JAMS, D-Duke}.  

%

\subsubsection{Main ingredients in the proof} 
The long-time Strichartz estimates take the following form:

\begin{prop}[Long-time Strichartz estimate]\label{P:tw-lts}  Suppose $\vec u(t)$ is a traveling wave critical element for \eqref{eq:nlw}.  Let $\eps \in (0, 1)$ be arbitrary. Then, 
\[
\|u_{>N}\|_{S([t_0,t_0+N^{1- \eps}])}  = o_N(1) \mas  N \to \infty. 
\]
where $S(I)$ denotes any admissible, non-endpoint Strichartz norm at Sobolev regularity $s = s_p$ on the time interval $I$. 
\end{prop}

With the help of Proposition \ref{P:tw-lts}, we will also prove the following additional regularity result.

\begin{prop}[Additional regularity]\label{padditional}
Suppose $\vec u(t)$ is a traveling wave critical element for \eqref{eq:nlw}. For any $0<\nu<\frac12$,  
\begin{equation}\label{equ:hans_added_reg}
\| |\partial_{2}|^{1 - \nu} u \|_{L_{t}^{\infty} L_x^{2}(\R\times\R^3)} + \| |\partial_{3}|^{1 - \nu} u \|_{L_{t}^{\infty} L_x^{2}(\R\times\R^3)} < \infty.
\end{equation}
\end{prop}

Using Proposition~\ref{P:tw-lts} and Proposition~\ref{padditional}, we can then prove the following Morawetz-type estimate. In the sequel, we use the notation 
\[
x=(x_1, x_{2,3}).
\] 
\begin{prop}[Morawetz-type estimate]\label{P:tw-morawetz} Suppose $\vec u(t)$ is a traveling wave critical element for \eqref{eq:nlw}.  Then there exists $\delta>0$ and $\eps>0$ such that
\[
\lim_{T\to\infty} \frac{1}{T^{1-\eps}}\int_0^{T^{1-\eps}} \int_{|x_{2,3}|\leq T^\delta} |u_{\leq T}(t,x)|^{p+1}\,\ud x\,\ud t = 0. 
\]
\end{prop}

Combining Proposition~\ref{P:tw-morawetz} with the nontriviality of critical elements will yield a contradiction and complete the proof of Theorem~\ref{p:hans_zero}. 

We turn to the proofs of the three preceding propositions. In Section~\ref{s:mor} we also give the proof of Proposition~\ref{p:hans_zero}.

\subsection{Long-time Strichartz estimates}  \label{s:ltse} 

In this subsection we prove the long-time Strichartz estimate, Proposition~\ref{P:tw-lts} and then deduce a few technical corollaries.

\begin{proof}[Proof of Proposition~\ref{P:tw-lts}]  For technical reasons we fix a small parameter $0<\theta\ll1$ and introduce the following norm: given a time interval $I$, 
\begin{align}\label{equ:lts_space}
\|u\|_{S_{\theta}(I)}  & = \|u\|_{L_{t,x}^{2(p-1)}} +  \| |\nabla|^{-\frac{2 -3\theta}{2(p-1)}}u\|_{L_t^{p-1}L_x^{\frac{2(p-1)}{\theta}}} \\
& \quad + \| |\nabla|^{-\frac{1-\theta}{p-1}} u\|_{L_t^{\frac{2(p-1)}{2-\theta}}L_x^{\frac{2(p-1)}{\theta}}} + \| |\nabla|^{s_p-\theta}u \|_{L_t^{\frac{2}{\theta}} L_x^{\frac{2}{1-\theta}}} \\
&+ \| |\nabla|^{\frac{2 s_p}{3} - \frac{1}{3} }u \|_{L^\frac{6}{1+s_p}_t L^\frac{6}{2-s_p}_x } +  \| |\nabla| ^{\frac{3}{4} - \frac{3}{2(p-1)}} u \|_{L^{2(p-1)}_t L^4_x} ,
\end{align} 
where all space-time norms are over $I\times\R^3$.  Restrictions will be put on $\theta$ below.  One can check that each of these norms correspond to wave-admissible exponent pairs at $\dot H^{s_p}$ regularity; this already requires $0<\theta<p-3$. We will prove Propostion~\ref{P:tw-lts} for the space $S_{\te}$ and note here that the same estimates then easily follow for the whole family of admissible Strichartz norms. We also remark that a nearly identical (but simpler) argument works in the case $p=3$, with the caveat that we need to perturb away from the inadmissible $(2,\infty)$ endpoint.  

Let $\eta_0>0$ and $\eps>0$.   We will actually prove that there exists $N_0\gg 1$ such that for $N\geq N_0$, we have
\[
\|u_{>N}\|_{S_\theta([t_0,t_0+(\frac{N}{N_0})^{1-\eps}])} <\eta_0
\]
for any $t_0\in\R$ and $\theta < 2 \eps/3$. This implies the estimate appearing in the statement of Proposition~\ref{P:tw-lts} upon enlarging $\eps$ and $N_0$; indeed, $N^{1-\eps'}\leq (\tfrac{N}{N_0})^{1-\eps}$ provided $N\geq N_0^{\frac{1-\eps}{\eps'-\eps}}$. 

By compactness and $N(t)\equiv 1$, there exists $N_0$ sufficiently large such that
\[
\|u_{>N_0}\|_{S_\theta([t_0,t_0+9^{1-\eps}])} <\tfrac12\eta_0
\]
for any $t_0\in\R$.  This implies the desired estimate for $N_0\leq N\leq 9N_0$.  We will prove the result for larger $N$ by induction. 

Note that by choosing $N_0$ possibly even larger, we can guarantee
\begin{equation}\label{tw-lts-small1}
\|P_{>N}\vec u\|_{L_t^\infty \H^{s_p}(\R \times \R^{3})} <\tfrac12\eta_0
\end{equation}
for any $N\geq N_0$. 

Before completing the inductive step, we make a few simplifications.  First, by time-translation invariance, it suffices to consider $t_0=0$. Next, to keep formulas within the margins, we will assume all space-time norms are over $[0,(\tfrac{N}{N_0})^{1-\eps}]\times\R^3$ unless otherwise stated. 

By the Taylor's theorem, we can write
\begin{align}
F(u) &= F(u_{\leq N}) +  u_{> N} \int_0^1 F'(u_{< N} +  \theta u_{>N})\\
&= F(u_{\leq N}) +  u_{> N} F'(u_{< N}) + u_{>N}^2 \iint_0^1 F''(u_{< N} +  \theta_1 \theta_2 u_{>N})\\
&= F(u_{\leq N}) +  u_{> N} F'(u_{< N}) + u_{>N}^2 F''(u_{< N}) \\
& \hspace{14mm}+  u_{>N}^3 \iiint_0^1 F'''(u_{< N} +  \theta_1 \theta_2 \theta_3  u_{>N})
\end{align}
for any $N$.  Thus (ignoring absolute values and constants) we need to estimate four types of terms
\[
 u_{>\frac{N}{8}} u_{\leq \frac{N}{8}}^{p-1}   + u_{>\frac{N}{8}}^2 u_{<\frac{N}{8}}^{p-2} + u_{\leq \frac{N}{8}}^p +  u_{>\frac{N}{8}}^3 F_2 = : I + II + III + IV,
\]
where
\[
F_2 =  \iiint_0^1 F'''(u_{< \frac{N}{8}} +  \theta_1 \theta_2 \theta_3 u_{>\frac{N}{8}}).
\]

We will estimate the contribution of each term using Strichartz estimates. 

\subsection*{Term I} We let $0\leq \theta<p-3$ as in \eqref{equ:lts_space} and further impose $\theta<\tfrac23\eps$. We estimate
\begin{align*}
\| |\nabla|^{s_p-1}& P_{>N}( u_{\leq \frac{N}{8}}^{p-1}u_{\geq \frac{N}{8}})\|_{L_t^{1}L_x^{2} } \\&\lesssim N^{s_p-1} \|u_{\leq \frac{N}{8}}\|^{p-1}_{L_t^{p-1}L_x^\frac{2(p-1)}{\theta}}  \|u_{>\frac{N}{8}}\|_{L^{\infty}_t L^{\frac{2}{1-\theta}}_x}\\
& \lesssim N^{-1+\frac{3\theta}{2}} \|u_{\leq \frac{N}{8}}\|^{p-1}_{L_t^{p-1}L_x^\frac{2(p-1)}{\theta}}  \||\nabla|^{s_p - \frac{3\theta}{2} }u_{>\frac{N}{8}}\|_{L^{\infty}_t L^{^\frac{2}{1-\theta }}_x} \\
& \lesssim \bigl[N^{-\frac{2-3\theta}{2(p-1)}}\|u_{\leq \frac{N}{8}}\|_{L_t^{p-1} L_x^{\frac{2(p-1)}{\theta}}}\bigr]^{p-1} \| |\nabla|^{s_p}u_{>\frac{N}{8}}\|_{L_t^\infty L_x^2}. 
\end{align*}
Recalling \eqref{tw-lts-small1}, it remains to prove
\begin{equation}\label{tw-lts-term2}
N^{-\frac{2-3\theta}{2(p-1)}}\|u_{\leq \frac{N}{8}}\|_{L_t^{p-1}L_x^\frac{2(p-1)}{\theta}} \lesssim \eta_0.
\end{equation}
We let $C_0\gg1$ to be determined shortly and begin by splitting
\begin{align*}
N^{-\frac{2 -3\theta}{2(p-1)}} \|u_{\leq \frac{N}{8}}\|_{L_t^{p-1}L_x^\frac{2(p-1)}{\theta}}  &\lesssim N^{-\frac{2 -3\theta}{2(p-1)}}\|u_{\leq C_0}\|_{L_t^{p-1}L_x^\frac{2(p-1)}{\theta}}  \\
& \quad +N^{-\frac{2 -3\theta}{2(p-1)}} \|u_{C_0\leq\cdot\leq N_0}\|_{L_t^{p-1}L_x^\frac{2(p-1)}{\theta}} \\
& \quad + N^{-\frac{2 -3\theta}{2(p-1)}}  \sum_{N_0\leq M\leq\frac{N}{8}} \|u_{M}\|_{L_t^{p-1}L_x^\frac{2(p-1)}{\theta}}.
\end{align*}
By Bernstein's inequality and $N(t)\equiv 1$, we can estimate
\begin{align*}
N^{-\frac{2 -3\theta}{2(p-1)}} \|u_{\leq C_0}\|_{L_t^{p-1}L_x^\frac{2(p-1)}{\theta}} & \lesssim N^{-\frac{2 -3\theta}{2(p-1)}}C_0^{\frac{2-3\theta}{2(p-1)}} \bigl(\tfrac{N}{N_0}\bigr)^{\frac{(1-\eps)}{p-1}}
\end{align*}
on $[0,(\tfrac{N}{N_0})^{1-\eps}]\times\R^3$.  To guarantee that the overall power of $N$ is negative, we need
\[
\frac{3\theta}{2} < \varepsilon.
\]
 Thus, for $N_0$ sufficiently large depending on $C_0$, we may guarantee that
\[
N^{-\frac{2 -3\theta}{2(p-1)}} \|u_{\leq C_0}\|_{L_t^{p-1}L_x^\frac{2(p-1)}{\theta}} \lesssim \eta_0. 
\]
Next, choosing $C_0=C_0(\eta_0)$ large enough and using $N(t)\equiv 1$, we estimate 
\begin{align*}
&N^{-\frac{2 -3\theta}{2(p-1)}}\|u_{C_0\leq \cdot \leq N_0} \|_{L_t^{p-1}L_x^\frac{2(p-1)}{\theta}} \\
& \lesssim N^{-\frac{2 -3\theta}{2(p-1)}}N_0^{\frac{2-3\theta}{2(p-1)}}\||\nabla|^{- \frac{2-3\theta}{2(p-1)}}u_{>C_0}\|_{L_t^{p-1}L_x^\frac{2(p-1)}{\theta}} \\
& \lesssim \eta_0 \bigl(\tfrac{N}{N_0}\bigr)^{-\frac{2  -3\theta}{2(p-1)}+\frac{(1-\eps)}{p-1}} \lesssim \eta_0. 
\end{align*}
For the final term, we begin by estimating 
\begin{align*}
N^{-\frac{2 -3\theta}{2(p-1)}}&\sum_{N_0\leq M\leq \frac{N}{8}}  \|u_M\|_{L_t^{p-1}L_x^\frac{2(p-1)}{\theta}} \\ 
\quad & \lesssim \sum_{N_0\leq M\leq \frac{N}{8}} \bigl(\tfrac{M}{N}\bigr)^{\frac{2 -3\theta}{2(p-1)}}\| |\nabla|^{-\frac{2 -3\theta}{2(p-1)}} u_M\|_{L_t^{p-1}L_x^\frac{2(p-1)}{\theta}}. 
\end{align*}
We now apply the inductive hypothesis to the last term. To do so, we divide the interval $[0,(\tfrac{N}{N_0})^{1-\eps}]$ into $\approx (\tfrac{N}{M})^{1-\eps}$ intervals of length $(\tfrac{M}{N_0})^{1-\eps}$. Continuing from above, this leads to
\begin{align*}
N^{-\frac{2 -3\theta}{2(p-1)}}&\sum_{N_0\leq M\leq \frac{N}{8}}  \|u_M\|_{L_t^{p-1}L_x^\frac{2(p-1)}{\theta}}  \lesssim \sum_{N_0\leq M\leq \frac{N}{8}} \bigl(\tfrac{M}{N}\bigr)^{\frac{2  -3\theta}{2(p-1)}-\frac{(1-\eps)}{p-1}} \eta_0 \lesssim \eta_0,
\end{align*}
where we have used that the exponent appearing is, in this case, positive. This completes the estimation of Term I. 

\subsection*{Term II}
We estimate 
\begin{align}
\| |\nabla|^{s_p - 1}& P_{>N}( u_{\leq \frac{N}{8}}^{p-2}u_{\geq \frac{N}{8}}^2)\|_{L_t^{1}L_x^{2} }\\
& \lesssim N^{ s_p-1 } \|u_{> \frac{N}{8}} \|^2_{L^{2(p-1)}_t L^4_x} \|u_{\leq \frac{N}{8}} \|_{L^{p-1}_t L^\infty_x}^{p-2}\\
& \lesssim N^{ s_p-1+ 1 - \frac{1}{p-1}} \|u_{> \frac{N}{8}} \|^2_{L^{2(p-1)}_t L^4_x}  N^{- \frac{p-2}{(p-1)}} \|u_{\leq \frac{N}{8}} \|_{L^{p-1}_t L^\infty_x}^{p-2}\\
& \lesssim \bigr[ N^{ \frac{3}{4}- \frac{3}{2(p-1)}} \|u_{> \frac{N}{8}} \|_{L^{2(p-1)}_t L^4_x} \bigr]^2  N^{- \frac{p-2}{(p-1)}} \|u_{\leq \frac{N}{8}} \|_{L^{p-1}_t L^\infty_x}^{p-2}.
\end{align}
We can argue as above (now with $\theta = 0$) for the low frequency term, and we note that $(2(p-1), 4)$ is a wave admissible pair at regularity
\[
\frac{3}{2} - \frac{1}{2(p-1)} - \frac{3}{4} = s_p - \left(\frac{3}{4} -  \frac{3}{2(p-1)} \right),
\]
and we conclude using the inductive hypothesis on 
\[
\| |\nabla| ^{\frac{3}{4} - \frac{3}{2(p-1)}} u_{> \frac{N}{8}} \|_{L^{2(p-1)}_t L^4_x} .
\]

\subsection*{Term III}
Next using the fractional chain rule we estimate 
\begin{align*}
\| |\nabla|^{s_p-1} &P_{>N}(u_{\leq \frac{N}{8}}^{p})\|_{L_t^{1}L_x^{2} }\\
& \lesssim N^{s_p - 2} \|u_{\leq \frac{N}{8}}\|^{p-1}_{L_t^{\frac{2(p-1)}{2-\theta}}L_x^\frac{2(p-1)}{\theta}}  \||\nabla| u_{\leq \frac{N}{8}}\|_{L^{\frac{2}{\theta}}_t L^{\frac{2}{1-\theta}}_x}\\
& \lesssim N^{-1 + \theta } \|u_{\leq \frac{N}{8}}\|^{p-1}_{L_t^{{\frac{2(p-1)}{2-\theta}}}L_x^\frac{2(p-1)}{\theta}} N^{- \theta + s_p - 1} \||\nabla| u_{\leq\frac{N}{8}}\|_{L^{\frac{2}{\theta}}_t L^{\frac{2}{1-\theta}}_x}.
\end{align*}

To complete the estimation of term III, we need to prove
\begin{align*}
N^{-\frac{1-\theta}{p-1}}\|u_{\leq \frac{N}{8}}\|_{L_t^{\frac{2(p-1)}{2-\theta}} L_x^{\frac{2(p-1)}{\theta}}}+N^{-\theta+s_p-1}\| |\nabla| u_{\leq N}\|_{L_t^{\frac{2}{\theta}} L_x^{\frac{2}{1-\theta}}}&  \lesssim \eta_0. 
\end{align*}

For this, we argue as in term I,  that is, we split
\[
u_{\leq \frac{N}{8}} = u_{\leq C_0}+u_{C_0\leq\cdot\leq N_0}+\sum_{N_0\leq M \leq \frac{N}{8}} u_M. 
\]
and estimate each term separately, relying on the inductive hypothesis (and a splitting of the time interval) for the final sum.  Comparing with those estimates, we see that this requires
\[
\tfrac{1-\theta}{p-1}-\tfrac{1-\eps}{p-1}>0
\]
to deal with the first term and
\[
\theta+1-s_p-\tfrac{\theta(1-\eps)}{2}>0
\]
to deal with the second term.  These conditions are satisfied provided $0<\theta<\eps$.

\subsection*{Term IV}
We estimate
\[
\|u_{> \frac{N}{8}}^3 F_2 \|_{L^\frac{2}{1+s_p}_t L^\frac{2}{2-s_p}_x } \lesssim \|u_{> \frac{N}{8}}^p \|_{L^\frac{2}{1+s_p}_t L^\frac{2}{2-s_p}_x }  + \|u_{> \frac{N}{8}}^3 u_{\leq \frac{N}{8}}^{p-3} \|_{L^\frac{2}{1+s_p}_t L^\frac{2}{2-s_p}_x } .
\]
For the first expression we estimate
\[
\|u_{> \frac{N}{8}}^p \|_{L^\frac{2}{1+s_p}_t L^\frac{2}{2-s_p}_x } = \|u_{> \frac{N}{8}} \|^p_{L^\frac{2p}{1+s_p}_t L^\frac{2p}{2-s_p}_x } \lesssim \eta_0^p,
\]
while for the second expression we have
\begin{align}
\|u_{> \frac{N}{8}}^3 u_{\leq \frac{N}{8}}^{p-3} \|_{L^\frac{2}{1+s_p}_t L^\frac{2}{2-s_p}_x }  \lesssim \|u_{> \frac{N}{8}} \|_{L^\frac{6}{1+s_p}_t L^\frac{6}{2-s_p}_x }^3 \|u_{\leq \frac{N}{8}} \|^{p-3}_{L^\infty_{t,x}}.
\end{align}
Now, 
\begin{align}
\|u_{\leq \frac{N}{8}} \|^{p-3}_{L^\infty_{t,x}} \lesssim N^{\frac{2(p-3)}{p-1}}\| u_{< \frac{N}{8}}\|^{p-3}_{L^{\infty}_{t} L^{\frac{3(p-1)}{2}}}\lesssim N^{\frac{2(p-3)}{p-1}}.
\end{align}
For the first term, we see that $(6/(1+s_p), 6/(2- s_p))$ is an admissible Strichartz pair at regularity
\[
s_p + \frac{1}{3} - \frac{2 s_p}{3} < s_p,
\]
and hence
\begin{align}
&\|u_{> \frac{N}{8}}^3 u_{\leq \frac{N}{8}}^{p-3} \|_{L^\frac{2}{1+s_p}_t L^\frac{2}{2-s_p}_x }  \\
& \lesssim  N^{ 1 - 2s_p }\| |\nabla|^{\frac{2 s_p}{3} - \frac{1}{3} }u_{> \frac{N}{8}} \|^3_{L^\frac{6}{1+s_p}_t L^\frac{6}{2-s_p}_x } N^{\frac{2(p-3)}{p-1}}\| u_{< \frac{N}{8}}\|^{p-3}_{L^{\infty}_{t} L^{\frac{3(p-1)}{2}}}.
\end{align}
Finally, note that \[
-2s_p + 1 + \frac{2(p-3)}{p-1} = - 3 +1 + \frac{4}{p-1} + \frac{2p - 6}{p-1} = 0.
\]

Hence, by the inductive hypothesis, putting all the pieces of the argument together, we obtain
\[
\|u_{>N}\|_{S_{\theta}([0,(\frac{N}{N_0})^{1-\eps}])} \leq \tfrac12\eta_0 + C\eta_0^{3},
\]
which suffices to complete the induction for $\eta_0$ sufficiently small. 
\end{proof}

We will need the following corollary of Proposition~\ref{P:tw-lts}, which provides some control over the low frequencies as well.

\begin{cor}[Control of low frequencies]\label{C:tw-lts} Suppose $\vec u$ is a traveling wave critical element for \eqref{eq:nlw}. Let $\eps>0$ and $0<\theta<\tfrac23\eps$.  For any $\eta_0$ there exists $N$ sufficiently large such that
\begin{align}
 \|u_{\leq N}\|_{L_t^{p-1}L_x^\frac{2(p-1)}{\theta}([t_0,t_0+N^{1-\eps}])\times\R^3)} &\lesssim \eta_0 N^{\frac{2-3\theta}{2(p-1)}},\label{tw-clf1}\\
  \|u_{\leq N}\|_{L_t^{\frac{2(p-1)}{2-\theta}}L_x^\frac{2(p-1)}{\theta}([t_0,t_0+N^{1-\eps}]\times\R^3)}  &\lesssim \eta_0 N^{\frac{1-\theta}{p-1}},\nonumber \\
 \| |\nabla| u_{\leq N}\|_{L^{\frac{2}{\theta}}_t L^{\frac{2}{1-\theta}}_x([t_0,t_0+N^{1-\eps}]\times\R^3)} &\lesssim \eta_0 N^{\theta - s_p + 1}.\nonumber
\end{align}
uniformly over $t_0\in\R$. 
\end{cor}

\begin{proof} We let $\eta_0$ and choose $N_0=N_0(\eta_0)\geq 1$ as in Proposition~\ref{P:tw-lts}.  By time-translation invariance, it suffices to consider $t_0=0$. We focus our attention on \eqref{tw-clf1}, as the other estimates follow similarly.  For $N\geq N_0$, we estimate 
\begin{align*}
\|u_{\leq N}&\|_{L_t^{p-1}L_x^\frac{2(p-1)}{\theta}([0,N^{1-\eps}]\times\R^3)} \\
& \lesssim \|u_{\leq N_0}\|_{L_t^{p-1}L_x^\frac{2(p-1)}{\theta}([0,N^{1-\eps}]\times\R^3)} \\
& \quad + \sum_{N_0\leq M\leq N} \|u_{M}\|_{L_t^{p-1}L_x^\frac{2(p-1)}{\theta}([0,M^{1-\eps}]\times\R^3)} \\
& \quad + \sum_{N_0\leq M\leq N}\|u_{\leq N}\|_{L_t^{p-1}L_x^\frac{2(p-1)}{\theta}([M^{1-\eps},N^{1-\eps}]\times\R^3)}.
\end{align*}

For the first term, we use Bernstein's inequality and $N(t)\equiv 1$ to get
\[
\|u_{\leq N_0}\|_{L_t^{p-1}L_x^\frac{2(p-1)}{\theta}([0,N^{1-\eps}]\times\R^3)} \lesssim N_0^{\frac{2-3\theta}{2(p-1)}}N^{\frac{1-\eps}{p-1}}.
\]
Recalling that $\frac{1-\eps}{p-1}<\frac{2-3\theta}{2(p-1)}$, we see that this term is acceptable provided we choose $N$ sufficiently large. 

Next, we use Proposition~\ref{P:tw-lts} to estimate
\[
\sum_{N_0\leq M\leq N} \|u_M\|_{L_t^{p-1}L_x^{\frac{2(p-1)}{\theta}}([0,M^{1-\eps}]\times\R^3)} \\ \lesssim \eta_0 \sum_{N_0\leq M \leq N} M^{\frac{2-3\theta}{2(p-1)}} \lesssim \eta_0N^{\frac{2-3\theta}{2(p-1)}},
\]
which is also acceptable.

For the remaining term, we split $[M^{1-\eps},N^{1-\eps}]$ into $\approx (\tfrac{N}{M})^{1-\eps}$ intervals of length $M^{1-\eps}$.  Applying Proposition~\ref{P:tw-lts} once more, we have
\begin{align*}
\sum_{N_0\leq M\leq N}&\|u_{M}\|_{L_t^{p-1}L_x^\frac{2(p-1)}{\theta}([M^{1-\eps},N^{1-\eps}]\times\R^3)} \\
& \lesssim \eta_0\sum_{N_0\leq M\leq N}M^{\frac{2-3\theta}{2(p-1)}}\bigl(\tfrac{N}{M}\bigr)^{\frac{1-\eps}{p-1}} \lesssim \eta_0N^{\frac{2-3\theta}{2(p-1)}},
\end{align*}
where we recall $0 <\theta<\tfrac23\eps$ in order to sum. This term is also acceptable, and so we complete the proof of \eqref{tw-clf1} and Corollary~\ref{C:tw-lts}. \end{proof}

Finally, we will need certain long-time Strichartz estimates with regularity in the $x_2, x_3$ directions.
\begin{cor}[Long-time Strichartz estimates for $\nabla_{x_2, x_3} u$]\label{cor:ltse_23}
Suppose that Proposition~\ref{padditional} holds with $\nu>0$.  Then for any $\nu_0 > \nu$,
\[
\| |\nabla_{x_{2,3}}|^{1- \nu_0} u_{N}\|_{L^{\frac{2}{1-s_p}}_{t} L^{\frac{2}{s_p}}_x ([t_0, t_0 +  N^{1-\eps}])} \lesssim N^{1-s_p}.
\]
\end{cor}
\begin{proof}
We only sketch this argument as it follows in the same manner as the standard LTSE with some additional technical details. First we note that $(2/(1-s_p), 2/s_p)$ is an admissible Strichartz pair at regularity $1- s_p$. By compactness, it suffices to argue with $t_0 = 0$. Let $S(I)$ denote any collection of Strichartz pairs at regularity $s= 0$. We will show that 
\begin{align}\label{equ:tlong23}
\| |\nabla_{x_{2,3}}|^{1- \nu_0} u_{N}\|_{S ([t_0, t_0 +  N^{1-\eps}])}  \lesssim 1,
\end{align}
for $\nu_0 > \nu$, from which the result follows.

Let $\vec u$ be a solution with the compactness property on $\R$ with $N(t) = 1$.  By the Gagliardo-Nirenberg inequality
\begin{align}\label{gn_ineq}
\||\nabla_{2,3}|^{1- \nu_0} u \|_{L^2_{x_2, x_3}} \leq C \||\nabla_{2,3}|^{1- \nu} u \|_{L^2_{x_2, x_3}}^{\alpha} \| u\|_{L_{x_2,x_3}^{\frac{2}{1-s_p}}}^{1-\alpha}
\end{align}
for
\[
\alpha = \frac{1 - s_p - \nu_0}{1- s_p - \nu}.
\]

Next, we observe the Sobolev embedding
\begin{align}\label{equ:sobolev_embed}
\dot H_{x}^{s_p} \hookrightarrow L_{x_1}^2 L_{x_{2,3}}^{\frac{2}{1-s_p}},
\end{align}
which follows from Sobolev embedding in $\R^2$ and Plancherel:
\begin{align*}
\int\biggl(\int & |u(x_1,x_{2,3})|^{\frac{2}{1-s_p}}\,dx_{2,3} \biggr)^{1-s_p}\,dx_1 \\
& \lesssim \int \bigl| |\nabla_{x_{2,3}}|^{s_p} u\bigr|^2 \,\ud x \sim \int \bigl| | \xi_{2,3} |^{s_p} \hat u(\xi)\bigr|^2 \,d\xi  \lesssim \int \bigl| |\xi|^{s_p}\hat u(\xi)\bigr|^2 \,d\xi. 
\end{align*}
Thus we may take the $L^2_{x_1}$ norm of both sides of \eqref{gn_ineq} and us H\"older's inequality on the right to conclude that the trajectory $\abs{\na_{2, 3}}^{1-\nu_0} u$ has the compactness property in $L^2_x$, and hence there exists $N_0 = N_0 (\eta_0)$ such that for all $N>N_0$,
\begin{align}\label{equ:base_case}
\|P_{>N} \abs{\na_{2, 3}}^{1- \nu_0}  u\|_{S([0, \, 9^{1- \varepsilon}])} \leq \eta_0, \qquad \textup{for all  } N \geq N_0(\eta_0),
\end{align}
which proves the base case, that is \eqref{equ:tlong23} holds for $N_0 \leq N \leq 9 N_0$. 

We now proceed to the inductive step. Suppose that \eqref{equ:tlong23} holds up to frequency $N_1$ for $N_1 \geq 9 N_0$. We will show that \eqref{equ:tlong23} holds for $N = 2N_1$. The argument we employ is similar to a persistence of regularity argument. Note that $\abs{\na_{2, 3}}^{1- \nu_0} u$ solves the equation
\[
\partial_t \abs{\na_{2, 3}}^{1- \nu_0} u - \Delta \abs{\na_{2, 3}}^{1- \nu_0} u = \abs{\na_{2, 3}}^{1- \nu_0} F(u).
\]
By the Strichartz estimates we have:
\begin{align}
&\| P_{> N} \abs{\na_{2, 3}}^{1- \nu_0} u \|_{S([0, \left( N /N_0\right)^{1-\varepsilon} ])} \\
&\lesssim \| P_{> N} \abs{\na_{2, 3}}^{1- \nu_0} \vec u \|_{L^\infty_t \dot \cH_x^{0} ([0, \left( N /N_0\right)^{1-\varepsilon} ])} + \|P_{>N} \abs{\na_{2, 3}}^{1- \nu_0}  F(u) \|_{\cN([0, \left( N /N_0\right)^{1-\varepsilon} ])}.
\end{align}
where $\cN$ is the dual space to $S$. Let $\widehat{P}_M$ denote a Fourier projection in the $\xi_2, \xi_3$ variables. The first term can be bounded using compactness, so we focus on the second term. We again write
\begin{align}
F(u) &= F(u_{\leq N}) +  u_{> N} \int_0^1 F'(u_{< N} +  \theta u_{>N})\\
&= F(u_{\leq N}) +  u_{> N} F'(u_{< N}) + u_{>N}^2 \iint_0^1 F''(u_{< N} +  \theta_1 \theta_2 u_{>N})\\
&= F(u_{\leq N}) +  u_{> N} F'(u_{< N}) + u_{>N}^2 F''(u_{< N}) \\
& \hspace{14mm}+  u_{>N}^3 \iiint_0^1 F'''(u_{< N} +  \theta_1 \theta_2 \theta_3  u_{>N}).
\end{align}
We will estimate the first term as an example, since the other terms will be similar generalizations of the proof of Proposition~\ref{P:tw-lts}. We have
\begin{align}
& \||\nabla|^{-1} |\nabla_{2,3}|^{1-\nu_0} P_{>N} F(u_{\leq \frac{N}{8}})\|_{L^1_t L^2_x} \\
& \lesssim N^{- 1} \| |\nabla_{2,3}|^{1-\nu_0} P_{>N} F(u_{\leq \frac{N}{8}})\|_{L^1_t L^2_x} \\
& \lesssim N^{- 2} \| |\nabla_{2,3}|^{1-\nu_0}  u_{\leq \frac{N}{8}}\|^{p-1}_{L_t^{\frac{2(p-1)}{2-\theta}}L_x^\frac{2(p-1)}{\theta}}  \||\nabla| u_{\leq \frac{N}{8}}\|_{L^{\frac{2}{\theta}}_t L^{\frac{2}{1-\theta}}_x} \\
& \hspace{14mm} +  N^{- 2} \|  u_{\leq \frac{N}{8}}\|^{p-1}_{L_t^{\frac{2(p-1)}{2-\theta}}L_x^\frac{2(p-1)}{\theta}}  \||\nabla||\nabla_{2,3}|^{1-\nu_0}  u_{\leq \frac{N}{8}}\|_{L^{\frac{2}{\theta}}_t L^{\frac{2}{1-\theta}}_x}\\
& \lesssim N^{-1 + \theta - s_p} \| |\nabla_{2,3}|^{1-\nu_0}  u_{\leq \frac{N}{8}}\|^{p-1}_{L_t^{\frac{2(p-1)}{2-\theta}}L_x^\frac{2(p-1)}{\theta}} N^{-1 - \theta + s_p}  \||\nabla| u_{\leq \frac{N}{8}}\|_{L^{\frac{2}{\theta}}_t L^{\frac{2}{1-\theta}}_x} \\
& \hspace{14mm} +  N^{- 1 + \theta} \|  u_{\leq \frac{N}{8}}\|^{p-1}_{L_t^{\frac{2(p-1)}{2-\theta}}L_x^\frac{2(p-1)}{\theta}} N^{-1 - \theta} \||\nabla||\nabla_{2,3}|^{1-\nu_0}  u_{\leq \frac{N}{8}}\|_{L^{\frac{2}{\theta}}_t L^{\frac{2}{1-\theta}}_x},
\end{align}
and all four terms can be treated analogously to the low frequency component in Term I in Proposition~\ref{P:tw-lts}. \end{proof}

\subsection{Proof of Proposition~\ref{P:tw-morawetz} and Proposition~\ref{p:hans_zero}, assuming Proposition~\ref{padditional}}  \label{s:mor} 
As mentioned above, the long-time Strichartz estimate (Proposition~\ref{P:tw-lts}) will be a key ingredient to proving additional regularity (Proposition~\ref{padditional}).  Before turning to the rather technical proof, let us use Proposition~\ref{padditional} (together with Proposition~\ref{P:tw-lts} and Corollary~\ref{C:tw-lts}) to prove the Morawetz estimate, Proposition~\ref{P:tw-morawetz}.  With the Morawetz estimate in hand, we can then quickly rule out the possibility of traveling waves and hence complete the proof of the main result, Proposition~\ref{p:hans_zero}.

We recall the notation $x=(x_1,x_{2,3})$ and similarly write $\xi=(\xi_1,\xi_{2,3})$. 

\begin{proof}[Proof of Proposition~\ref{P:tw-morawetz}] Let $\psi:[0,\infty)\to\R$ be a smooth cutoff satisfying
\[
\begin{cases}
\psi(\rho) = 1 & \rho \leq 1, \\
\psi(\rho) = 0 & r >2.
\end{cases}
\]
We fix $R>0$ to be determined below and let $\psi_R(\rho) = \psi(\tfrac{\rho}{R})$. Next, let 
\[
\chi_R(r) = \frac{1}{r}\int_0^r \psi_R(s)\,ds. 
\]
We collect a few useful identities: 
\begin{equation}\label{tw-chi-id}
\partial_k[x^k\chi_R] = \chi_R+\psi_R, \quad r\partial_r \chi_R = -\chi_R + \psi_R,
\end{equation}
and we recall the Sobolev embedding \eqref{equ:sobolev_embed}.

In the following, we consider $\chi_R$ as a function of $|x_{2,3}|$. For $T>0$ and 
\[
I:=P_{\leq T},
\]
we define the Morawetz quantity
\[
M(t) = \int_{\R^3} \chi_R Iu_t\, x^k \partial_k Iu \,\ud x + \frac12\int_{\R^3} (\chi_R + \psi_R) Iu_t Iu\,\ud x,
\]
where repeated indices are summed over $k\in\{2,3\}$.  

We first compute the derivative of $M(t)$:
\begin{align*}
M'(t) & = \int \chi_R Iu_t (x^k\partial_k Iu_t) + \frac12 \int (\chi_R+\psi_R)(Iu_t)^2  \\
& \quad + \int \chi_R [x^k\partial_k Iu] Iu_{tt} + \frac12 \int(\chi_R+\psi_R)Iu Iu_{tt}.
\end{align*}

By \eqref{tw-chi-id} and integration by parts, we have
\[
\int [x^k \chi_R]\partial_k\tfrac 12 (Iu_t)^2 = -\frac12\int (\chi_R+\psi_R) (Iu_t)^2,
\]
so we are left to estimate 
\begin{align}\label{tw-mor2}
\int \chi_R [x^k\partial_k Iu] Iu_{tt} + \frac12 \int(\chi_R+\psi_R)Iu Iu_{tt}.
\end{align}
Using the equation for $u$ yields
\[
Iu_{tt} = \Delta Iu -F(Iu)+ [F(Iu)-IF(u)],\qtq{where} F(z) = |z|^{p-1}z. 
\]

We first consider the contribution of $\Delta Iu$ to \eqref{tw-mor2}.  We claim
\begin{equation}\label{tw-mor3}
\int x^k\chi_R [\partial_k Iu] \Delta Iu + \frac12(\chi_R+\psi_R)Iu\Delta Iu\,\ud x \leq \frac 12  \int \Delta(\chi_R+\psi_R)(Iu)^2\,\ud x. 
\end{equation}
In the proof of \eqref{tw-mor3} we will simplify notation by suppressing the operator $I$, suppressing the dependence on $R$, and writing $u_k=\partial_k u$.  We turn to the proof.

We begin by considering the first term on the left-hand side of \eqref{tw-mor3}.  Integrating by parts yields
\[
\int x^k\chi u_k u_{jj} = -\int\partial_j[x^k \chi]u_k u_j + \frac12 x^k\chi \partial_k(u_j^2), 
\]
where $k\in\{2,3\}$ and $j\in\{1,2,3\}$. Writing $r=|x_{2,3}|$ and using \eqref{tw-chi-id}, we have 
\begin{align*}
\int \partial_j[x^k\chi] u_k u_j & = \int \delta_{jk}\chi u_j u_k + \frac{x^k x^j}{r^2}r \chi' u_j u_k \\
& = \int \delta_{jk}\chi u_j u_k + \frac{x^k x^j}{r^2}\psi u_j u_k - \frac{x^k x^j}{r^2}\chi u_j u_k,
\end{align*}
where we may now restrict to $j\in\{2,3\}$.  Using the other identity in \eqref{tw-chi-id}, we also have 
\[
\frac12 \int x^k \chi\partial_k(u_j)^2 = -\frac12 \int (\chi+\psi) u_j^2. 
\]
As for the second term on the left-hand side of \eqref{tw-mor3}, we have
\[
\frac12 \int (\chi+\psi)u u_{jj} = \frac 12 \int  \partial_{jj}(\chi+\psi)u^2- \frac12(\chi+\psi) u_j^2.
\]
Collecting the computations above, we find
\begin{align*}
\int &\bigl( x^k\chi_R [\partial_k Iu] \Delta Iu + \frac12(\chi_R+\psi_R)Iu\Delta Iu\,\bigr) \ud x \\
 & = \frac12 \int \Delta(\chi+\psi)u^2  -\int\bigl[\delta_{jk}-\frac{x^j x^k}{r^2}\bigr]\chi u_j u_k - \int \psi\bigl[\frac{x_{2,3}}{r}\cdot \nabla_{x_{2,3}}u\bigr]^2\,\ud x,
\end{align*}
which yields \eqref{tw-mor3}. 

We next consider the contribution of $-F(Iu)$ to \eqref{tw-mor2}.  Using \eqref{tw-chi-id} and integration by parts, 
\begin{equation}\label{tw-mor4}
\begin{aligned}
-\int&\bigl[x^k\chi_R \partial_k Iu + \frac12(\chi_R+\psi_R) Iu]F(Iu) \,\ud x\\
& = -\int(x^k\chi_R)\frac{1}{p+1}\partial_k|Iu|^{p+1} + \frac 12 (\chi_R+\psi_R)|Iu|^{p+1}\,\ud x \\
& = \int \left(\frac{1}{p+1}-\frac{1}{2}\right)(\chi_R+\psi_R)|Iu|^{p+1}\,\ud x. 
\end{aligned}
\end{equation}
Hence, by \eqref{tw-mor2}, \eqref{tw-mor3}, and \eqref{tw-mor4} and the fundamental theorem of calculus, we deduce 
\begin{equation}\label{TWM1}
\begin{aligned}
\iint_{|x_{2,3}| \leq R}& | Iu|^{p+1}\,\ud x\,\ud t \\
& \lesssim \sup_{t\in J} | M(t)| +  \iint \Delta(\chi_R + \psi_R) |Iu|^2 \,\ud x\,\ud t  \\
& \quad + \biggl| \iint [x^k\chi_R \partial_k Iu+(\chi_R+\psi_R)Iu][F(Iu)-IF(u)]\,\ud x \,\ud t \biggr|
\end{aligned}
\end{equation}
for any interval $J$.  In the following, we choose $J=[0,T^{1-\eps}]$, where $\eps>0$ will be chosen below and $T$ is large enough that Proposition~\ref{P:tw-lts} and Corollary~\ref{C:tw-lts} hold. 

We need to estimate the terms on the right-hand side of this inequality.  We first bound $|M(t)|$. By Bernstein's inequality, the Sobolev embedding \eqref{equ:sobolev_embed}, and Proposition~\ref{padditional}, we have
\begin{equation}
\label{tw-mor-bd}
\begin{split}
\sup_t |&M(t)|\\
 &\lesssim \| I u_t \|_{L_t^\infty L_x^2}\bigl(R \|\nabla_{x_{2,3}} Iu\|_{L_t^\infty L_x^2} +\|Iu\|_{L_t^\infty L_{x_1}^2 L_{x_{2,3}}^{\frac{2}{1-s_p}}}\|\chi_R+\psi_R\|_{L_{x_{2,3}}^{\frac{2}{1+s_p}}} \bigr)   \\
& \lesssim T^{1-s_p}\bigl( RT^\nu + R^{1+s_p}\|u\|_{L_t^\infty \dot H_x^{s_p}}\bigr) \\
& \lesssim T^{1-s_p}\bigl( RT^\nu + R^{1+s_p}\bigr),
\end{split}
\end{equation}
for $\nu>0$ to be chosen sufficiently small below. 

For the next term, we have
\begin{equation}\label{TWM2}
\begin{aligned}
\iint \Delta(\chi_R + \psi_R)|Iu|^2\,\ud x \,\ud t & \lesssim 
T^{1-\eps}\|Iu\|_{L_t^\infty L_{x_1}^2 L_{x_{2,3}}^{\frac{2}{1-s_p}}}^2 \|\Delta(\chi_R+\psi_R)\|_{L_{x_{2,3}}^{\frac{2}{1+s_p}}} \\
& \lesssim T^{1-\eps} \| u\|_{L_t^\infty \dot H_x^{s_p}}^2 R^{-(2-2s_p)}\\
& \lesssim T^{1-\eps}R^{-(2-2s_p)}. 
\end{aligned}
\end{equation}

Now we turn to the final term. Arguing as in the long-time Strichartz estimates, we need to estimate terms of the form
\[
u_{\leq T}^{p} + u_{> T} u_{\leq T}^{p-1} +  u_{> T}^2 u_{\leq T}^{p-2} + u_{>T}^3 F_2,
\]
where $F_2$ involves both high and low frequencies.
Thus we estimate
\begin{align}
 \iint [x^k\chi_R \partial_k &Iu +(\chi_R+\psi_R)Iu][F(Iu)-IF(u)]\,\ud x \,\ud t\\
& \lesssim\iint x^k \chi_R \partial_k u_{\leq T}P_{>T}[u_{\leq T}^p]\,\ud x\,\ud t \label{tw-emor1}\\
&\quad + \iint x^k \chi_R \partial_k u_{\leq T} P_{>T}[|u_{>T}| |u_{\leq T}|^{p-1}]\,\ud x \ud t \label{tw-emor2} \\
 & \quad +\iint x^k \chi_R \partial_k u_{\leq T}P_{>T}[|u_{>T}|^2 |u_{\leq T}|^{p-2}]\,\ud x \ud t \label{tw-emor3} \\
&\quad + \iint x^k \chi_R \partial_k u_{\leq T}P_{>T} [|u_{>T}|^3 F_2 ]\,\ud x \ud t \label{tw-emor4} \\
& \quad+\iint (\chi_R + \psi_R)u_{\leq T} P_{>T} [ u_{\leq T}^p]\,\ud x\,\ud t \label{tw-emor5}\\
& \quad+\iint (\chi_R + \psi_R)u_{\leq T } P_{>T} [ |u_{> T}| |u_{\leq T} |^{p-1]}\,\ud x\,\ud t \label{tw-emor6}\\
& \quad+\iint (\chi_R+\psi_R) u_{\leq T}P_{>T}[ |u_{> T}|^2 |u_{\leq T} |^{p-2}]\,\ud x\,\ud t \label{tw-emor7}\\
& \quad+\iint (\chi_R+\psi_R) u_{\leq T}P_{>T}[|u_{> T}|^3 F_2]\,\ud x\,\ud t \label{tw-emor8}
\end{align}
where all the integrals are taken over $[0,T^{1-\eps}]\times\R^3$. We treat each of these terms separately.

We first consider \eqref{tw-emor1}. Estimating as in the long-time Strichartz estimates and using Corollary \ref{C:tw-lts}, we obtain
\begin{align*}
&\left|\iint x^k \chi_R \partial_k u_{\leq T}P_{>T}[u_{\leq T}^p]\,\ud x\,\ud t\right|\\
& \lesssim RT^{1- s_p} \| \nabla_{x_{2,3}}u_{\leq T}\|_{L_t^\infty L_x^2}T^{s_p - 2} \||\nabla| P_{>T}(u_{\leq T})^p\|_{L_t^1 L_x^2} \\
&\lesssim RT^{1 - s_p + \nu}\cdot T^{s_p - 2} \| u_{\leq T}\|_{L_t^{\frac{2(p-1)}{2-\theta}}L_x^{\frac{2(p-1)}{\theta}}}^{p-1} \|\nabla u_{\leq T}\|_{L_t^{\frac{2}{\theta}}L_x^{\frac{2}{1-\theta}}} \\
&\lesssim R T^{1-s_p+\nu}.
\end{align*}

We next consider \eqref{tw-emor2}. We let $0 < \theta <\tfrac23\eps$.  By Bernstein's inequality,  Proposition~\ref{padditional}, and Corollary~\ref{C:tw-lts}, we obtain
\begin{align*}
&\left| \iint x^k \chi_R \partial_k u_{\leq T}[|u_{>T}| |u_{\leq T}|^{p-1}]\,\ud x\, \ud t \right| \\
& \lesssim R \|\nabla_{x_{2,3}} u_{\leq T}\|_{L^\infty_t L^2_x}  \|u_{\leq T}\|^{p-1}_{L_t^{p-1}L_x^\frac{2(p-1)}{\theta}}  \|u_{>T}\|_{L^{\infty}_t L^{^\frac{2}{1-\theta}}_x} \\
& \lesssim R T^{1- s_p} \|\nabla_{x_{2,3}} u_{\leq T}\|_{L^\infty_t L^2_x} T^{s_p - 1} \|u_{\leq T}\|^{p-1}_{L_t^{p-1}L_x^\frac{2(p-1)}{\theta}}  \|u_{>T}\|_{L^{\infty}_t L^{^\frac{2}{1-\theta}}_x} \\
& \lesssim R T^{1 - s_p + \nu}.
\end{align*}

For \eqref{tw-emor3} we again argue as in the proof of the long-time Strichartz estimates, and using Corollary~\ref{C:tw-lts}, we obtain
\begin{align*}
&\left| \iint x^k \chi_R \partial_k u_{\leq T}[|u_{>T}| ^2|u_{\leq T}|^{p-2}]\,\ud x\, \ud t \right| \\
& \lesssim R \|\nabla_{x_{2,3}} u_{\leq T}\|_{L^\infty_t L^2_x} \|u_{\leq T}\|^{p-2}_{L_t^{p-1}L_x^{\infty}}  \|u_{>T}\|_{L^{2(p-1)}_t L^{4}_x} \\
& \lesssim R T^{ \nu + s_p - 1} \||\nabla_{x_{2,3}}|^{1-\nu} u_{\leq T}\|_{L^\infty_t L^2_x} N^{1 - s_p} \|u_{\leq T}\|^{p-2}_{L_t^{p-1}L_x^{\infty}}  \|u_{>T}\|_{L^{2(p-1)}_t L^{4}_x} \\
& \lesssim R T^{1 - s_p + \nu}.
\end{align*}

For \eqref{tw-emor4}, we once again use the bounds from the proof of the long-time Strichartz estimates as well as Corollary \ref{cor:ltse_23}, and we obtain
\begin{align*}
&\left| \iint x^k \chi_R \partial_k u_{\leq T} P_{>T} [|u_{>T}| ^3F_2]\,\ud x\, \ud t \right| \\
& \lesssim R \|\nabla_{x_{2,3}} u_{\leq T}\|_{L^{\frac{2}{1-s_p}}_{t} L^{\frac{2}{s_p}}_x} \|P_{>T} [|u_{>T}| ^3F_2]\|_{L^{\frac{2}{1+s_p}}_t L^{\frac{2}{2-s_p}}_x} \\
& \lesssim R \sum_{N \leq T} N^{\nu_0} \left(\frac{T}{N}\right)^{\frac{(1-\eps)(1-s_p)}{2}} \| |\nabla_{x_{2,3}}|^{1- \nu_0} u_{N}\|_{L^{\frac{2}{1-s_p}}_{t} L^{\frac{2}{s_p}}_x} \\
& \lesssim R T^{\nu_0} \sum_{N \leq T}   \left(\frac{T}{N}\right)^{\frac{(1-\eps)(1-s_p)}{2}} N^{1-s_p} \\
& \lesssim R T^{1- s_p + \nu_0} \sum_{N \leq T}   \left(\frac{N}{T}\right)^{1- s_p -\frac{(1-\eps)(1-s_p)}{2}} \lesssim R T^{ 1- s_p + \nu_0}
\end{align*}
for any $\nu_0 > \nu$, where $\nu > 0$ is as in Proposition \ref{padditional}.

 Arguing analogously for the remaining terms, the estimates are almost identical, up to noting that by H\"older's inequality in the $x_{2,3}$ variables we have
 \begin{align}
\|(\chi_R + \psi_R) u_{\leq T}\|_{L^\infty_t L^2_{x}} \lesssim R^{s_p}  \| u_{\leq T}\|_{L^\infty_t L^2_{x_1}L^{\frac{2}{1-s_p}}_{x_2, x_3}},
 \end{align}
which is controlled by the $\dot H^{s_p}$ norm by the Sobolev embedding \eqref{equ:sobolev_embed}. Thus we obtain for \eqref{tw-emor5} - \eqref{tw-emor7} the estimates:
\begin{align}
\eqref{tw-emor5}  &\lesssim R^{s_p} T^{1- s_p} \| u_{\leq T}\|_{L^\infty_t L^2_{x_1} L^{\frac{2}{1-s_p}}_{x_2, x_3}}  T^{s_p - 2} \| |\nabla| P_{>T}[u_{\leq T}^p] \|_{L^1_t L^2_x} 
 \lesssim R^{s_p} ,\\\\
\eqref{tw-emor6} &\lesssim R^{s_p}  T^{1 - s_p} \| u_{\leq T}\|_{L^\infty_t L^2_{x_1}L^{\frac{2}{1-s_p}}_{x_2, x_3}}  T^{s_p - 1} \| P_{>T}[ u_{>T} u_{\leq T}^{p-1}] \|_{L^1_t L^2_x} \lesssim R^{s_p} ,\\\\
 \eqref{tw-emor7}& \lesssim R^{s_p} T^{1 - s_p} \| u_{\leq T}\|_{L^\infty_tL^2_{x_1} L^{\frac{2}{1-s_p}}_{x_2, x_3}}  T^{s_p - 1} \| P_{>T}[ u_{>T}^2 u_{\leq T}^{p-2}] \|_{L^1_t L^2_x}\lesssim R^{s_p}.
 \end{align}
For the last term \eqref{tw-emor8}, we note that 
 \[
 \frac{2}{s_p} \leq \frac{2}{1-s_p},
 \]
and we use H\"older's inequality in the $x_{2,3}$ variables to estimate
 \begin{align}
\eqref{tw-emor8} &\lesssim  \|(\chi_R + \psi_R) u_{\leq T}\|_{L^{\frac{2}{1-s_p}}_{t} L^{\frac{2}{s_p}}_x}  \| P_{>T}[ u_{>T}^3 F_2] \|_{L^{\frac{2}{1+s_p}}_{t} L^{\frac{2}{2-s_p}}_x}\\
& \lesssim  T^{\frac{(1-\eps)(1-s_p)}{2}} T^{\frac{1 -s_p}{2}} R^{2s_p - 1} \|(\chi_R + \psi_R) u_{\leq T}\|_{L^\infty_{t} L^2_{x_1} L^{\frac{2}{1-s_p}}_{x_{2,3}}} \\
& \lesssim  T^{1- s_p} R^{2s_p - 1} \|(\chi_R + \psi_R) u_{\leq T}\|_{L^\infty_{t} L^2_{x_1} L^{\frac{2}{1-s_p}}_{x_{2,3}}} .
\end{align}

Now, using \eqref{tw-mor-bd}, \eqref{TWM2}, and our estimates for \eqref{tw-emor1}--\eqref{tw-emor8}, we have established that
\begin{align*}
&\iint_{|x_{2,3}|\leq R}|u_{\leq T}|^{p+1}\,\ud x\,\ud t  \lesssim RT^{1+\nu-s_p}+R^{s_p} + R^{2s_p - 1}  T^{1 - s_p }. 
\end{align*}
We now choose $R=T^{\frac{1}{2}+}$ to obtain that the right-hand side is $o(T^{1-\eps})$.  This can be achieved provided $\nu+\eps<s_p-\frac12,$ and hence we complete the proof. 
\end{proof}

As mentioned above, with the Morawetz estimate Proposition~\ref{P:tw-morawetz} in hand, we can quickly rule out traveling waves.  The final ingredient we will need is the non-triviality for compact solutions appearing in Corollary~\ref{C:acax}.  Combining this corollary with Proposition~\ref{P:tw-morawetz}, we can now prove Proposition~\ref{p:hans_zero}.

\begin{proof}[Proof of Proposition~\ref{p:hans_zero}] Suppose toward a contradiction that $\vec u$ is a traveling wave critical element for \eqref{eq:nlw}.  It suffices to prove that 
\begin{equation}\label{tw-lb}
\int_0^{T^{1-\eps}}\int_{|x_{2,3}|\leq T^{\frac{1}{2}+}} |u_{\leq T}(t,x)|^{p+1}\,\ud x\,\ud t \gtrsim T^{1-\eps}
\end{equation}
for $T$ sufficiently large, as this contradicts Proposition~\ref{P:tw-morawetz}. By Corollary~\ref{C:acax}, the definition of the critical element, and the fact that $N(t)\equiv 1$, there exists $C\gg1$ and $T\gg 1$ large enough that
\begin{equation}\label{tw-lb2}
\int_{t_0}^{t_0+1} \int_{|x-x(t)|\leq C} |u_{\leq T}(t,x)|^{\frac{3(p-1)}{2}} \,\ud x\,\ud t \gtrsim_u 1
\end{equation}
for all $t_0\in\R$.  Recalling $|x(t) -(t,0,0)| \lesssim \sqrt{t}$ we see that for $T> C^{2}$, we have
\[
\{|x-x(t)| \leq C\} \subset \{|x_{2,3}|\leq T^{\frac{1}{2}+}\}
\]
for all $t\in[0,T^{1-\eps}]$.  Thus \eqref{tw-lb2} implies \eqref{tw-lb}, as desired. 
\end{proof}


\subsection{Additional  Regularity: Proof of Proposition~\ref{padditional}}
Our final task is to prove Proposition~\ref{padditional}, namely, additional regularity for traveling waves.  More precisely, we can establish additional regularity in the directions orthogonal to the direction of travel. 

Recall the notation $x=(x_1,x_{2,3})$.  We similarly use $\xi=(\xi_1,\xi_{2,3})$ for the frequency variable.  We also introduce the following modified Littlewood--Paley operators:

For $N,M\in 2^{\mathbb{Z}}$,  we let $\widehat P_{N,>M}$ be the Fourier multiplier operator that is equal to one where
\[
|\xi|\simeq N \qtq{and} |\xi_{2,3}|\gtrsim M.
\]
We let $\widehat P_{N, M} = \widehat P_{N,>2M} - \widehat P_{N,>M}$, and we let $P_N = \widehat P_{N,\leq M}+\widehat P_{N,>M}$.

\medskip

We will occasionally abuse notation slightly and apply these multipliers to a vector, where this should be taken to mean applying these multipliers component-wise. We note that this notation differs from that of the previous sections, however we would like to make explicit that $N$ corresponds to $\xi$ frequencies while $M$ to those of $\xi_{2,3}$.
%

We fix $\nu>0$. We begin with the observation that
\begin{align}\label{equ:order1}
\| P_{\leq N_0} u\|_{L_t^\infty \dot H_x^1} \lesssim_{N_0} 1.
\end{align}
We will choose the precise value of $N_0 \gg 1$ in the course of the proof.  On the other hand,
 we have 
\begin{align}\label{equ:lowm}
\sum_{N>N_0}& \| |\nabla_{x_{2,3}}|^{1-\nu} \widehat P_{N,\leq N^{\frac{s_p}{1-\nu}}}u(t)\|_{L_x^2}^2\\
 & \lesssim \sum_{N>N_0} N^{2[\frac{s_p}{1-\nu}(1-\nu)-s_p]}\| |\nabla|^{s_p}u_N(t)\|_{L_x^2}^2  \lesssim \| |\nabla|^{s_p} u(t)\|_{L_x^2}^2, 
\end{align}
Therefore, we are left to show that
\begin{equation}\label{tw-addreg-nts0}
\sum_{N \geq N_0}\sum_{C_0 N^{\frac{s_p}{1-\nu}}\leq M \leq N} M^{2(1-\nu)}\|\widehat P_{N, \geq M}u(t)\|^2_{L_x^2} \lesssim 1,
\end{equation}
for some fixed $C_0 > 0$ (uniformly in $t$). 

We will use a double Duhamel argument together with a frequency envelope to estimate this expression. We will estimate
\begin{align}
\| \widehat{P}_{N, \geq M} u(t_0) \|_{L^{2}(\bR^{3})}^2 &\simeq N^{-2s_p} \| \widehat{P}_{N, \geq M} u(t_0) \|_{\dot H^{s_p} (\bR^{3})}^2\\
& =  N^{-2s_p}  \langle  \widehat{P}_{N, \geq M} u(t_0),\, \widehat{P}_{N,  \geq M} u(t_0) \rangle_{\dot H^{s_p}(\bR^{3})}.
\end{align}
We will show that there exists a frequency envelope $\gamma_{M,N}$ such that
\[
 \| \widehat{P}_{N,  \geq M} u(t_0) \|_{\dot H^{s_p}(\bR^{3})} \lesssim \gamma_{N,M}(t_0)
\]
and such that
\begin{align}\label{equ:frequency_env}
\sum_{N \geq N_0} \biggl(\sum_{C_0 N^{\frac{s_p}{1-\nu}}\leq M \leq N}  M^{2(1 - \nu)} N^{-2s_p} \gamma_{N,M}(t_0)^2 \biggr) \lesssim 1.
\end{align}
Consequently, this will show that
\begin{align}
\sum_{N \geq N_0}&\sum_{C_0 N^{\frac{s_p}{1-\nu}}\leq M \leq N}  M^{2(1-\nu)}\|\widehat P_{N, \geq M}u(t_0)\|^2_{L_x^2} \\
&\lesssim \sum_{N \geq N_0}\sum_{C_0 N^{\frac{s_p}{1-\nu}}\leq M \leq N}  M^{2(1-\nu)} N^{-2s_p } \|\widehat P_{N, \geq M}u(t_0)\|^2_{\dot H^{s_p}_x} \lesssim 1.
\end{align}
Together with \eqref{equ:order1} and \eqref{equ:lowm} (and time translation invariance), this will imply that
\[
\| |\partial_{2}|^{1 - \nu} u \|_{L_{t}^{\infty} L_x^{2}} + \| |\partial_{3}|^{1 - \nu} u \|_{L_{t}^{\infty} L_x^{2}} < \infty,
\]
and hence prove Proposition \ref{padditional}. 
Thus, we let
\[
\Gamma_{N,M}(t_0) =  N^{s_p} \| \widehat{P}_{N, \geq M} u(t_0) \|_{L^{2}(\bR^{3})},
\]
and we fix some $\sigma > 0$ to be specified later. 
We define the frequency envelope
\begin{align}\label{equ:alpha_env}
\gamma_{M,N}(t_0) =\sum_{N'} \sum_{M' \leq M} \min\biggl\{\frac{N}{N'},\frac{N'}{N}\biggr\}^{\sigma}\cdot \left(\frac{M'}{M}\right)^{\sigma}\Gamma_{N',M'}(t_0).
\end{align}

By time-translation symmetry, it suffices to consider the case $t_0 = 0$.   Once again, we complexify the solution, letting 
 \[
 w = u+ \frac{i}{\sqrt{-\Delta}} u_t.
 \]
 Then
 \[
 \|w(t) \|_{\dot H^1} \simeq \|\vec u(t) \|_{\dot H^1 \times L^2},
 \]
 and if $\vec u(t)$ solves \eqref{eq:nlw}, then $w(t)$ is a solution to
 \begin{align}
 w_t = -i \sqrt{-\Delta} w \pm \frac{i}{\sqrt{-\Delta}} |u|^{p-1} u.
 \end{align}
 By Duhamel's principle, for any $T$, we have
 \[
 w(0) = e^{iT \sqrt{-\Delta}} w(T) \pm \frac{i}{\sqrt{-\Delta}} \int_T^0 e^{i \tau \sqrt{-\Delta}} F(u)(\tau) d\tau
 \]
 where $F(u) = |u|^{p-1} u$.  To estimate $\gamma_{N,M}$, we write
\begin{align}
\widehat{P}_{N,  \geq M} w(0) &= \widehat{P}_{N,  \geq M} e^{-iT\sqrt{-\De}} w(T) - \frac{1}{\sqrt{-\De}} \int_0^T e^{-it\sqrt{-\De}} \widehat{P}_{N,  \geq M} F(u) \ud t \\
&=\widehat{P}_{N,   \geq M} e^{-iT\sqrt{-\De}} w(-T) - \frac{1}{\sqrt{-\De}} \int_{-T}^0 e^{-i\tau\sqrt{-\De}} \widehat{P}_{N,  \geq M} F(u) \ud\tau 
\end{align}
When we pair these expressions and take $T \to \infty$, we use the facts that
\[
e^{-iT\sqrt{-\De}} \widehat{P}_{N,  \geq M}  w(T) \rightharpoonup 0 \quad\text{and}\quad e^{iT\sqrt{-\De}} \widehat{P}_{N,  \geq M}  w(-T) \rightharpoonup 0,
\]
and ultimately we are left to estimate
\begin{align}\label{equ:double_duhamel_term_hans}
\left \langle  \int_0^\infty S(-t) \widehat{P}_{N,  \geq M} F(u) \ud t ,\, \int_{-\infty}^0 S(-\tau) \widehat{P}_{N,  \geq M} F(u) d\tau \right \rangle_{\dot H^{s_p}_x}.
\end{align}
where we have introduced the notation 
\EQ{
S(t):= \frac{1}{\sqrt{-\De}} e^{i t\sqrt{-\De}}
}
above.

As we have done in previous sections, we will estimate this expression by dividing space-time into three regions: a compact time interval, an outer region, and a region inside the light-cone. We note, however, that the arguments on the compact time interval and the region inside the light cone will be considerably different than in previous sections. 

Thus, we let $\eta_0>0$ and $\eps>0$ be sufficiently small parameters and define the smooth cut-off
\begin{align}\label{equ:chi_0}
\chi_{0}(t,x) = 1_{\{ |x - x(N^{1-\eps})| \geq R(\eta_0) + (t - N^{1 - \eps}),\, t \geq N^{1 - \eps}\}},
\end{align}
where $R(\eta_0)$ is such that
\begin{align}\label{equ:small1}
\| \chi_{0}(N^{1-\eps},x)u(N^{1 - \eps}, x) \|_{\dot{H}^{s_p}} + \| \chi_0(N^{1-\eps},x) \partial_t u(N^{1 - \eps}, x) \|_{\dot{H}^{s_p - 1}} \leq \eta_0.
\end{align}

By the small-data theory, we may solve the Cauchy problem
\begin{equation}
 \left\{ \begin{aligned}
&v_{tt} - \Delta v +  F(v) = 0 \text{ on } \bR \times \bR^3, \\
 & (v, \partial_t v)|_{t={0}} = \bigl( \chi_{0}(N^{1-\eps},x) u(N^{1- \eps}, x),  \chi_{0}(N^{1-\eps},x) u_{t}(N^{1- \eps}, x)\bigr) \in \dot \H^{s_p}(\bR^3).
 \end{aligned} \right.
\end{equation}

Note that by finite propagation speed, $v = u$ on the set
\begin{align}\label{equ:supp_set}
\{ (t,x) : |x - x(N^{1-\eps})| \geq R(\eta_0) + (t - N^{1 - \eps}), \,\, t \geq N^{1 - \eps} \}.
\end{align}

We now write
\begin{align*}
&\int_0^\infty S(-t) \widehat P_{N,  \geq M} F(u)\,\ud t = A+B+C,
\end{align*}
where 
\begin{equation}\label{tw-ABC}
\begin{aligned}
& A = \int_{N^{1-\eps}}^\infty  S(-t) \widehat P_{N,  \geq M} F(v)\,\ud t, \\
& B = \int_0^{N^{1-\eps}} S(-t)\widehat P_{N, \geq M} F(u)\,\ud t,\\
& C = \int_{N^{1-\eps}}^\infty S(-t)\widehat P_{N,  \geq M}[ F(u)- F(v)]\,\ud t 
\end{aligned}
\end{equation}
and perform a similar decomposition in the negative time direction, yielding quantities $A',B',C'$. We will use the estimate
\begin{equation}\label{eq:ABC}
\begin{split}
|\langle A+B+&C,A'+B'+C'\rangle| \\
& \lesssim \|A\|_{\dot H^{s_p}_x}^2 + \|A'\|_{\dot H^{s_p}_x}^2+\|B\|_{\dot H^{s_p}_x}^2+\|B'\|_{\dot H^{s_p}_x}^2+|\langle C,C'\rangle_{\dot H^{s_p}_x}|
\end{split}
\end{equation}
whenever $A+B+C=A'+B'+C'$.

\subsection*{Term A} We first estimate $\langle A,A\rangle_{\dot H^{s_p}_x}$ and $\langle A',A'\rangle_{\dot H^{s_p}_x}$, where 
\[
A = \int_{N^{1 - \eps}}^\infty  S(-t) \widehat{P}_{N, \geq M} F(v) \ud t\qtq{and} A' = \int_{-\infty}^{-N^{1 - \eps}} S(-\tau) \widehat{P}_{N, \geq M} F(v) d\tau.
\]  
We introduce two parameters $q$ and $r$ satisfying
\begin{equation}\label{tw-parameters}
2<q<\min\{p-1,\tfrac{2}{s_p},\tfrac{5p-9}{3p-7}\}\qtq{and}\tfrac{2}{s_p}\leq r\leq\min\{\tfrac{2p(p-1)}{2p-3},4+\}. 
\end{equation}
and let $I = [N^{1 - \eps}, \infty)$. We fix $\sigma > 0$ to be specified later, and we define
\[
a_{N', M} = \biggl[ (N')^{-(\frac{2}{q} -s_p ) } \| \widehat{P}_{N', \geq M} v \|_{L_{t}^{q} L_{x}^{\frac{2q}{q-2}}} + (N')^{\frac{2}{r} - 1 + s_p} \| \widehat{P}_{N', \geq M} v \|_{L_{t}^{\frac{2r}{r-2}} L_{x}^{r}}\biggr],
\]
and let
\begin{align}
\alpha_{N,M}=\sum_{N'} \sum_{M' \leq M} \left( \frac{M'}{M}\right)^\sigma&\min\biggl\{\frac{N}{N'},\frac{N'}{N}\biggr\}^{\sigma}a_{N', M'}.
\end{align}
All space-time norms are taken over $I\times\R^3$. 
Our goal is to prove the following result.
\begin{lem}\label{l:AA'} 
Let $A, A'$ and $\gamma_{N,M}$ be as above, then
\begin{align}\label{equ:a_ests}
\sum_{N'}\sum_{M' \leq M} \left( \frac{M'}{M} \right)^\sigma \min\biggl\{\frac{N}{N'},\frac{N'}{N}\biggr\}^{\sigma} (\|A\|_{\dot H^{s_p}_x} + \|A'\|_{\dot H^{s_p}_x})
&\lesssim \eta_0^{p-1}\, \alpha_{N, M},
\end{align}
and we also have
\begin{align}\label{alpha_bds}
\alpha_{N, M} \lesssim \gamma_{N,M}(N^{1-\eps}) + \eta_0^{p-1} \alpha_{N, M}.
\end{align}
\end{lem}
\begin{proof}
On this region, we will use the small data theory, which implies, in particular, that
\begin{align} \label{equ:small_v}
\| v \|_{L_{t,x}^{2(p-1)}(\R\times\R^3)}  \lesssim \eta_0.
\end{align}

By Strichartz estimates, we may write
\begin{align}\label{equ:first_a_bds}
&N^{-\left(\frac{2}{q} -s_p \right) } \| \widehat{P}_{N,  \geq M} v \|_{L_{t}^{q} L_{x}^{\frac{2q}{q-2}}} + N^{\frac{2}{r} - 1 + s_p} \| \widehat{P}_{N,  \geq M} v \|_{L_{t}^{\frac{2r}{r-2}} L_{x}^{r}} \\
& \lesssim \| \widehat{P}_{N, \geq  M} (v, v_{t})(N^{1 - \eps}) \|_{\dot{\cH}^{s_p}}  + \|P_{N, \geq M} F(v)\|_{N(\R)},
\end{align}
and recall that, by definition, we have
\begin{align}\label{equ:data_bds}
\| \widehat{P}_{N, \geq M} (v, v_{t})(N^{1 - \eps}) \|_{\dot{H}^{s_p}} \lesssim \Gamma_{N,M}(N^{1-\eps}).
\end{align}
Here, we let the norm $\|F\|_{N(\R)}$ denote any finite combination $\sum_j \|F_j\|_{N_j(\R)}$, with $F=\sum_j F_j$ and each $N_j(\R)$ being a dual admissible Strichartz space with the appropriate scaling and number of derivatives.  

It will be useful to introduce the quantities
\[
v_{lo} = \sum_{N'} \widehat{P}_{N',\leq M}v\qtq{and} v_{hi} = \sum_{N'}\widehat{P}_{N',>M}v,
\]
where low and high is meant to refer to the $\xi_{2,3}$ frequency component. We decompose the nonlinearity via
\[
F(v)=F(v_{lo})+v_{hi}\int_0^1 F'(v_{lo}+\theta v_{hi})\,d\theta,
\]
which we write schematically as
\[
F(v)=F(v_{lo})+v_{hi} v^{p-1}.
\]

For the high frequency (in $M$) component, we write
\[
v_{hi} v^{p-1} =  (P_{\leq N} v_{hi} ) v^{p-1}  + (P_{\geq N} v_{hi} ) v^{p-1},
\]
and to estimate these terms we may use the dual Strichartz spaces
\[
L_t^{\frac{2q}{q+2}} \dot H_x^{-(\frac2q-s_p),\frac{q}{q-1}} \quad \textup{and} \quad  L_t^{\frac{r}{r-1}}\dot H_x^{\frac2r-1 +s_p,\frac{2r}{r+2}  }
\]
respectively. This yields
\begin{align}
\|&P_{N', \geq M}  v_{hi} v^{p-1}  \|_{N(\R)} \\
& \lesssim (N')^{-\left(\frac{2}{q} -s_p \right) } \sum_{N'' \leq N'} \| \widehat{P}_{N'', \geq M} v \|_{L_{t}^{q} L_{x}^{\frac{2q}{q-2}}(I \times \bR^{3})} \| v \|_{L_{t,x}^{2(p-1)}}^{p-1}\\
& + (N')^{\frac{2}{r} - 1 + s_p} \sum_{N'' \geq N'} \| \widehat{P}_{N'', \geq M } v \|_{L_{t}^{\frac{2r}{r-2}}L_{x}^{r}  (I \times \bR^{3})} \| v \|_{L_{t,x}^{2(p-1)}}^{p-1}\phantom{\int}\\
& := \NN_1 + \NN_2
\end{align}
hence we conclude that
\begin{align}
 \sum_{N'}\min\biggl\{\frac{N}{N'},\frac{N'}{N}\biggr\}^{\sigma}& \|P_{N', \geq M}  v_{hi} v^{p-1}  \|_{N(\R)} \\
  & \lesssim \sum_{N'}\min\biggl\{\frac{N}{N'},\frac{N'}{N}\biggr\}^{\sigma} (\NN_1 + \NN_2).
\end{align}
Thus, we argue in order to bound the $\NN_1$ and $\NN_2$ terms. 
We only treat the first term as an example since the other term follows analogously. We obtain
\begin{align}
&\sum_{N' \leq N}\sum_{N'' \leq N'} \left( \frac{N'}{N}\right)^{\sigma} (N')^{-\left(\frac{2}{q} -s_p \right) }  \| \widehat{P}_{N'', \geq M} v \|_{L_{t}^{q} L_{x}^{\frac{2q}{q-2}}}\\
& \lesssim \sum_{N' \leq N} \sum_{N'' \leq N'} \left( \frac{N'}{N}\right)^{\sigma}\left(\frac{N''}{N'}\right)^{\frac2q-s_p} (N'')^{-\left(\frac{2}{q} -s_p \right) }  \| \widehat{P}_{N'', \geq M} v \|_{L_{t}^{q} L_{x}^{\frac{2q}{q-2}}}\\
&\lesssim\sum_{N'' \leq N} \left( \frac{N''}{N}\right)^{\sigma}  a_{N'', M}  \sum_{N'\geq N''}\left(\frac{N'}{N''}\right)^{-(\frac{2}{q}-s_p)+\sigma}.
\end{align}
Hence this term can be bounded by $ \alpha_{N, M}$ provided $\sigma<\frac{2}{q}-s_p.$

We also have
\begin{align}
&\sum_{N' \geq N}  \left( \frac{N}{N'}\right)^{\sigma} (N')^{-\left(\frac{2}{q} -s_p \right) } \sum_{N'' \leq N'} \| \widehat{P}_{N'', \geq M} v \|_{L_{t}^{q} L_{x}^{\frac{2q}{q-2}}(I \times \bR^{3})}\\
&\lesssim \sum_{N'' \leq N}  \left( \frac{N''}{N}\right)^{\sigma}  a_{N'', M}  \sum_{N'' \leq N', N \leq N'} \left( \frac{N^2}{N' N''}\right)^{\sigma}   \left(\frac{N '}{N''} \right)^{-\left(\frac{2}{q} -s_p \right) } \\
& \hspace{14mm} + \sum_{N'' \geq N} \left( \frac{N}{N''}\right)^{\sigma}  a_{N'', M}  \sum_{N'' \leq N'} \left( \frac{N''}{N'}\right)^{\sigma} \left(\frac{N '}{N''} \right)^{-\left(\frac{2}{q} -s_p \right) } .
\end{align}
Using that in the first term,
\[
 \left( \frac{N^2}{N' N''}\right)^{\sigma}  \leq  \left( \frac{(N')^2}{N' N''}\right)^{\sigma}  =  \left( \frac{N'}{N''}\right)^{\sigma},
 \]
 we can bound this expression by $\gamma_{N,M}$ provided
 \[
   \sum_{N'' \leq N'} \left( \frac{N'}{N''}\right)^{\sigma} \left(\frac{N '}{N''} \right)^{-\left(\frac{2}{q} -s_p \right) } \lesssim   1 \quad \Longleftrightarrow \quad \sigma < \frac{2}{q} - s_p,
\]
and so we obtain
\begin{align}\label{gamma_bd1}
 \sum_{N'} \sum_{M' \leq M} \left( \frac{M'}{M} \right)^\sigma& \min\biggl\{\frac{N}{N'},\frac{N'}{N}\biggr\}^{\sigma} \|P_{N', \geq M'} v_{hi} v^{p-1}  \|_{N(\R)}  \\
 & \lesssim \eta_0^{p-1} \alpha_{N, M}
\end{align}
Next we estimate the contribution of the low-frequency piece.  We can write
\begin{align}
\|P_{N, \geq M} F(v_{lo})\|_{N(\R)} \leq M^{ - 2 } \|P_{N, \geq M} \Delta_{x_2,x_3} F(v_{lo})\|_{N(\R)}.
\end{align}
applying the chain rule and decomposing $v = P_{\geq N} v + P_{\leq N} v$, we obtain for $j=2,3$ that
\begin{align}
\partial_{x_j} F(v_{lo}) &= \partial_{x_j}v_{lo} F'(v_{lo})\\
& = \partial_{x_j} P_{\leq N} v_{lo} F'(v_{lo})  +  \partial_{x_j} P_{\geq N} v_{lo} F'(v_{lo}) ,
\end{align}
and hence
\begin{align}
\|P_{N, \geq M} \Delta_{x_2,x_3} F(v_{lo})\|_{N(\R)} &\leq \sum_{j=1}^2 \|\partial_{x_j} P_{N, \geq M}   (\partial_{x_j}v_{lo} )F'(v_{lo})\|_{N(\R)}\\
& \leq \sum_{j=1}^2 M \|P_{N, \geq M}   (\partial_{x_j}v_{lo} )F'(v_{lo})\|_{N(\R)}.
\end{align}
Estimating as above, using the dual Strichartz spaces
\[
L_t^{\frac{2q}{q+2}} \dot H_x^{-(\frac2q-s_p),\frac{q}{q-1}} \quad \textup{and} \quad  L_t^{\frac{r}{r-1}}\dot H_x^{\frac2r-1 +s_p,\frac{2r}{r+2}  },
\]
we conclude that
\begin{align}
&\|P_{N,\geq M}F(v_{lo})  \|_{N(\R)} \\
& \lesssim N^{-\left(\frac{2}{q} -s_p \right) } \sum_{N' \leq N} \sum_{M' \leq M} M^{-1} \| \nabla_{x_2,x_3} \widehat{P}_{N', M'} v_{lo} \|_{L_{t}^{q} L_{x}^{\frac{2q}{q-2}}(I \times \bR^{3})} \| v \|_{L_{t,x}^{2(p-1)}}^{p-1}\\
& + N^{\frac{2}{r} - 1 + s_p} \sum_{N' \geq N} \sum_{M' \leq M} M^{-1}  \| \nabla_{x_2,x_3} \widehat{P}_{N', M' } v_{lo} \|_{L_{t}^{\frac{2r}{r-2}}L_{x}^{r}  (I \times \bR^{3})} \| v \|_{L_{t,x}^{2(p-1)}}^{p-1}\phantom{\int}\\
& \lesssim N^{-\left(\frac{2}{q} -s_p \right) } \sum_{N' \leq N} \sum_{M' \leq M} \left( \frac{M'}{M} \right)  \| \widehat{P}_{N', M'} v_{lo} \|_{L_{t}^{q} L_{x}^{\frac{2q}{q-2}}(I \times \bR^{3})} \| v \|_{L_{t,x}^{2(p-1)}}^{p-1}\\
& + N^{\frac{2}{r} - 1 + s_p} \sum_{N' \geq N} \sum_{M' \leq M} \left( \frac{M'}{M} \right)  \| \widehat{P}_{N', M' } v_{lo} \|_{L_{t}^{\frac{2r}{r-2}}L_{x}^{r}  (I \times \bR^{3})} \| v \|_{L_{t,x}^{2(p-1)}}^{p-1}\phantom{\int}.
\end{align}
To establish a bound for this expression, it is useful to introduce the notation
\[
\widetilde{a}_{N, M'} = \sum_{N'} \min\biggl\{\frac{N}{N'},\frac{N'}{N}\biggr\}^{\sigma}a_{N', M'}.
\]
Thus summing over $N$ and $M' \leq M$, we can again argue exactly as above to bound this expression by 
\begin{align}
&\eta_0^{p-1} \sum_{M' \leq M} \left( \frac{M'}{M} \right)^\sigma \sum_{M'' \leq M'} \left( \frac{M''}{M'} \right) \widetilde{a}_{N, M''} \\
&\leq \eta_0^{p-1} \sum_{M'' \leq M} \left( \frac{M''}{M} \right)^\sigma \sum_{M'' \leq M'} \left( \frac{M''}{M'} \right) \left( \frac{M'}{M''}\right)^\sigma \widetilde{a}_{N, M''}  \\
&\leq \eta_0^{p-1} \sum_{M'' \leq M} \left( \frac{M''}{M} \right)^\sigma \widetilde{a}_{N, M''}  \sum_{M'' \leq M'} \left( \frac{M''}{M'} \right) \left( \frac{M'}{M''}\right)^\sigma\\
& \lesssim \eta_0^{p-1} \alpha_{N, M}
\end{align}
provided $\sigma < 1$. Thus, we obtain
\begin{align}
\sum_{N'} &\sum_{M' \leq M} \left( \frac{M'}{M} \right)^\sigma \min\biggl\{\frac{N}{N'},\frac{N'}{N}\biggr\}^{\sigma}  \|P_{N, \geq M} F(v_{lo})  \|_{N(\R)}  \\ & \lesssim \eta_0^{p-1} \alpha_{N, M}.\label{gamma_bd2}
\end{align}

By Strichartz estimates, we have
\[
\|A'\|_{\dot H^{s_p}_x} + \|A\|_{\dot H^{s_p}_x} \lesssim \|P_{N, \geq M} F(v)\|_{N(I)}.
\]
Hence, putting these bounds together with \eqref{gamma_bd1} and \eqref{gamma_bd2} we obtain
\begin{align}
\sum_{N'}\sum_{M' \leq M} \left( \frac{M'}{M} \right)^\sigma \min\biggl\{\frac{N}{N'},\frac{N'}{N}\biggr\}^{\sigma} (\|A\|_{\dot H^{s_p}_x} + \|A'\|_{\dot H^{s_p}_x})
&\lesssim \eta_0^{p-1}\, \alpha_{N, M}.
\end{align}
Together with \eqref{equ:first_a_bds} and \eqref{equ:data_bds}, we also have
\begin{align}
\alpha_{N, M} \lesssim \gamma_{N,M}(N^{1-\eps}) + \eta_0^{p-1} \alpha_{N, M},
\end{align}
as required.
\end{proof}

\subsection*{Term B}  We next estimate the terms $\langle B,B\rangle$ and $\langle B',B'\rangle$ from \eqref{tw-ABC}.  On this region, we use the long-time Strichartz estimates (Proposition~\ref{P:tw-lts}) and another frequency envelope argument for this contribution. In the following we suppose, unless otherwise specified, that norms be taken over 
\[
I:= [0,  N^{1 - \eps}] \times \bR^{3}.
\]
We define
\begin{align}
b_{N', M} &=  \biggl[  (N')^{-\left(\frac{2}{q} -s_p \right) } \| \widehat{P}_{N', \geq M} u \|_{L_{t}^{q} L_{x}^{\frac{2q}{q-2}}} +  (N')^{\frac{3(p-3)}{2(p-1)}} \| P_{N',  \geq M} u\|_{L_t^{2(p-1)} L_x^{\frac{2(p-1)}{p-2}}}  \\
& + (N')^{-\frac{1}{p-1} + \frac{3\theta}{2(p-1)}}  \|P_{N',  \geq M} u  \|_{L_t^{p-1}L_x^\frac{2(p-1)}{\theta}}   \\
&+(N')^{s_p - \frac{3\theta}{2}}  \| P_{N',  \geq M} u\|_{L^{\infty}_t L^{^\frac{2}{1-\theta}}_x} +  (N')^{- \theta + s_p } \| P_{N' ,  \geq M} u \|_{L^{\frac{2}{\theta}}_t L^{\frac{2}{1-\theta}}_x} \\
& + (N')^{\frac{\ell}{2}} \|P_{N', \geq M} u \|_{L^\frac{2p}{1+s_p - p\ell}_t L^\frac{2p}{2-s_p}_x } + (N')^{\frac{2 s_p}{3} - \frac{1}{3} } \|  P_{N',\geq M}\|_{L^\frac{6}{1+s_p}_t L^\frac{6}{2-s_p}_x } \biggr],
\end{align}
where $\ell>0$ will be determined more precisely below and $\theta$ is as in Proposition~\ref{P:tw-lts}.  These are just a collection of admissible Strichartz pairs at regularity $s_p$. We then define the frequency envelope
\begin{align} \label{equ:beta_env}
\beta_{N,M} = \sum_{N'} \sum_{M' \leq M} \left( \frac{M'}{M}\right)^\sigma \min\biggl\{\frac{N}{N'},\frac{N'}{N}\biggr\}^{\sigma} b_{N', M'} .
\end{align}
Our goal in this section is to prove the following result.
\begin{lem}\label{l:BB'} 
Let $B, B'$ and $\beta_{N,M}$ be as above. Then
\begin{align}\label{equ:b_ests}
\sum_{N'}\hspace{-.5mm} \sum_{M' \leq M} \hspace{-1mm} \left( \frac{M'}{M} \right)^\sigma\hspace{-1mm} \min\biggl\{\frac{N}{N'},\frac{N'}{N}\biggr\}^{\sigma}\hspace{-1.5mm} (\|B\|_{\dot H^{s_p}_x} + \|B'\|_{\dot H^{s_p}_x})
&\lesssim   \eta_0^{p-1}\, \beta_{N, M},
\end{align}
and we also have
\begin{align}\label{equ:beta_bds}
\gamma_{N,M}(N^{1-\eps}) + \beta_{N, M} \lesssim \gamma_{N, M}(0) + \eta_0^{p-1} \beta_{N, M}.
\end{align}
\end{lem}

\begin{proof}
Fix $t_0 = 0$. Throughout, we will assume that $N \gg N_0$ as in the statement of the long-time Strichartz estimates. By Strichartz estimates
\begin{align}\label{equ:first_b_bds}
\|P_{N', \geq M} (u, u_t) \|_{L^\infty_t \dot \cH^{s_p}([0, N^{1-\eps}])}  &+ b_{N',M} \\
& \lesssim \|P_{N', \geq M} (u, u_t)(0) \|_{\cH^{s_p}} + \|P_{N', \geq M} F \|_{N(I)}.
\end{align}
Once again, it will be useful to introduce the quantities
\[
u_{lo} = \sum_{N'} \widehat{P}_{N',\leq M}u\qtq{and} u_{hi} = \sum_{N'}\widehat{P}_{N',>M}u,
\]
where low and high is meant to refer to the $\xi_{2,3}$ frequency component. We decompose the nonlinearity via
\[
F(u)=F(u_{lo})+u_{hi}\int_0^1 F'(u_{lo}+\theta u_{hi})\,d\theta,
\]
which we write schematically as
\[
F(u)=F(u_{lo})+u_{hi} u^{p-1}.
\]
These two expressions will be estimated almost identically to, up to requiring additional exponential gains for the low frequency (in $M$) term, $F(u_{lo})$. We will only estimate this term since the other is easier. Arguing as above via the chain rule with the Laplacian in the $x_{2,3}$ directions, we have
\begin{align}
\|P_{N, \geq M} F(u_{lo})\|_{N(I)} \leq M^{-1} \| P_{N, \geq M} (\nabla_{2,3} u_{lo})\: F'(u_{lo})\|_{N(I)}.
\end{align}
We write
\begin{equation}\label{EQX}
(\nabla_{2,3} u_{lo})\: F'(u_{lo}) = (\nabla_{2,3} P_{\geq N} u_{lo})\: F'(u_{lo}) + (\nabla_{2,3} P_{\leq N} u_{lo})\: F'(u_{lo}) := 1 + 2
\end{equation}
and we begin with Term 1. We set
\[
P_{\geq N} u_{lo} := u_{lo, \geq N}, \quad P_{\leq N} u_{lo} := u_{lo, \leq N},
\]
and decompose
\begin{align}
 (\nabla_{2,3} u_{lo, \geq N})\: F'(v_{lo}) &=   (\nabla_{2,3} u_{lo, \geq N})\: F'(u_{lo, \leq N})  \\
 & \hspace{4mm}+ (\nabla_{2,3} u_{lo, \geq N}) u_{lo, \geq N} \int_0^1 F''(u_{lo, \leq N} + \theta u_{lo, \geq N}) \\
& = (\nabla_{2,3} u_{lo, \geq N})\: F'(u_{lo, \leq N}) + \nabla_{2,3} u_{lo, \geq N} u_{lo, \geq N} F''(u_{lo, \leq N})\\
 & \hspace{4mm}+   (\nabla_{2,3} u_{lo, \geq N})u_{lo, \geq N}^2 \iint_0^1 F'''(u_{lo, \leq N} + \theta_1 \theta_2 u_{lo, \geq N}) 
 \\
 & := 1.I + 1.II + 1.III.
\end{align}

\subsection*{Term 1.I}
We estimate using Corollary~\ref{C:tw-lts} to get
\begin{align}
&M^{-1} \||\nabla|^{s_p - 1}P_{N} \nabla_{2,3} u_{lo, \geq N}\: F'(u_{lo, \leq N}) \|_{L^1_t L^2_x} \\
& \leq M^{-1} N^{s_p - 1} \|\nabla_{2,3} u_{lo, \geq N}\: F'(u_{lo, \leq N}) \|_{L^1_t L^2_x}\\
&  \leq M^{-1} N^{s_p - 1} \|u_{lo, \leq N} \|^{p-1}_{L_t^{p-1}L_x^\frac{2(p-1)}{\theta}}  \|\nabla_{2,3} u_{lo, \geq N}\|_{L^{\infty}_t L^{\frac{2}{1-\theta}}_x}\\
&  \leq N^{s_p - 1} \|u_{lo, \leq N} \|^{p-1}_{L_t^{p-1}L_x^\frac{2(p-1)}{\theta}}  \sum_{M' \leq M} \sum_{N' \geq N} \left(\frac{M'}{M} \right)  \|P_{N', M'} u\|_{L^{\infty}_t L^{\frac{2}{1-\theta}}_x}\\
&  \leq N^{s_p - 1} N^{-s_p + \frac{3\theta}{2}} \|u_{lo, \leq N} \|^{p-1}_{L_t^{p-1}L_x^\frac{2(p-1)}{\theta}} \\
& \hspace{14mm} \times \sum_{M' \leq M} \sum_{N' \geq N} \left(\frac{M'}{M} \right) \left(\frac{N'}{N} \right)^{-s_p + \frac{3\theta}{2}}  (N')^{s_p - \frac{3\theta}{2}}  \|P_{N', M'} u\|_{L^{\infty}_t L^{\frac{2}{1-\theta}}_x}\\
&  \leq \eta_0^{p-1} \sum_{M' \leq M} \sum_{N' \geq N} \left(\frac{M'}{M} \right) \left(\frac{N'}{N} \right)^{-s_p + \frac{3\theta}{2}}  b_{N', M'}.
\end{align}
\subsection*{Term 1.II}
We estimate
\begin{align}
&M^{-1} \||\nabla|^{s_p - 1}P_{N} \nabla_{2,3} u_{lo, \geq N}\: u_{lo, \geq N} F''(u_{lo, \leq N}) \|_{L^1_t L^2_x}\\
& \lesssim M^{-1} N^{ s_p-1 } \|\nabla_{2,3} u_{lo, \geq N}\|_{L^{2(p-1)}_t L^4_x} \|u_{lo, \geq N}\|_{L^{2(p-1)}_t L^4_x} \|u_{lo, \leq N} \|_{L^{p-1}_t L^\infty_x}^{p-2}\\
& \lesssim M^{-1} N^{ s_p-1 } \|\nabla_{2,3} u_{lo, \geq N}\|_{L^{2(p-1)}_t L^4_x} \|u_{lo, \geq N}\|_{L^{2(p-1)}_t L^4_x} \|u_{lo, \leq N} \|_{L^{p-1}_t L^\infty_x}^{p-2}\\
& \lesssim \eta_0^{p-1} \sum_{M' \leq M} \sum_{N' \geq N} \left(\frac{M'}{M}\right) \left(\frac{N}{N'} \right)^{ \frac{3}{4} - \frac{3}{2(p-1)} }  (N')^{ \frac{3}{4} - \frac{3}{2(p-1)} }\|P_{N', M'} u\|_{L^{2(p-1)}_t L^4_x} \\
& \lesssim \eta_0^{p-1} \sum_{M' \leq M} \sum_{N' \geq N} \left(\frac{M'}{M}\right) \left(\frac{N}{N'} \right)^{ \frac{3}{4} - \frac{3}{2(p-1)} }  b_{N', M'} .
\end{align}

\subsection*{Term 1.III}
As in the proof of Term IV in the long-time Strichartz estimates, there are two terms. For the first we estimate
\begin{align}
&M^{-1} \|\nabla_{2,3} u_{lo, \geq N}\: u_{lo, \geq N}^2u_{lo, \geq N}^{p-3}  \|_{L^\frac{2}{1+s_p}_t L^\frac{2}{2-s_p}_x }\\
& \lesssim M^{-1} N^{\frac{\ell}{2}} \|\nabla_{2,3} u_{lo, \geq N}\: u_{lo, \geq N}^2 u_{lo, \geq N}^{p-3}   \|_{L^{\frac{2}{1+s_p - \ell}}_t L^\frac{2}{2-s_p}_x }\\
& \lesssim  M^{-1} N^{\frac{\ell}{2}} \|u_{> \frac{N}{8}} \|^{p-1}_{L^\frac{2p}{1+s_p}_t L^\frac{2p}{2-s_p}_x } \|u_{lo, \geq N} \|_{L^\frac{2p}{1+s_p - p\ell}_t L^\frac{2p}{2-s_p}_x }\\
& \lesssim  \eta_0^{p-1}  \sum_{M' \leq M} \sum_{N' \geq N} \left(\frac{M'}{M}\right) \left(\frac{N}{N'} \right)^{\frac{\ell}{2}}  (N')^{\frac{\ell}{2}} \|P_{N', M'} u \|_{L^\frac{2p}{1+s_p - p\ell}_t L^\frac{2p}{2-s_p}_x }\\
& \lesssim  \eta_0^{p-1}  \sum_{M' \leq M} \sum_{N' \geq N} \left(\frac{M'}{M}\right) \left(\frac{N}{N'} \right)^{\frac{\ell}{2}}  b_{N',M'},
\end{align}
where we have used that for $p > 3$ and $\ell > 0$, the pair
\[
\left( \frac{2p}{1+s_p - p\ell}, \frac{2p}{2-s_p} \right)
\]
is wave admissible at regularity
\[
\frac{3}{2} - \frac{1+s_p - p\ell}{2p} - \frac{6-3s_p}{2p} = s_p -  \frac{\ell}{2}.
\]

For the second term, we have
\begin{align}
&M^{-1} \|\nabla_{2,3} u_{lo, \geq N}\: u_{lo, \geq N}^2 u_{lo, \leq N}^{p-3}  \|_{L^\frac{2}{1+s_p}_t L^\frac{2}{2-s_p}_x }\\
& \lesssim  \eta_0^{p-1}  \sum_{M' \leq M} \sum_{N' \geq N} \left(\frac{M'}{M}\right) \left(\frac{N}{N'} \right)^{\frac{2 s_p}{3} - \frac{1}{3} }(N')^{\frac{2 s_p}{3} - \frac{1}{3} } \|  P_{N',M'} u\|_{L^\frac{6}{1+s_p}_t L^\frac{6}{2-s_p}_x } \\
& \lesssim  \eta_0^{p-1}  \sum_{M' \leq M} \sum_{N' \geq N} \left(\frac{M'}{M}\right) \left(\frac{N}{N'} \right)^{\frac{2 s_p}{3} - \frac{1}{3} } b_{N', M'}.
\end{align}

This completes the estimation of Term 1 in \eqref{EQX}.

Now we turn to Term 2 in \eqref{EQX}, namely
\[
\nabla_{2,3} P_{\leq N} u_{lo}\: F'(u_{lo}).
\]
We decompose
\begin{align}
 \nabla_{2,3} u_{lo, \leq N}\: F'(v_{lo}) &=   \nabla_{2,3} u_{lo, \leq N}\: F'(u_{lo, \leq N})  \\
 & \hspace{4mm}+ \nabla_{2,3} u_{lo, \leq N}  u_{lo, \geq N} \int_0^1 F''(u_{lo, \leq N} + \theta u_{lo, \geq N}) \\
& = \nabla_{2,3} u_{lo, \leq N}\: F'(u_{lo, \leq N}) + \nabla_{2,3} u_{lo, \geq N} u_{lo, \geq N} F''(u_{lo, \leq N})\\
 & \hspace{4mm}+   \nabla_{2,3} u_{lo, \leq N}u_{lo, \geq N}^2 \iint F'''(u_{lo, \leq N} + \theta_1 \theta_2 u_{lo, \geq N})
 \\
 & := 2.I + 2.II + 2.III.
\end{align}
We omit the estimates for the first two terms since they follow as above, and we focus on 
\[
 \nabla_{2,3} u_{lo, \leq N}u_{lo, \geq N}^2 \iint_0^1 F'''(u_{lo, \leq N} + \theta_1 \theta_2 u_{lo, \geq N}) =:  \nabla_{2,3} u_{lo, \leq N}u_{lo, \geq N}^2 F_3.
 \]
Here, we will need to introduce some new exponent pairs compared to the proof of the long-time Strichartz estimates. We divide this expression into two parts: 
\begin{align}
\| &|\nabla|^{- \frac{p-3}{2(p-1)} +s_p} P_{N}(  \nabla_{2,3} u_{lo, \leq N}u_{lo, \geq N}^2 F) \|_{L_t^{\frac{p-1}{p-2}}L_x^{1} }\\
& \lesssim N^{ - \frac{p-3}{2(p-1)} + s_p}\| \nabla_{2,3} u_{lo, \leq N}u_{lo, \geq N}^{p-1}  \|_{L_t^{\frac{p-1}{p-2}}L_x^{1} } \\
& \qquad +  N^{ - \frac{p-3}{2(p-1)} + s_p}\| \nabla_{2,3} u_{lo, \leq N} u_{lo, \geq N}^{2} u_{lo, \leq N}^{p-3}  \|_{L_t^{\frac{p-1}{p-2}}L_x^{1} }.
\end{align}
Note that $(\tfrac{p-1}{p-2},1)$ is dual wave admissible for $p\geq 3$.  

For the first term, we have a bound of 
\begin{align}
&N^{ - \frac{p-3}{2(p-1)} - s_p} \||\nabla|^{s_p} u_{> N } \|_{L^\infty_t L^2_x}^2 \|u_{\leq N} \|_{L^{p-1}_t L^\infty_x}^{p-3 } \sum_{M' \leq M}  M'  \| P_{M'} u _{\leq N } \|_{L^{p-1}_t L^\infty_x} \\
&\lesssim \eta_0^{p-1} \sum_{M' \leq M} \sum_{N' \leq N}  M' \left( \frac{N'}{N}\right)^{\frac{1}{p-1}} (N')^{-\frac{1}{p-1}} \| P_{N',M'} u \|_{L^{p-1}_t L^\infty_x}.
\end{align}

For the second term we estimate
\begin{align}
&N^{ - \frac{p-3}{2(p-1)} + s_p}\| \nabla_{2,3} u_{lo, \leq N}u_{lo, \geq N}^{p-1}  \|_{L_t^{\frac{p-1}{p-2}}L_x^{1} } \\
& \lesssim N^{ - \frac{p-3}{2(p-1)} + s_p} \|u_{\geq N} \|_{L^{2(p-1)}_{t,x}}^{p-3} \|u_{\geq N} \|^2_{ L^{4}_t L^{\frac{4(p-1)}{p+1}}_x }\\
& \hspace{12mm} \times \sum_{M' \leq M} \sum_{N' \leq N} M'  \left( \frac{N'}{N} \right)^{\frac{2}{p-1}} (N')^{-\frac{2}{p-1}} \| P_{N', M'}u \|_{L^{\infty}_t L^{\infty}_x} .
\end{align}
Now we note that the pair
\[
\left(4, \frac{4(p-1)}{p+1} \right)
\]
is wave admissible at regularity
\[
\frac{3}{2} - \frac{1}{4} - \frac{3p + 3}{4(p-1)} = s_p + \frac{2}{p-1} - \frac{1}{4} - \frac{3p + 3}{4(p-1)}.
\]
Noting that 
\[
 \frac{3p + 3}{4(p-1)} =  \frac{3(p + 1)}{4(p-1)} >  \frac{3}{(p-1)},
\]
we see that this is number strictly less that $s_p$. Thus we obtain a bound of
\[
\eta_0^{p-1} \sum_{M' \leq M} \sum_{N' \leq N} M'  \left( \frac{N'}{N} \right)^{\frac{2}{p-1}} (N')^{-\frac{2}{p-1}} \| P_{N', M'}u \|_{L^{\infty}_t L^{\infty}_x}.
\]

Arguing as in the estimates for Term $A$, we may determine the restrictions on $\sigma$. First, we need to assume that $\sigma < 1$ so that we can perform the summation in $M$, and we further require that $\sigma$ be bounded above by the power appearing on the $N' / N$ factor when $N' \leq N$ and the $N / N'$ factor when $N \leq N'$. Examining the exponents in the definition of $\mathcal{S}_{N, M}$ this amounts to requiring $\sigma$ smaller than the smallest (in absolute values) exponent in that expression, and hence we may assume the most restrictive of these will be taking $\sigma < \ell  /2$ in Term 1.III. 

Provided this is the case, we obtain
\begin{align}
\sum_{N'}\sum_{M' \leq M} \left( \frac{M'}{M} \right)^\sigma \min\biggl\{\frac{N}{N'},\frac{N'}{N}\biggr\}^{\sigma} \|P_{N', \geq M} F(u) \|_{N(I)}
&\lesssim \eta_0^{p-1}\, \beta_{N, M},
\end{align}
and since, by Strichartz estimates
\[
\|B\|_{\dot H^{s_p}_x} + \|B'\|_{\dot H^{s_p}_x} \lesssim \|P_{N', \geq M} F \|_{N(I)}.
\]
we have
\begin{align}
\sum_{N'}\hspace{-.5mm} \sum_{M' \leq M} \hspace{-1mm} \left( \frac{M'}{M} \right)^\sigma\hspace{-1mm} \min\biggl\{\frac{N}{N'},\frac{N'}{N}\biggr\}^{\sigma}\hspace{-1.5mm} (\|B\|_{\dot H^{s_p}_x} + \|B'\|_{\dot H^{s_p}_x})
&\lesssim  \eta_0^{p-1}\, \beta_{N, M},
\end{align}
as well as the estimate
\begin{align}
\gamma_{N, M}(N^{1-\eps}) + \beta_{N, M} \lesssim \gamma_{N, M}(0) + \eta_0^{p-1} \beta_{N, M},
\end{align}
as required.
\end{proof}

\subsection*{Term C}  We turn to the $\langle C,C'\rangle$ term (cf. \eqref{tw-ABC} and \eqref{eq:ABC}). 
In this section we prove the following lemma. 
\begin{lem}  \label{l:CC'} 
Let $C, C'$ be defined as in~\eqref{tw-ABC}, and let $M \geq C_0 N^{\frac{s_p}{1-\nu}}$. Then, for any $L \in \N$ we have 
\EQ{
\abs{\ang{C, C'}_{\dot H^{s_p}_x}} \lesssim_L \frac{1}{M^L} 
}
where the implicit constant above depends only on $L $. 
\end{lem} 

\begin{proof}[Proof of Lemma~\ref{l:CC'}]

We introduce the notation
\[
G(u,v)(t)=F(u(t))-F(v(t)),
\]
which we may abbreviate as $G(t)$ or even $G$.  We are faced with estimating
\[
\ang{C, C'}_{\dot H^{s_p}_x} \simeq N^{2s_p} \ang{C, C'}_{L^2_x},
\]
where
\begin{align} \label{equ:cc_prime}
&\ang{C, C'}_{L^2_x}  
\\&= \int_{-\infty}^{-N^{1 - \eps}} \hspace{-3mm}\int_{N^{1 - \eps}}^{\infty} \langle   S(-t) \widehat{P}_{N,  M} G(u,v)(t), \,  S(-\tau)\widehat{P}_{N,M}G(u,v)(\tau) \rangle_{L_x^2}\, \ud t\, d\tau \\
 &= \int_{-\infty}^{-N^{1 - \eps}} \hspace{-3mm}\int_{N^{1 - \eps}}^{\infty} \langle     G(u,v)(t), \, S(t-\tau)\widehat{P}_{N,M}^2\frac{1}{\abs{\na}}G(u,v)(\tau) \rangle_{L_x^2}\, \ud t\, d\tau.
\end{align}
Since $M \geq C_0 N^{\frac{s_p}{1-\nu}}$, it suffices to show that 
\[
|\ang{C, C'}_{L^2_x} |\lesssim_L \frac{1}{M^L},
\]
and all inner products in this proof will be $L^2_x$ inner products.

For each fixed $t, \tau$ as above we estimate the pairing, 
\EQ{ \label{eq:ddpair} 
\langle  G(u,v)(t) ,\,  S(t- \tau)\widehat{P}_{N,  M}^2 G(u,v)(\tau) \rangle_{L_x^2},
}
Recall that by the definition of $\vec v(\tau)$, $G(u, v)(\tau)$ is supported in the region 
\EQ{ \label{eq:calG} 
\hspace{.2in} \calG_{\pm}(\tau):= \{ x \mid \abs{x - x(\pm N^{1-\eps})} \le R(\eta_0) +  \abs{\tau}- N^{1-\eps} ,}
for all  $\pm \tau \ge N^{1-\eps}\}$.  This points to an immediate problem in any naive implementation of the double Duhamel trick by way of Huygens principle as performed in previous sections. Namely, the support of the $S(t-\tau)$ evolution of $G(u, v)(\tau)$ intersects with the support of $G(u, v)(t)$ in the ``wave zone,'' i.e., near the boundary of the light cone where the kernel of $S(t -\tau)$ only yields $\ang{t-\tau}^{-1}$ decay, which is not sufficient for integration in time. However, we are saved here by a gain in \emph{angular separation} in the wave zone guaranteed by our directional frequency localization $\hat P_{N, M}$. Indeed, application of $\hat P_{N, M}$ restricts to frequencies $\xi = (\xi_1, \xi_{2, 3})$ with 
\EQ{
\frac{\abs{\xi_{2, 3}}}{ \abs{\xi}} \simeq \frac{M}{N}
}
whereas for any $x = (x_1, x_{2, 3}) \in \calG(t) \cap \{(t, x) \mid \abs{x} \ge t - R(\eta_0)\}$ we claim that 
\EQ{
\frac{\abs{x_{2, 3}}}{\abs{x}} \ll \frac{M}{N}
}
for all $M \ge N^{\frac{s_p}{1-\nu}}$. We establish this fact in Lemma~\ref{l:angle} below.  

We introduce some additional notation. Let $R(\cdot)$ be the compactness modulus function. For given $t \in \R$ let
\begin{equation}
\label{eq:calC}
\begin{split}
\calC_{\ext}(t) := \{ x \mid \abs{x} \ge \abs{t} - R(\eta_0), \\
 \calC_{\inte}(t):= \{ x \mid \abs{x} \le \abs{t} - R(\eta_0)\}
\end{split}
\end{equation}

 We decompose $\ang{C, C'}$ as follows. First, we write
\begin{align}
 G(u,v)(t) &=   G(u,v)(t)1_{\calC_{\ext}(t)} +G(u,v)(t)1_{\calC_{\inte(t)} }.
\end{align}
Using this decomposition in \eqref{equ:cc_prime} leads to four terms:
\begin{align}\label{c_outout}
\iint\bigl\langle S(t - \tau) \frac{1}{\abs{\na}}\widehat{P}_{N,  M}   1_{\calC_{\ext}}(\tau) G(\tau)  ,  \, 1_{\calC_{\ext}}(t) G(t) \bigr \rangle \ud t \ud\tau,
\end{align}
\begin{align}\label{c_outin}
\iint \bigl\langle S(t - \tau)  \frac{1}{\abs{\na}}\widehat{P}_{N,  M}   1_{\calC_{\ext}}(\tau) G(\tau) ,  \, 1_{\calC_{\inte}}(t) G(t) \bigr \rangle \ud t \ud\tau,
\end{align}
\begin{align}\label{c_inout}
\iint\bigl\langle S(t - \tau) \frac{1}{\abs{\na}} \widehat{P}_{N,  M}   1_{\calC_{\inte}}(\tau) G(\tau) , \,  1_{\calC_{\ext}}(t) G(t) \bigr \rangle \ud t \ud\tau,
\end{align}
\begin{align}\label{c_inin}
\iint \bigl\langle S(t - \tau)  \frac{1}{\abs{\na}}\widehat{P}_{N,  M}   1_{\calC_{\inte}}(\tau) G(\tau)  ,  \, 1_{\calC_{\inte}}(t) G(t) \bigr \rangle \ud t \ud\tau,
\end{align}
where the integrals are over $[-\infty,-N^{1-\eps}]\times[N^{1-\eps},\infty]$.   We will refer to these terms are $C_{\ext - \ext}$, $C_{\ext - \inte}$, $C_{\inte - \ext}$ and $C_{\inte -  \inte}$ respectively, and we will handle these terms separately below.
Were it not for the frequency localization $\hat P_{N, M}$ all but the first term above would vanish using the support properties of $G(u, v)$, together with the particular pairing of the cutoffs $1_{C_{\inte}}$ and $1_{C_{\ext}}$, and the sharp Huygens principle.  On the other hand, whereas in previous scenarios (e.g. the subluminal soliton) the first term would vanish, in the present setting there truly is an interaction between these two terms.  This is the origin of the essential technical difficulty faced in the present scenario, and indeed we will find that the first term~\eqref{c_outout} requires the most careful analysis. The crucial observation is that in this setting we can rely on angular separation to exhibit decay.

\subsection*{The term \texorpdfstring{$C_{\ext - \ext}$}{C}:} We will rely crucially on the following two lemmas, which together make precise the gain in decay from angular separation. 

\begin{lem}[Angular separation in the wave zone] \label{l:angle} For any $c>0$ there exists $N_0 = N_0(c)>0$ with the following property. Fix $\nu \in (0, 1)$ and let $\eps>0$ be any number with $\eps< \frac{2s_p}{1-\nu} -1$.  Let $(t, x)$ satisfy 
\EQ{
\abs{t} \ge N^{1-\eps}, \quad  x = ( x_1, x_{2, 3}) \in \calG(t) \cap \calC_{\ext}(t)
}
where $\calG(t)$ are defined in~\eqref{eq:calG},~\eqref{eq:calC}. Then, 
\EQ{ \label{eq:angle} 
\frac{\abs{x_{2, 3}}}{\abs{x}} \lesssim \frac{1}{N^{\frac{1}{2}- \frac{\eps}{2}}} \le  c \frac{M}{N}
}
for all $ N \ge N_0$ and  $M \ge N^{\frac{s_p}{1-\nu}}$. 

\end{lem} 

See Figure~\ref{f:angle} for a depiction of Lemma~\ref{l:angle}. 

\begin{figure}[h] 
  \centering
  \includegraphics[width=14cm]{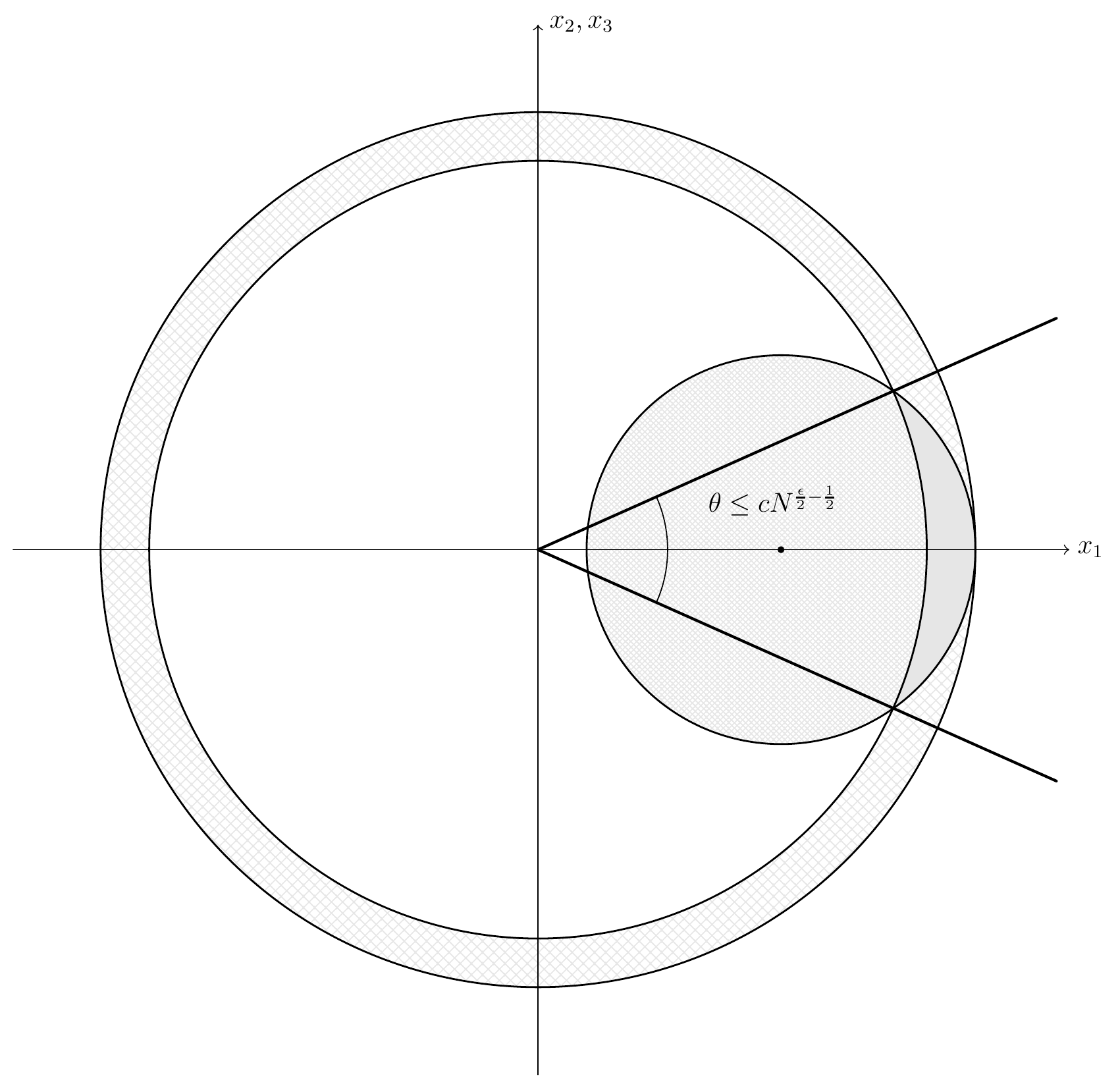}
  \caption{The dark gray region above represents the region $\calG_+(t) \cap \calC_{\ext}$ in space at fixed time $t > N^{1-\eps}$.}  \label{f:angle} 
\end{figure}

Next, we show that if we restrict to those $x \in \R^3$ satisfying~\eqref{eq:angle} then we get strong pointwise decay for the kernel of the operator $S(t)\frac{1}{ \abs{\na}}\widehat{P}_{N, M}^2$. 

To state the result, we define 
\EQ{ \label{eq:calS}
\calS_{N} := \Big\{ x \in \R^3 \mid \frac{\abs{x_{2, 3}}}{\abs{x}} \lesssim \frac{1}{N^{\frac{1}{2}- \frac{\eps}{2}}}  \Big\}.
}

\begin{lem}[Kernel estimates via angular separation] \label{lem:ang}
Let $K_{N, M}(t,x)$ denote the kernel of the operator $S(t)\frac{1}{ \abs{\na}}\widehat{P}_{N, M}^2$. Let $N \ge N_0$ where $N_0$ is as in the hypothesis of Lemma~\ref{l:angle}. Then for any $L$,
\begin{align}\label{equ:ang}
 | 1_{\calS_N}(x)K_{N, M}(t, x)| \lesssim_{L} N \frac{N^{L }}{M^L} \frac{1}{ \ang{ M |x|}^{L}}, \quad   \forall t \ge N^{1-\eps}, \quad  
 \end{align}
where $\calS_N$ is the set defined in~\eqref{eq:calS} and where we have used the notation $\ang{z}:= (1 + \abs{z}^2)^{\frac{1}{2}}$ above. 
\end{lem}

\begin{proof}[Proof of Lemma~\ref{l:angle}]
We assume that $t \ge 0$. Since we are assuming $0< \eps< \frac{2s_p}{1-\nu} -1$ and that $M \ge N^{\frac{s_p}{1-\nu}}$ it suffices to show the first inequality in~\eqref{eq:angle}, i.e., that 
\EQ{ \label{eq:angleest} 
\frac{\abs{x_{2, 3}}}{\abs{x}} \lesssim  \frac{1}{ N^{\frac{1}{2} -\frac{\eps}{2}}} 
}
for all $x \in \calG(t) \cap \calC_{\ext}(t)$ for some uniform constant. 

First, we claim that ~\eqref{eq:angleest} holds at time $t = N^{1-\eps}$. Suppose 
\EQ{
x \in \calG_+(N^{1-\eps}) \cap \calC_{\ext}(N^{1-\eps}).
}
By the definition of traveling wave (i.e. \eqref{eq:xhans1} and \eqref{eq:xhans2}) we have 
\EQ{
\abs{x} \simeq N^{1-\eps}, \quad \abs{x_{2, 3}} \lesssim N^{\frac{1}{2} - \frac{\eps}{2}}
}
and thus, 
\EQ{
\frac{\abs{x_{2, 3}}}{\abs{x}}  \lesssim N^{\frac{\eps}{2} - \frac{1}{2}}
}
as desired. 

Now suppose $t > N^{1- \eps}$. We introduce some notation. Let $\te_{x(N^{1-\eps})}$ denote the angle between the unit vector $\vec e_1$ (the unit vector in the positive $x_1$-direction) and the vector $x(N^{1-\eps})$, where we recall that $x(t)$ denotes the spatial center of $\vec u$. Above, we have just shown that 
\EQ{
\abs{\sin(\te_{x(N^{1-\eps})})} \simeq \abs{\te_{x(N^{1-\eps})}} \le A_1 N^{\frac{\eps}{2} - \frac{1}{2}}
}
for some uniform constant $A_1>0$. To finish the proof it will suffice to show that for any $x \in \calG(t) \cap \calC_{\ext}(t)$, the angle $\te_{(x, x(N^{1-\eps}))}$ formed between the vectors $x$ and $x(N^{1-\eps})$ satisfies 
\EQ{
\abs{\te_{(x, x(N^{1-\eps}))}} \le A_2 N^{\frac{\eps}{2} - \frac{1}{2}}
}
for some other uniform constant $A_2 >0$, as then the sine of the total angle between $x$ and the $x_1$-axis, i.e., $\frac{\abs{x_{2, 3}}}{\abs{x}}$ would satisfy~\eqref{eq:angleest}. To get a hold of $\te_{(x, x(N^{1-\eps}))}$ we square both sides of the inequality defining the set $\calG_+(t)$. For $x \in \calG(t)$ we have 
\EQ{
\abs{x}^2 - 2  x \cdot x(N^{1-\eps}) + \abs{x(N)^{1-\eps}}^2 \le \left( R(\eta_0) + t - N^{1-\eps}\right)^2.
}
Using that $x \cdot x(N^{1-\eps}) = \abs{x} \abs{x(N^{1-\eps})} \cos \te_{(x, x(N^{1-\eps}))}$ the above yields the inequality, 
\EQ{
-2 \abs{x} \abs{x(N^{1-\eps})} \cos\te_{(x, x(N^{1-\eps}))}\le \left( R(\eta_0) + t - N^{1-\eps}\right)^2  - \abs{x}^2 - \abs{x(N^{1-\eps})}^2
}
Bootstrapping, we may assume that $\te_{(x, x(N^{1-\eps}))}$ is small enough to use the estimate, 
\EQ{
 \cos \te_{(x, x(N^{1-\eps}))} \le 1 - \frac{\te_{(x, x(N^{1-\eps}))}^2}{4}. 
}
Pugging the above in we arrive at the inequality 
\EQ{
\frac{\te_{(x, x(N^{1-\eps}))}^2}{2} &\le 2  + \frac{1}{\abs{x}\abs{x(N^{1-\eps})}}\left( \left( R(\eta_0) + t - N^{1-\eps}\right)^2  - \abs{x}^2 - \abs{x(N^{1-\eps})}^2  \right) \\
& \le  \frac{ 2 \abs{x}\abs{x(N^{1-\eps})} + \left( R(\eta_0) + t - N^{1-\eps}\right)^2  - \abs{x}^2 - \abs{x(N)^{1-\eps}}^2}{\abs{x}\abs{x(N^{1-\eps})}}.
}
The requirement that $x \in C_{\ext}(t)$, finite speed of propagation, and~\eqref{eq:xhans1} imply that we have 
\EQ{
t - R(\eta_0) \le \abs{x} \le t + R(\eta_0) \mand  N^{1-\eps} - R(\eta_0) \le \abs{x(N^{1-\eps})} \le N^{1-\eps} + R(\eta_0).
}
Plugging the above into the previous line gives, 
\begin{align}
\frac{\te_{(x, x(N^{1-\eps}))}^2}{2} &\le \frac{ 2 (t + R(\eta_0))(N^{1-\eps} + R(\eta_0))+ \left( R(\eta_0) + t - N^{1-\eps}\right)^2}{(t - R(\eta_0)) ( N^{1-\eps} - R(\eta_0) )} \\
& \hspace{34mm} - \frac{ (t - R(\eta_0))^2 - (N^{1-\eps} - R(\eta_0))^2}{(t - R(\eta_0)) ( N^{1-\eps} - R(\eta_0) )}\\
&  =  \frac{ 6 t R(\eta_0) + 2 N^{1-\eps} R(\eta_0) + R(\eta_0)^2}{(t - R(\eta_0)) ( N^{1-\eps} - R(\eta_0) )}  \\
& \lesssim  \frac{1}{N^{1-\eps}} + \frac{1}{t}. 
\end{align}
Taking the square root and noting that $t \ge N^{1-\eps}$ we arrive at 
\EQ{
\abs{\te_{(x, x(N^{1-\eps}))}} \lesssim  \frac{1}{N^{\frac{1}{2} - \frac{\eps}{2}}}
}
as desired. This completes the proof. 
\end{proof}

Next, we prove Lemma~\ref{lem:ang}. 

\begin{proof}[Proof of Lemma~\ref{lem:ang}]

The kernel $K_{N, M}$ of the operator $S(t)\frac{1}{ \abs{\na}}\widehat{P}_{N, M}^2$ is given by 
\EQ{
K_{N, M}(t, x) :=  \int e^{i x \cdot \xi}  \abs{\xi}^{-2} e^{ it |\xi|}  \phi^2\bigl(\tfrac{|\xi|}{N}\bigr)  \phi^2 \bigl(\tfrac{|(\xi_2, \xi_3)|}{M}\bigr)  d\xi  \\
}
where $\phi \in C^{\infty}_{0}(\R)$ is satisfies $ \phi(r) = 1$ if $1 \le r  \le  2$ and  $ \supp \phi \in (\frac{1}{4}, 4)$. 
Now, 
recall that we are restricting to only those $x \in \calS_N$, as defined in~\eqref{eq:calS}. We express any 
such $x$ in spherical coordinates 
\EQ{
x = \abs{x}(  \cos \te_x, \sin \te_x  \cos \om, \sin \te_x \sin \om)
}
where $\te_x$ denotes the angle formed by $x$ and the unit vector in the $e_1$-direction. And recall that any $x \in \calS_N$ satisfies  
\EQ{ \label{eq:xa}
 \frac{\abs{x_{2, 3}}}{\abs{x}} =  \sin \te_x  \simeq \abs{\te_x} \lesssim \frac{1}{N^{1-\eps}}. 
}
Similarly, we change to the spherical variables 
\EQ{
\xi  = \abs{\xi}( \cos \te_{\xi}, \sin \te_\xi \cos \al, \sin \te_\xi \sin \al)
}
in the integral defining $K_{N, M}$ and note that because of the frequency localization $\hat P_{N, M}$ we have 
\EQ{\label{eq:xia}
\frac{\abs{\xi_{2, 3}}}{\abs{\xi}}  = \sin \te_\xi \simeq \frac{M}{N}. 
}
This yields, 
\begin{align*}
K_{N, M}(t, x) =  \int_0^{2\pi}\int_0^\pi\int_{N/4}^{4N} & e^{i \abs{x}\abs{\xi}f( \te_x, \te_\xi, \om, \al)} \abs{\xi}^{-2} e^{ it |\xi|}  \phi^2\bigl(\tfrac{|\xi|}{N}\bigr)   \\
& \quad\times \phi^2 \bigl(\tfrac{ \abs{\xi}\sin \te_\xi }{M}\bigr)   \, \abs{\xi}^2  \sin \te_\xi \ud \abs{\xi} \ud \te_\xi \ud \al,
\end{align*}
where the angular phase function $f( \te_x, \te_\xi, \om, \al)$ is given by 
\EQ{
f( \te_x, \te_\xi, \om, \al) = \cos \te_x \cos \te_\xi + \sin \te_x \sin \te_\xi( \cos \om \cos \al + \sin \om \sin \al).
}
The idea is that the angular separation between $x$ and $\xi$ given by~\eqref{eq:xa} and~\eqref{eq:xia} allows us to integrate by parts in $\te_\xi$. Indeed,  using~\eqref{eq:xa} and~\eqref{eq:xia} we have the lower bound
\EQ{
\biggl| \frac{\ud}{\ud \te_\xi} &[\abs{x} \abs{\xi}   f( \te_x, \te_\xi, \om, \al) ] \biggr|\\
 &= \abs{x}{\abs \xi}  \Bigl| -\cos \te_x  \sin \te_\xi + \sin \te_x \cos \te_\xi( \cos \om \cos \al + \sin \om \sin \al ) \Bigr| \\
& \ge \abs{x}{\abs \xi} \Big( \frac{M}{N}  -    O( \frac{1}{N^{\frac{1}{2}-\frac{\eps}{2}}}) \Big) \\
& \gtrsim \abs{x} M.
}
Moreover, note that for any $L \in \N$ and $M\lesssim N$, 
\EQ{
\abs{\frac{\ud^L}{\ud \te_\xi^L} \big(\phi^2 \bigl(\tfrac{ \abs{\xi}\sin \te_\xi }{M}\bigr)  \sin \te_\xi \Big)} \lesssim \frac{N^{L}}{M^{L}}
}
Thus, integration by parts $L$-times in $\te_\xi$ yields the estimate 
\EQ{
\abs{K_{N, M}(t, x)} \lesssim_L N^2 \frac{N^{L}}{M^{L}} \frac{1}{ \ang{M\abs{x}}^L}, \quad \forall t \ge N^{1-\eps}, \quad   x \in \calG(t) \cap \calC_{\ext}(t)
}
as desired. \end{proof} 

We can now estimate~\eqref{c_outout}. Here will rely crucially on Lemma~\ref{l:angle} and 
Lemma~\ref{lem:ang}. First we write, 
\begin{multline} 
 \bigl\langle S(t - \tau) \frac{1}{\abs{\na}} \widehat{P}_{N,  M}   1_{\calC_{\ext}}(\tau) G(u,v)(\tau)  ,  \, 1_{\calC_{\ext}}(t) G(u,v)(t) \bigr \rangle   \\
 =  \ang{ K_{N, M}(t- \tau) \ast 1_{\calC_{\ext}}(\tau) G(u,v)(\tau)  , \, 1_{\calC_{\ext}}(t) G(u,v)(t)}.
\end{multline} 
We claim that in fact, the above can be expressed as 
\begin{multline} 
  \ang{ K_{N, M}(t- \tau) \ast 1_{\calC_{\ext}}(\tau) G(\tau)  , \, 1_{\calC_{\ext}}(t) G(t)}  \label{eq:SNttau}  \\
   = \ang{ \big(1_{\calS_N}( \cdot) 1_{ \{\abs{ \cdot} \ge\frac{1}{2} \abs{t -\tau}   \}}(\cdot) K_{N, M}(t- \tau) \big) \ast 1_{\calC_{\ext}}(\tau) G(\tau)  , \, 1_{\calC_{\ext}}(t) G(t)},
\end{multline} 
where the set $\calS_N$ is defined in~\eqref{eq:calS}. 
Indeed, note that above we have 
\EQ{ \label{eq:xyout} 
x \in \calG_+(t) \cap \calC_{\ext}(t) \mand  y \in \calG_-(\tau) \cap \calC_{\ext}(\tau)
}
where $\calG_{\pm}$ are as in~\eqref{eq:calG} and $C_{\ext}$ is as in~\eqref{eq:calC}. 
Thus, 
\EQ{
\abs{x- y} \ge  \abs{t - \tau} - 2 R(\eta_0)  \ge \frac{1}{2}  \abs{t - \tau}
}
as long as $N$ is chosen large enough. Similarly by~\eqref{eq:xyout} we have $\abs{x-y} \ge \abs{x}$ and $\abs{x-y} \ge \abs{y}$ and thus, 
\EQ{
\frac{\abs{x_{2, 3} - y_{2, 3}}}{ \abs{x - y}} \le \frac{\abs{x_{2, 3}}}{\abs{x}}  +  \frac{\abs{y_{2, 3}}}{\abs{y}}  \lesssim \frac{1}{N^{\frac{1}{2} - \frac{\eps}{2}}},
}
where in the last inequality above we used Lemma~\ref{l:angle}. This proves the equality in~\eqref{eq:SNttau}. 

Now, let $q_{p}$ denote the Sobolev embedding exponent for $\dot H^{s_p}$, i.e. $q_p=\frac{3(p-1)}{2}$. Note that $q_p\geq p$ for $p\geq 3$ and $(q_p/p)'\geq 2$ for $p>0$ (where $x'$ denotes the H\"older dual of $x$).  By H\"older's and Young's inequalities we then have 
\begin{align*} 
&\abs{\ang{ \big(1_{S_N}( \cdot)  1_{ \{\abs{ \cdot} \ge\frac{1}{2} \abs{t -\tau}   \}}(\cdot)K_{N, M}(t- \tau) \big) \ast 1_{\calC_{\ext}}(\tau) G(\tau)  , \, 1_{\calC_{\ext}}(t) G(t)}}  \\
&\quad \le \| 1_{S_N}( \cdot)  1_{ \{\abs{ \cdot} \ge\frac{1}{2} \abs{t -\tau}   \}}(\cdot) K_{N, M}(t- \tau) \|_{L_x^{\left(\frac{q_{p}}{p} \right)'/2}} \\
& \quad\times \| G(u, v)(t) \|_{L_x^{\frac{q_{p}}{p}}}\| G(u, v)(\tau) \|_{L_x^{\frac{q_{p}}{p}}}.
\end{align*} 
Using~\eqref{equ:ang} we see that 
\begin{align}  \label{eq:MNL} 
&\| 1_{S_N}( \cdot)  1_{ \{\abs{ \cdot} \ge\frac{1}{2} \abs{t -\tau}   \}}(\cdot) K_{N, M}(t- \tau) \|_{L_x^{\left(\frac{q_{p}}{p} \right)'/2}} \\
 &\lesssim \frac{N^{L+1}}{M^{2L}} \left( \int_{\abs{x} \ge \frac{1}{2}\abs{t-\tau}}  \frac{1}{ \abs{x}^{2L\left(\frac{q_{p}}{p} \right)'} } \, \ud x \right)^{\frac{2}{\left(\frac{q_{p}}{p} \right)'}} \\
& \lesssim \frac{N^{L+1}}{M^{2L}} \frac{1}{  \abs{t-\tau}^{L-1}}
\end{align}
Since $G(u, v) = F(u) - F(v)$ we have 
\EQ{
 \| G(u, v)(t) \|_{L_x^{\frac{q_p}{p}}} \lesssim \| u(t) \|_{L_x^{q_p}}^p + \| v(t) \|_{L_x^{q_p}}^p \lesssim \|u(t) \|_{\dot H^{s_p}_x}^p + \|v(t) \|_{ \dot H^{s_p}}^p.
}
Putting this all together we arrive at the estimate 
\begin{align} 
&\abs{ \int_{-\infty}^{-N^{1 - \eps}} \hspace{-3mm} \int_{N^{1 - \eps}}^{\infty} \bigl\langle S(t - \tau) \widehat{P}_{N,  M}   1_{\calC_{\ext}}(\tau) G(u,v)(\tau)  ,  \, 1_{\calC_{\ext}}(t) G(u,v)(t) \bigr \rangle  \ud t \, \ud \tau } \\
 &\lesssim_L \int_{-\infty}^{-N^{1 - \eps}} \hspace{-3mm} \int_{N^{1 - \eps}}^{\infty}  \frac{N^{L+1}}{M^{2L}} \frac{1}{  \abs{t-\tau}^{L-1}}  \left(\|u(t) \|_{\dot H^{s_p}_x}^p + \|v(t) \|_{ \dot H^{s_p}}^p \right)\\
 & \hspace{54mm} \times \left(\|u(\tau) \|_{\dot H^{s_p}_x}^p + \|v(\tau) \|_{ \dot H^{s_p}}^p \right)  \ud t \, \ud \tau  \\
& \lesssim_L N^{L + 1 + (1-\eps)(2- L)} M^{-2L}  \left(\|u \|_{L^\infty_t \dot H^{s_p}}^{2p} + \|v \|_{L^\infty_t \dot H^{s_p}}^{2p} \right)  \\
& \lesssim_L M^{-L},
\end{align} 
where to obtain the last line we ensure that $\eps>0$ is small enough so that when $M \ge N^{\frac{s_p}{1-\nu}}$ we also have $M^L \ge N^{4+ \eps L}$. We have proved that 
\EQ{
\eqref{c_outout} \lesssim_L M^{-L},
}
as desired.  This completes the treatment of the $C_{\ext - \ext}$ term.

\subsection*{The term \texorpdfstring{$C_{\inte - \inte}$}{C}:}  

Here we will use a combination of arguments based on sharp Huygens principle and the techniques developed to deal with the previous term $C_{\ext-\ext}$. 

First we record an estimate for the kernel of the modified frequency projection. 

\begin{lem} 
Let $p_{N, M}^2$ denote the kernel of the operator $\widehat{P}_{N, M}^2$. Then, 
\EQ{ \label{eq:pnm} 
\abs{p_{N, M}^2 (x)} \lesssim_L \frac{N^3}{\ang{N \abs{x}}^L} + \frac{N^3}{\ang{M \abs{x}}^L} 
}
\end{lem} 

Next, consider the following decomposition of the forward cone centered at $(t, x) = (N^{1-\eps}, x(N^{1-\eps}))$ of width $R(\eta_0)$, i.e., the set 
\EQ{
\calG_+ := \bigcup_{t \ge N^{1-\eps}}\calG_+(t)
} 
where $\calG(t)$ is defined as in~\eqref{eq:calG}.  This decomposition is depicted in Figure~\ref{f:CG}.

We write 
\EQ{
\calG_+ = \bigcup_{j \ge 1} \calC_{+, j} \cup \bigcup_{j \ge 0} \calG_{+, j}
}
We define the $\calC_{+, j}, G_{+, j}$ as follows. First, set 
\EQ{
\ti\calC_{+, 1}&:=   \{(t, x)  \mid  \abs{x - x(2N^{1-\eps})} \ge R(\eta_0) + t - 2N^{1-\eps} , \\
& \hspace{74mm} t \ge 2N^{1-\eps} \} \cap \calG_+ 
}
and for $j \ge 1$, 
\EQ{
\ti \calC_{+, j}&:=   \bigg\{\{ (t, x)  \mid  \abs{x - x(2^jN^{1-\eps})} \ge R(\eta_0) + t - 2^jN^{1-\eps}, \\
& \hspace{64mm} t \ge 2^jN^{1-\eps} \} \cap \calG_+ \bigg\} \setminus \calC_{+, j-1}.
}
For $j \ge 0$, define sets $\ti\calG_{+, j}$ to be the regions, 
\EQ{
\ti \calG_{+, j} &=   \{(t, x) \mid  \abs{x - x(2^jN^{1-\eps})} \le R(\eta_0) + t - 2^jN^{1-\eps}, \\
& \hspace{64mm} 2^j N^{1-\eps} \le t \le 2^{j+1}N^{1-\eps} \} \cap \calG_+.
}
Then we define 
\EQ{
&\calC_{+, j} := \ti \calC_{+, j}  \cap \{ (t, x) \mid \abs{x}  \le t - 2^{j}N^{1-\eps}, \quad  t \ge 2^{j}N^{1-\eps}\} \\
&\calG_{+, j}:= \ti \calG_{+, j} \cup [\ti\calC_{+, j+1} \setminus \calC_{+, j+1}].
}
The regions $\calC_{+,j}$ and $\calG_{+, j}$ are depicted in Figure~\ref{f:CG}. 
\begin{figure}[h]
  \centering
  \includegraphics[width=14cm]{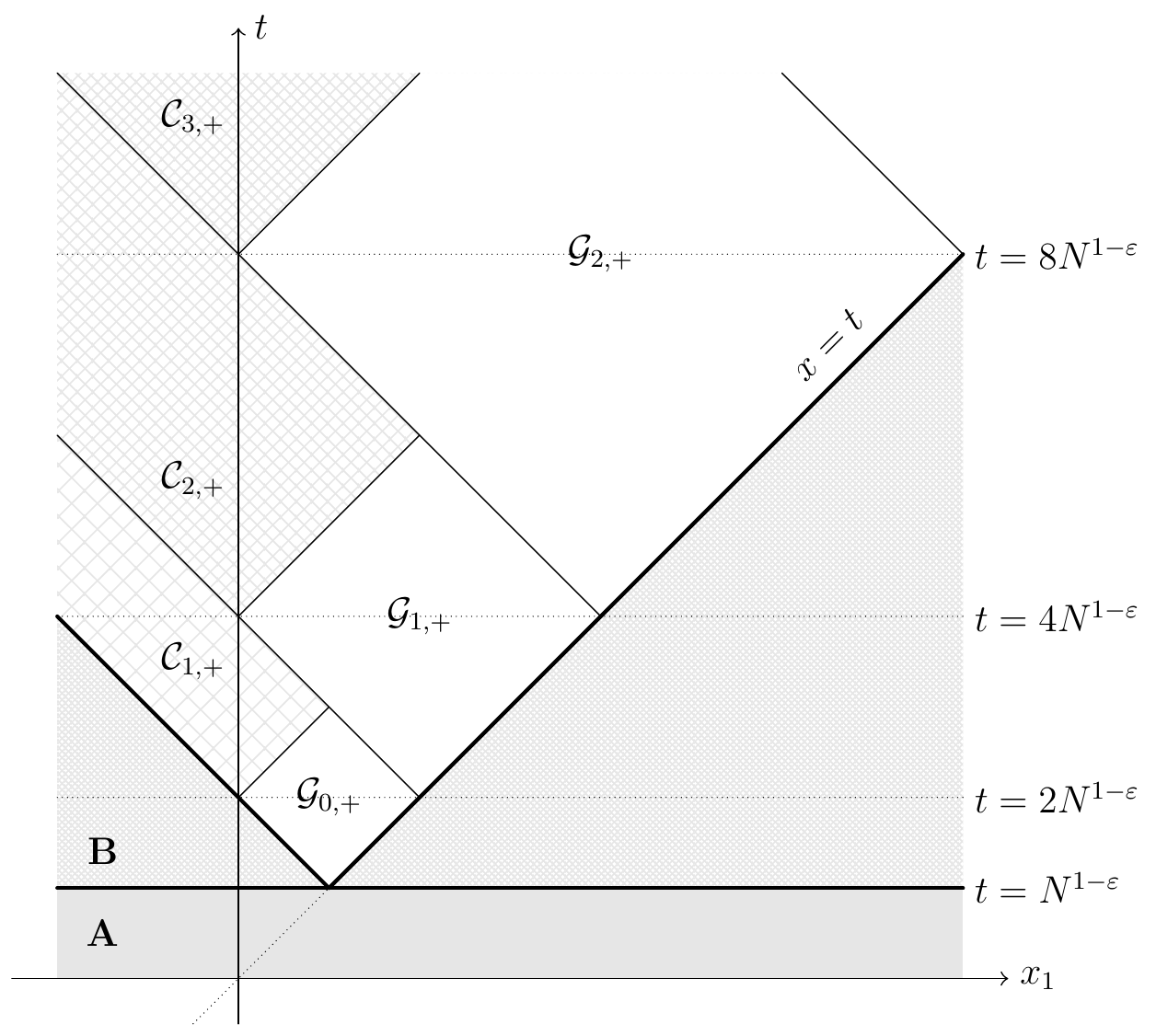}
  \caption{A depiction of the first few regions $\calC_{j, +}$ and $\calG_{+, j}$ within the region $C$.} \label{f:CG} 
\end{figure}

Now, split the integrand of~\eqref{c_inin} in the four pieces, 
\begin{align}
\eqref{c_inin} = \int_{-\infty}^{-N^{1 - \eps}} \hspace{-3mm} \int_{N^{1 - \eps}}^{\infty} (I + II + III + IV) \ud t\, \ud \tau,
\end{align} 
where 
\begin{align}  
\hspace{10mm} I &= \sum_{j,k}  \bigl\langle \widehat{P}_{N,  M}^2 \frac{1}{\abs{\na}}S(t - \tau)  [1_{\calC_{-, j}}   1_{\calC_{\inte}} G](\tau)  ,  [1_{\calC_{+, k}}1_{\calC_{\inte}} G](t) \bigr \rangle, 
\label{eq:CC} \\
 II&=   \sum_{j,k}  \bigl\langle \widehat{P}_{N,  M}^2 \frac{1}{\abs{\na}}S(t - \tau)  [1_{\calC_{-, j}}   1_{\calC_{\inte}} G](\tau)  ,   [1_{\calG_{+, k}}1_{\calC_{\inte}} G](t) \bigr \rangle, \label{eq:GC}  \\
 III &=\sum_{j,k}  \bigl\langle \widehat{P}_{N,  M}^2 \frac{1}{\abs{\na}}S(t - \tau)  [1_{\calG_{-, j}}   1_{\calC_{\inte}} G](\tau)  ,   [1_{\calC_{+, k}}1_{\calC_{\inte}}G](t) \bigr \rangle,  \label{eq:CG} \\
 IV &= \sum_{j,k}  \bigl\langle \widehat{P}_{N,  M}^2 \frac{1}{\abs{\na}}S(t - \tau)  [1_{\calG_{-, j}}   1_{\calC_{\inte}} G](\tau)  ,   [1_{\calG_{+, k}}1_{\calC_{\inte}} G](t) \bigr \rangle.       \label{eq:GG}
\end{align}

First we estimate the term~\eqref{eq:CC} above. The key points are the following. First, 
%
by the support properties of $1_{\calC_{+, k}}1_{\calC_{\inte}}(\tau, y)$,  $1_{\calC_{-, j}}   1_{\calC_{\inte}}(t, y) $ and the sharp Huygens principle, we must  have 
\begin{align} \label{eq:gjgk}
&\abs{x - y} \gtrsim (2^j +2^k) N^{1-\eps} \\
 & \forall\,  x \in \supp (1_{\calC_{+, k}}1_{\calC_{\inte}} G(u,v))(t), \, \,  y \in \supp [S(t- \tau) 1_{\calC_{-, j}}   1_{\calC_{\inte}} G(u,v)](\tau) .
\end{align} 
Second, by the definitions of the spacetime cutoffs $1_{\calC_{-, j}}$, and $1_{\calC_{+, k}}$ the functions $1_{\calC_{-, j}} u(\tau)$ and $1_{\calC_{+, k}} u(t)$ are restricted  to the exterior \emph{small data} regime and we thus have 
\EQ{ \label{eq:cjck}
\|1_{\calC_{-, j}}   1_{\calC_{\inte}} G(u,v) \|_{\NN( (-\infty, -2^j N^{1-\eps}])}  \lesssim \| \vec u \|_{L^\infty_t \dot \HH^{s_p}}  \lesssim 1\\
\|1_{\calC_{+, k}}   1_{\calC_{\inte}} G(u,v) \|_{\NN( [ 2^j N^{1-\eps}, \infty))}  \lesssim \| \vec u \|_{L^\infty_t \dot \HH^{s_p}}  \lesssim 1,
}
where $\NN$ denote suitable dual spaces. 

We argue as follows.  For any $q\ge2$, and up to fattening the projection $\hat P_{N, M}$, we have 
\begin{align*}
&\abs{\bigl\langle \widehat{P}_{N,  M}^2\frac{1}{\abs{\na}} S(t - \tau)  [1_{\calC_{-, j}}   1_{\calC_{\inte}} G(u,v)](\tau)  ,  \, [1_{\calC_{+, k}}1_{\calC_{\inte}} G(u,v)](t) \bigr \rangle}  \\ 
&\lesssim  \| 1_{ \{\abs{ \cdot} \gtrsim (2^j +2^k) N^{1-\eps} \} }p_{N, M} \|_{L^1} \\
& \quad\times \| P_{N} \abs{\na}^{-1-s_p + \frac{2}{q}}S(t - \tau)  [1_{\calC_{-, j}}   1_{\calC_{\inte}} G(u,v)](\tau) \|_{L^{q}_x}  \\
  &\quad \times \| \abs{\na}^{s_p - \frac{2}{q}} [1_{\calC_{+, k}}1_{\calC_{\inte}} G(u,v)](t) \|_{L^{q'}} 
\end{align*} 
We estimate last line above as follows. Note that by~\eqref{eq:pnm}  and~\eqref{eq:gjgk} (and the lower bound on $M$), we have 
\EQ{
\| 1_{ \{\abs{ \cdot} \gtrsim (2^j +2^k) N^{1-\eps} \} }p_{N, M} \|_{L^1} \lesssim_L  \frac{N^3}{ [(2^j + 2^k) N^{1-\eps}]^{L-1}}.
}
By the dispersive estimate for the wave equation (and noting that $\abs{t-\tau} \ge 2 N^{1-\eps}$ we have 
\begin{align}
& \| P_{N} \abs{\na}^{-1-s_p + \frac{2}{q}}S(t - \tau)  [1_{\calC_{-, j}}   1_{\calC_{\inte}} G(u,v)](\tau) \|_{L^{q}_x}  \\ 
&\lesssim \frac{1}{\abs{t-\tau}^{1- \frac{2}{q}}} N^{1-\frac{4}{q}} \| P_N \abs{\na}^{-1-s_p + \frac{2}{q}}[1_{\calC_{-, j}}   1_{\calC_{\inte}} G(u,v)](\tau) \|_{L^{q'}_x} \\
& \lesssim \frac{1}{\abs{t-\tau}^{1- \frac{2}{q}}} N^{-2s_p} \| P_N \abs{\na}^{s_p - \frac{2}{q}}[1_{\calC_{-, j}}   1_{\calC_{\inte}} G(u,v)](\tau) \|_{L^{q'}_x}.
\end{align} 
Thus,  using the above, Bernstein's inequality, the Hardy-Littlewood Sobolev inequality, and~\eqref{eq:cjck} in the last line below we have  
\begin{align} 
&\int_{N^{1-\eps}}^\I \int_{-\I}^{-N^{1-\eps}}\eqref{eq:CC} \, \ud t \, \ud \tau \\
& \lesssim_L  \sum_{j, k \ge 1} \frac{N^3N^{-2s_p}}{ [(2^j + 2^k) N^{1-\eps}]^{L-1}}  \int_{N^{1-\eps}}^\I \int_{-\I}^{-N^{1-\eps}} \hspace{-4mm}\Big(  \frac{1}{\abs{t-\tau}^{1- \frac{2}{q}}} 
 \\
 & \quad \times \| \abs{\na}^{s_p- \frac{2}{q}} [1_{\calC_{-, j}}   1_{\calC_{\inte}} G(u,v)](\tau) \|_{L^{q'}_x} \\
 &  \, \quad \qquad \times  \,\quad \| \abs{\na}^{s_p- \frac{2}{q}} [1_{\calC_{-, k}}   1_{\calC_{\inte}} G(u,v)](t) \|_{L^{q'}_x} \Big) 
 \, \ud t \, \ud \tau \\
& \lesssim_L  \sum_{j, k \ge 1} \frac{N^3N^{-2s_p}}{ [(2^j + 2^k) N^{1-\eps}]^{L-1}}   \| \abs{\na}^{s_p- \frac{2}{q}} [1_{\calC_{+, k}}   1_{\calC_{\inte}} G(u,v)] \|_{L_t^{\frac{2q}{q+2}}L^{q'}_x} \\
& \quad \qquad \times \quad \| \abs{\na}^{s_p- \frac{2}{q}} [1_{\calC_{-, j}}   1_{\calC_{\inte}} G(u,v)] \|_{L_t^{\frac{2q}{q+2}}L^{q'}_x} \\
 & \lesssim_L \sum_{j, k \ge 1} \frac{N^3N^{-2s_p}}{ [(2^j + 2^k) N^{1-\eps}]^{L-1}}  \lesssim_L N^{-L/2}, 
\end{align} 
where in the second to last line we have fixed $q>2$ above and note that the norms above are dual sharp admissible Strichartz pairs (e.g., one can take $q=4$). 

Next, consider the term~\eqref{eq:GG}. Here we cannot rely exclusively on separation of supports because the $S(t-\tau)$ evolution of the term localized to $\calG_{-, j}$ has some of its support within $2 R(\eta_0)$  of the term localized to $\calG_{+, k}$ for all $j, k$. The saving grace is that the pieces of the supports of $S(t - \tau)  [1_{\calG_{-, j}}   1_{\calC_{\inte}} G(u,v)](\tau)$ and $[1_{\calG_{+, k}}1_{\calC_{\inte}} G(u,v)](t) $ that are close to each other (say within $2^{\al j} + 2^{\al k}$ for some small parameter $\alpha>0$) come along with \emph{angular separation} in the sense of Lemma~\ref{lem:ang}. To make this precise we must further subdivide $G_{\pm, k}$ as follows. 

Let $\al>0$ be a small parameter to be fixed below. Let 
\EQ{
&\calG_{+, k, in}:=  \calG_{+, k} \cap  \{ (t, x) \mid \abs{x} \le t - 2^{\al k}N^{\al(1-\eps)} \} \\
& \calG_{+, k, out} :=  \calG_{+, k} \cap  \{ (t, x) \mid \abs{x} \ge t - 2^{\al k}N^{\al(1-\eps)} \}
}
We decompose~\eqref{eq:GG} as follows, noting symmetry in $j, k$ means it suffices to consider only the sum for $j \ge k$.  We write \eqref{eq:GG} in the form
\begin{align} 
&\sum  \bigl\langle \widehat{P}_{N,  M}^2 \frac{1}{\abs{\na}}S(t - \tau)  [1_{\calG_{-, j, in}}   1_{\calC_{\inte}} G]  ,  \, [1_{\calG_{+, k, in}}1_{\calC_{\inte}} G] \bigr \rangle   \label{eq:Ginin}  
\\
& +\sum  \bigl\langle \widehat{P}_{N,  M}^2 \frac{1}{\abs{\na}}S(t - \tau)  [1_{\calG_{-, j, in}}   1_{\calC_{\inte}} G]  ,  \, [1_{\calG_{+, k, out}}1_{\calC_{\inte}} G] \bigr \rangle   \label{eq:Ginout}  \\ 
& + \sum \bigl\langle \widehat{P}_{N,  M}^2 \frac{1}{\abs{\na}}S(t - \tau)  [1_{\calG_{-, j, out}}   1_{\calC_{\inte}} G]  ,  \, [1_{\calG_{+, k, in}}1_{\calC_{\inte}} G]\bigr \rangle  \label{eq:Goutin} 
\\
& + \sum \bigl\langle \widehat{P}_{N,  M}^2 \frac{1}{\abs{\na}}S(t - \tau)  [1_{\calG_{-, j, out}}   1_{\calC_{\inte}} G]  ,  \, [1_{\calG_{+, k, out}}1_{\calC_{\inte}} G] \bigr \rangle \label{eq:Goutout} ,
\end{align} 
where the sums are over $j,k\geq 0$ with $j\geq k$, $G=G(u,v)$, and the pairings are evaluated at $\tau,t$.

The key point will be that on the outer regions $\calG_{-, j, out}$, $\calG_{+, k, out}$ we can recover the same angular separation used to treat the term $C_{\ext -\ext}$ and on the inner regions $\calG_{-, j, in}$ and $\calG_{+, k, in}$ we obtain sufficient separation in support between the two factors after evolution by $S(t-\tau)$  to get enough decay in $j, k$  after the application of $\hat P_{N, M}^2 \frac{1}{\abs{\na}}$.

\begin{lem}[Angular separation in $\calG_{\pm, j, out}$] \label{l:Gout} Let $\al>0$ and let $S_{N, \al}$ be the set 
\EQ{ \label{eq:SNal} 
\calS_{N, \al} := \biggl \{ x \in \R^3 \mid \frac{\abs{x_{2, 3}}}{\abs{x}} \lesssim \frac{1}{N^{\frac{(1-\al)(1-\eps)}{2} }} \biggr\}
}
Then, there exists $\al>0$ small enough and $N_0>0$ large enough so that for all $x \in \calG_{\pm, j, out}$ we have 
\EQ{ \label{eq:angGout}
x \in S_{N, \al} \mand  \frac{1}{N^{\frac{(1-\al)(1-\eps)}{2} }} \ll \frac{M}{N}
}
for all $ N \ge N_0$ and  $M \ge N^{\frac{s_p}{1-\nu}}$ and for all $j \ge 0$.

\end{lem} 

\begin{proof} 
It suffices to consider $x \in\calG_{+, j, out}$. The proof is nearly identical to the proof of Lemma~\ref{l:angle}, but here we have allowed the region $\calG_{+, j, out}$ to deviate farther from the boundary of the cone as $j$ (and hence $t$) gets larger. As in Lemma~\ref{lem:ang} we have 
\EQ{
\abs{\sin(\te_{x(2^j N^{1-\eps})})} \simeq \abs{\te_{x(2^jN^{1-\eps})}} \le A_1 N^{\frac{\eps}{2} - \frac{1}{2}}
}
independently of $j \ge 0$. To finish the proof it suffices to show that for any $x \in \calG_{+, j, out}$, the angle $\te_{(x, x(2^jN^{1-\eps}))}$ formed between the vectors $x$ and $x(2^jN^{1-\eps})$ satisfies 
\EQ{
\abs{\te_{(x, x(2^jN^{1-\eps}))}} \le A_2 \frac{1}{N^{\frac{(1-\al)(1-\eps)}{2} }}
}
for some other uniform constant $A_2 >0$, as then the sine of the total angle between $x$ and the $x_1$-axis, i.e., $\frac{\abs{x_{2, 3}}}{\abs{x}}$ would satisfy~\eqref{eq:SNal}. Note that for any $(t, x) \in \calG_{+, j, out}$ 
\EQ{
2^{j}N^{1-\eps} - 2^{\al j} N^{\al(1-\eps)} \le \abs{x} \le 2^{j+1} N^{1-\eps} + 2^{\al j} N^{\al(1-\eps)}  
}
Arguing as in the proof of Lemma~\ref{l:angle} we see that for any $(t, x) \in \calG_{+, j, out} $
\EQ{
\te_{(x, x(2^jN^{1-\eps}))}^2 \lesssim \frac{ 2 t 2^{\al j} N^{\al(1-\eps)}}{ (t - 2^{\al j} N^{\al(1-\eps)})( 2^{j}N^{1-\eps})} \lesssim \frac{1}{ 2^{(1-\al)j} N^{(1-\al){(1-\eps)}}},
}
as desired. \end{proof} 
  
  With Lemma~\ref{l:Gout} in hand, we can estimate the term~\eqref{eq:Goutout} in an identical fashion as the term~\eqref{c_outout}, noting that applications of Lemma~\ref{lem:ang} are still valid in this new setting because for $x \in G_{+, k, out}$ and $y \in G_{-, j, out}$ we have 
  \EQ{
 \frac{ \abs{x_{2, 3} - y_{2, 3}}}{ \abs{x-y}} \lesssim \frac{\abs{x_{2, 3}}}{\abs{x}} + \frac{\abs{y_{2, 3}}}{\abs{y}} \lesssim \frac{1}{N^{(1-\al)(1-\eps)}} \ll \frac{M}{N},
  }
  i.e., sufficient angular separation since the Fourier variable $\xi$ satisfies 
  \[
  \abs{\xi_{2, 3}}/ \abs{\xi} \simeq M/N.
  \]
   Moreover we have 
  \EQ{
  \abs{x-y} \simeq (2^{j} + 2^k) N^{1-\eps} \mif x \in G_{+, k, out}, \, y \in G_{-, j, out}
  }
  This means that we are free to write, 
  \EQ{
  & \bigl\langle \widehat{P}_{N,  M}^2 \frac{1}{\abs{\na}}S(t - \tau)  [1_{\calG_{-, j, out}}   1_{\calC_{\inte}} G(u,v)](\tau)  ,  \, [1_{\calG_{+, k, out}}1_{\calC_{\inte}} G(u,v)](t) \bigr \rangle \\
  & =  \bigl\langle [1_{S_{N, \al}} 1_{\{\abs{\cdot} \simeq (2^{j} + 2^k) N^{1-\eps} \}}  K_{N, M} ] \ast   [1_{\calG_{-, j, out}}   1_{\calC_{\inte}} G(u,v)](\tau)  , \\ & \quad\quad\quad \, [1_{\calG_{+, k, out}}1_{\calC_{\inte}} G(u,v)](t) \bigr \rangle.
  }
  Mimicking the estimates of~\eqref{eq:Goutout} we see that as in~\eqref{eq:MNL} we have 
  \EQ{
  \| 1_{S_{N, \al}}&( \cdot)  1_{\{\abs{\cdot} \simeq (2^{j} + 2^k) N^{1-\eps} \}}  (\cdot) K_{N, M}(t- \tau) \|_{L_x^{\left(\frac{qp}{p} \right)'/2}}  \\
  & \lesssim_L \frac{N^{L+1}}{M^{2L}} \frac{1}{[(2^j + 2^k)N^{1-\eps}]^L}
  }
  This allows us to sum in $j, k$, and we obtain, 
  \EQ{
  \int_{- \infty}^{- N^{1-\eps}}  \int_{ N^{1-\eps}}^{\infty} \eqref{eq:Goutout} \, \ud t \, \ud \tau \lesssim_L \frac{1}{M^L}.
  }
  To handle the term~\eqref{eq:Ginin} we rely on the following observation: by the support properties of $1_{\calC_{+, k}}1_{\calC_{\inte}}(\tau, y)$,  $1_{\calC_{-, j}}   1_{\calC_{\inte}}(t, y) $ and the sharp Huygens principle, we must  have 
\begin{equation} \label{eq:gjgkin}
\abs{x - y} \gtrsim (2^j +2^k) N^{1-\eps} 
\end{equation}
for all  
\[
 x \in \supp (1_{\calG_{+, k, in}}1_{\calC_{\inte}} G(u,v))(t)
 \]
  and 
  \[
  y \in \supp S(t- \tau) [1_{\calG_{-, j, in}}  1_{\calC_{\inte}} G(u,v)](\tau).
  \]
   Hence, 
\begin{align}
&\abs{\bigl\langle \widehat{P}_{N,  M}^2\frac{1}{\abs{\na}} S(t - \tau)  [1_{\calG_{-, j, in}}   1_{\calC_{\inte}} G(u,v)](\tau)  ,  \, [1_{\calG_{+, k, in}}1_{\calC_{\inte}} G(u,v)](t) \bigr \rangle}  \\
& \lesssim  \| 1_{ \{\abs{ \cdot} \gtrsim (2^j +2^k) N^{1-\eps} \} }p_{N, M} \|_{L_x^{\left(\frac{q_p}{p}\right)'/2}} N^{-1}  \\
&  \quad  \times \| P_NS(t - \tau)  [1_{\calC_{-, j}}   1_{\calC_{\inte}} G(u,v)](\tau) \|_{L^{\frac{q_p}{p}}_x}\|  [1_{\calC_{+, k}}1_{\calC_{\inte}} G(u,v)](t) \|_{L_x^{\frac{q_p}{p}}} \\
&  \lesssim_L \frac{1}{[(2^j+ 2^k)N^{1-\eps}]^L} 
   \left(\|u \|_{L^\infty_t \dot H^{s_p}}^{2p} + \|v \|_{L^\infty_t \dot H^{s_p}}^{2p} \right)
\end{align} 
Hence, 
\EQ{
  \int_{- \infty}^{- N^{1-\eps}}  \int_{ N^{1-\eps}}^{\infty} \eqref{eq:Ginin} \, \ud t \, \ud \tau \lesssim_{L} \frac{1}{N^{L}} 
}
Next, for the term~\eqref{eq:Ginout} we note that the same argument used to treat~\eqref{eq:Ginin} applies.  However, we note that here we only obtain  spatial separation of $2^j N^{1-\eps}$.  Nonetheless,  since $j \ge k$ we have that 
\EQ{
2^jN^{1-\eps}  \simeq (2^j + 2^k) N^{1-\eps}
}
and hence we are able to sum in $j, k$, obtaining
\EQ{
  \int_{- \infty}^{- N^{1-\eps}}  \int_{ N^{1-\eps}}^{\infty} \eqref{eq:Ginout} \, \ud t \, \ud \tau \lesssim_L \frac{1}{M^L}
}

Lastly, consider the term~\eqref{eq:Goutin}. Here we use a mix of the arguments used to control~\eqref{eq:Ginin} and~\eqref{eq:Goutout}. In particular we split the sum into two pieces noting that if $j \simeq k$ then the same argument used to estimate~\eqref{eq:Ginin} applies since the spatial supports are separated by $\simeq 2^{k} N^{1-\eps} \simeq (2^j + 2^k) N^{1-\eps}$. If $j \gg k$ we obtain enough angular separation argument to use the same argument used to bound~\eqref{eq:Goutout}, since in this case we have 
\EQ{
 \frac{ \abs{x_{2, 3} - y_{2, 3}}}{ \abs{x-y}} \simeq  \frac{\abs{y_{2, 3}}}{\abs{y}} \lesssim \frac{1}{N^{(1-\al)(1-\eps)}} \ll \frac{M}{N}
}
for all $x \in \calG_{+, k, in}$ and $y \in \calG_{-, j, out}$ as long as $j \gg k$. 
 We obtain
\EQ{
  \int_{- \infty}^{- N^{1-\eps}}  \int_{ N^{1-\eps}}^{\infty} \eqref{eq:Ginin} \, \ud t \, \ud \tau \lesssim_{L} \frac{1}{N^{L}} + \frac{1}{M^L} 
}

This completes the estimation of \eqref{eq:GG}. 

At this point, the mixed terms \eqref{eq:GC} and \eqref{eq:CG} (i.e. the remaining contributions to the $C_{\inte-\inte}$ term), as well as the $C_{\inte-\ext}$ and $C_{\ext-\inte}$ terms (\eqref{c_outin} and \eqref{c_inout}) can be handled with a combination of the techniques developed above. For example, after further subdividing $\calG_{-}$ in the regions $\calC_{-, j}$ and $\calG_{-, j}$ consider the term of the form, 
\EQ{
&\sum_{j \ge 0}  \int_{-2^{j+2} N^{1-\eps}}^{-2^j N^{1-\eps}} \int_{N^{1-\eps}}^\infty \\
& \quad \bigl\langle \widehat{P}_{N,  M}^2 \frac{1}{\abs{\na}}S(t - \tau)  [1_{\calG_{-, j}}   1_{\calC_{\inte}} G(u,v)](\tau)  ,  \, 1_{\calC_{\ext}}G(u,v)](t) \bigr \rangle\, \ud t \, \ud \tau.
}
Fixing a large constant $K_1>0$,  we can divide the above into two further pieces, namely
\EQ{
 & \sum_{j \ge 0}  \int_{-2^{j+2} N^{1-\eps}}^{-2^j N^{1-\eps}} \int_{N^{1-\eps}}^{K_1 2^{j} N^{1-\eps}}  \\
 &\quad \bigl\langle \widehat{P}_{N,  M}^2 \frac{1}{\abs{\na}}S(t - \tau)  [1_{\calG_{-, j}}   1_{\calC_{\inte}} G(u,v)](\tau)  ,  \, 1_{\calC_{\ext}}G(u,v)](t) \bigr \rangle\, \ud t \, \ud \tau  \\
& + \sum_{j \ge 0}  \int_{-2^{j+2} N^{1-\eps}}^{-2^j N^{1-\eps}} \int_{K_1 2^j N^{1-\eps}}^\infty \\
& \quad + \bigl\langle \widehat{P}_{N,  M}^2 \frac{1}{\abs{\na}}S(t - \tau)  [1_{\calG_{-, j}}   1_{\calC_{\inte}} G(u,v)](\tau)  ,  \, 1_{\calC_{\ext}}G(u,v)](t) \bigr \rangle\, \ud t \, \ud \tau.
}
For the first term on the right-hand-side above we can copy the argument used to estimate~\eqref{eq:Ginin}. Indeed by the sharp Huygens principle the spatial supports (before application of $P_{N, M}$)  are separated for each fixed $t, \tau$ by a distance of at least $\simeq 2^{j}N^{1-\eps} \simeq_{K_1} \abs{t -\tau}$. For the second term above we can choose $K_1 \gg 1$ large enough to guarantee enough angular separation between the spatial and Fourier variables to mimic a combination of the  arguments used to estimate~\eqref{c_outout} (where one integrates in $t$) and~\eqref{eq:Goutout} (where one sums in $j$). The remaining interactions are handled similarly. We omit the details.  

We have thus proved that 
\EQ{
\abs{\ang{C, C'} } \lesssim_L \frac{1}{N^L} + \frac{1}{M^L}  \lesssim_L \frac{1}{M^L},
}
which finally completes the proof of Lemma~\ref{l:CC'}. 
\end{proof}

\medskip
We are now prepared to conclude the frequency envelope argument and the proof of Proposition~\ref{padditional}.

\begin{proof}[Proof of Proposition~\ref{padditional}] 
Recall that we are trying to prove that
\begin{equation}
\sum_{N \geq N_0}\sum_{C_0 N^{\frac{s_p}{1-\nu}} \leq M \leq N} M^{2(1-\nu)}\|\widehat P_{N, \geq M}u(t)\|^2_{L_x^2} \lesssim 1,
\end{equation}
for some fixed $C_0 > 0$, for which it suffices to prove that
\begin{equation}
\sum_{N \geq N_0}\sum_{C_0 N^{\frac{s_p}{1-\nu}} \leq M \leq N} M^{2(1-\nu)}N^{-2 s_p} \|\widehat P_{N, \geq M}u(t)\|^2_{\dot H^{s_p}_x} \lesssim 1.
\end{equation}
Once again, by time-translation invariance, we argue for $t = 0$. Recall that 
\begin{align}
| \langle  \widehat{P}_{N, \geq M} u(0),\, \widehat{P}_{N,  \geq M} u(0) \rangle_{\dot H^{s_p}_x}|& \lesssim \|A\|_{\dot H^{s_p}}^2 + \|A'\|_{\dot H^{s_p}}^2 \\
& \quad +\|B\|_{\dot H^{s_p}}^2+\|B'\|_{\dot H^{s_p}}^2+|\langle C,C'\rangle_{\dot H^{s_p}_x}|,
\end{align}
and hence by Lemmas \ref{l:AA'}, \ref{l:BB'} and \ref{l:CC'}, we obtain
\begin{align}\label{equ:first_env_bd}
\gamma_{N,M}(0) &= \sum_{N',M' \geq M }\min\biggl\{\frac{N}{N'},\frac{N'}{N}\biggr\}^{\sigma}\left( \frac{M'}{M} \right)^{\sigma} \| \widehat{P}_{N, \geq M} u(0)\|_{\dot H^{s_p}_x}^2 \\
& \lesssim \eta_0^{p-1} \alpha_{N,M} + \eta_0^{p-1} \beta_{N,M} + M^{-L}.
\end{align}

Furthermore by \eqref{alpha_bds} and \eqref{equ:beta_bds}, 
\[
\alpha_{N,M} \lesssim \gamma_{N, M}(N^{1-\eps}) + \eta_0^{p-1} \alpha_{N,M},
\]
and
\[
\gamma_{N, M}(N^{1-\eps}) + \beta_{N,M} \lesssim \gamma_{N, M}(0) + \eta_0^{p-1} \beta_{N,M}.
\]
Hence
\[
\beta_{N,M} \lesssim \gamma_{N, M}(0) , \qquad \alpha_{N,M} \lesssim \gamma_{N, M}(N^{1-\eps}) \lesssim  \gamma_{N, M}(0),
\]
and we conclude from \eqref{equ:first_env_bd} that
\begin{align}\label{equ:second_env_bd}
\gamma_{N,M}(0) \lesssim \eta_0^{p-1} \gamma_{N, M}(0) + M^{-L},
\end{align}
which implies
\begin{align}\label{equ:final_env_bds}
\gamma_{N,M}(0) \lesssim M^{-L}
\end{align}
for any $L \gg 1$. Consequently, we have established that
\begin{align}
\sum_{N \geq N_0}\sum_{M \geq C_0 N^{\frac{s_p}{1-\nu}}} M^{2(1-\nu)}N^{-2 s_p} \gamma_{N,M}(0)^2 \lesssim 1,
\end{align}
which concludes the proof.

\end{proof}

\bibliographystyle{myamsplain}
\bibliography{researchbib}

\medskip

\centerline{\scshape Benjamin Dodson}
\smallskip
{\footnotesize
 \centerline{Department of Mathematics, Johns Hopkins University}
\centerline{404 Krieger Hall, Baltimore, MD 21218}
\centerline{\email{dodson@math.jhu.edu}}
} 

\medskip

\centerline{\scshape Andrew Lawrie}
\smallskip
{\footnotesize
 \centerline{Department of Mathematics, Massachusetts Institute of Technology}
\centerline{77 Massachusetts Ave, 2-267, Cambridge, MA 02139, U.S.A.}
\centerline{\email{ alawrie@mit.edu}}
} 
\medskip

\centerline{\scshape Dana Mendelson}
\smallskip
{\footnotesize
 \centerline{Department of Mathematics, University of Chicago}
\centerline{5734 S. University Avenue, Chicago, IL  60637}
\centerline{\email{dana@math.uchicago.edu}}
} 

\medskip

\centerline{\scshape Jason Murphy}
\smallskip
{\footnotesize
 \centerline{Department of Mathematics and Statistics, Missouri University of Science and Technology}
\centerline{400 West 12th St., Rolla, MO 65409}
\centerline{\email{jason.murphy@mst.edu}}
} 

\end{document}